\documentclass[12pt]{article}
\usepackage[latin1]{inputenc}
\usepackage[english]{babel}
\usepackage[T1]{fontenc}
\usepackage{lmodern}
\usepackage{graphicx, xcolor, mathrsfs, ulem}
\usepackage{amsfonts}
\usepackage{stmaryrd}
\usepackage{blkarray, bigstrut}
\usepackage{mathtools}
\usepackage[nice]{nicefrac}
\usepackage{amsmath,amsthm,amssymb}
\usepackage[mathlines]{lineno}
\usepackage{mathabx}
\usepackage{accents}

\usepackage{enumerate}
\usepackage[colorlinks=true,citecolor=black,linkcolor=black,urlcolor=blue]{hyperref}
\usepackage[top=.9in, bottom=.9in, left=.5in , right=.5in]{geometry}
\usepackage{caption}
\usepackage{xifthen} 
\usepackage{wrapfig}
\usepackage[numbers, square]{natbib}
\usepackage{comment}
\usepackage{float}
\usepackage{ifthen}
\mathtoolsset{showonlyrefs}
\usepackage{tikz}
\usepackage{caption}
\usetikzlibrary{arrows}
\usetikzlibrary{graphs,graphs.standard}
\usetikzlibrary{positioning,arrows.meta}
\usetikzlibrary{shapes.multipart}
\usetikzlibrary{arrows}
\usetikzlibrary{automata}
\definecolor{arrowblue}{RGB}{0,0,0}  
\newcommand\ImageNode[3][]{
  \node[draw=arrowblue!80!black,line width=1pt,#1] (#2) {\includegraphics[width=3.1cm,height=3.1cm]{#3}};
}

\usepackage{caption}
\newtheorem{thm}{Theorem}[section]
\newtheorem{rmq}{Remark}[section]
\newtheorem*{df}{Definition}
\newtheorem{lem}[thm]{Lemma}
\newtheorem{clm}{Claim}[thm]
\newtheorem{prop}[thm]{Proposition}
\newtheorem{cor}[thm]{Corollary}
\newtheorem{ex}[rmq]{Example}

\def\supp{{\rm Supp}}

\newcommand{\ho}[1]{\ensuremath{h_{0}\left(#1\right)}}
\newcommand{\hoo}[1]{\ensuremath{h_{1}\left(#1\right)}}
\newcommand{\wh}[1]{\ensuremath{
\widehat{#1}}}

\newcommand{\R}{\mathbb R}
\newcommand{\N}{\mathbb N}
\newcommand{\E}[2][]{\ensuremath{\mathbb{E}_{#1}\left[#2 \right]}}

\newcommand{\Ex}[1]{\ensuremath{\mathbb{E} \left[#1 \right]}}
\newcommand{\Prob}[2][]{\ensuremath{\mathbb{P}_{#1} \left(#2 \right)}}

\newcommand{\Pdx}[1]{\ensuremath{\mathbb{P} \left(#1 \right)}}
\newcommand{\I}[1]{\ensuremath{\mathbf{1}_ {\{ #1 \} }}}

\newcommand{\eps}{\varepsilon}

\newcommand{\dg}[2]{d_{#2}(#1)}

\newcommand{\gmax}[1]{g_{\max{}}\left(#1\right)}
\newcommand{\gmin}[1]{g_{\min{}}\left(#1\right)}
\newcommand{\hmax}[1]{h_{\max{}}\left(#1\right)}
\newcommand{\hmin}[1]{h_{\min{}}\left(#1\right)}

\newcommand{\cE}{\mathcal {E}}

\newcommand{\cZ}{\mathcal{Z}}
\newcommand{\cT}{\mathcal{T}}
\newcommand{\Xitwo}{\Xi^{(2)}}
\newcommand{\urndeg}{\mathscr{D}}

\newcommand{\ind}[1]{\mathbf 1_{#1}}

\newcommand{\eig}[1]{\lambda_1^{(#1)}}
\newcommand{\eign}[1]{\tilde{\lambda}_1^{(#1)}}
\newcommand{\Lm}{\Pi^{(\mu)}}
\newcommand{\lamu}{\lambda^{(\mu)}}

\newcommand{\dd}{\mathrm d}

\newcommand{\p}{p}
\newcommand{\wmax}{w^*}
\newcommand{\T}{\mathscr{T}}

\newcommand{\Ve}{\mathcal{E}}
\newcommand{\U}{\mathcal{U}}
\newcommand{\A}{\mathcal{A}}

\DeclareMathOperator{\esssup}{ess\, sup}

\newcommand{\bfgamm}{\text{\boldmath$\gamma$}}
\newcommand{\dimurn}{D_m}
\newcommand{\truncurn}{K'}
\newcommand{\maxurn}{J'}
\newcommand{\maxurntwo}{J}
\newcommand{\actvectwo}{\mathbf{a}'}
\newcommand{\bfgammtwo}{\text{\boldmath$\gamma'$}}

\DeclareMathOperator{\exponentialrv}{Exp}
\newcommand{\Exp}[1]{\exponentialrv\left( #1 \right)}

\newboolean{parentmodel}
\setboolean{parentmodel}{false} 

\def\ti#1{\textcolor{red}{[TI Comments: #1\textbf{}]}}

\title{Condensation phenomena in preferential attachment trees with neighbourhood influence}
\author{N.\ Fountoulakis\footnote{School of Mathematics, University of Birmingham, Birmingham, UK.} \footnote{Research supported by the EPSRC grant 
EP/P026729/1.}
, \, T.\ Iyer\footnotemark[1]}
\begin{document}
\maketitle
\abstract{We introduce a model of evolving preferential attachment trees where vertices are assigned weights, and the evolution of a vertex depends not only on its own weight, but also on the weights of its neighbours. We study the distribution of edges with endpoints having  certain weights, and the distribution of degrees of vertices having a given weight. We show that the former exhibits a condensation phenomenon under a certain critical condition, whereas the latter converges almost surely to a distribution that resembles a power law distribution. Moreover, in the absence of condensation, we prove almost-sure setwise convergence of the related quantities. This generalises  existing results on the Bianconi-Barab\'{a}si tree as well as on an evolving tree model introduced by the second author.
}

\ifthenelse{\boolean{parentmodel}}{
\section{Introduction}
\normalem
In this paper, we consider a class of growing weighted recursive trees where new vertices attach to existing vertices with probability depending on the weights of existing vertices, and the weights of their neighbours. More specifically, we consider evolving sequences of \textit{weighted trees} $T_{i \geq 0}$, which are labelled trees with vertices carrying real valued weights. Vertices arrive one at a time, connect to an existing vertex in the tree, and are then assigned a weight. We assume that the vertices of the trees are labelled with natural numbers which index their order of arrival. 
\\\\
Let $\mathcal{T}$ denote the set of all such weighted trees with, and given a tree $T \in \mathcal{T}$ with vertex $v$ let $N(v)$ denote the neighbourhood of the vertex $v$ (also incorporating information about the weights of these neighbours, and v). In order to describe the trees, we require a probability distribution $\mu$, which, without loss of generality is supported on $[0,1]$ (this can be generalised to any bounded subset of $\mathbb{R}$), and a \textit{fitness function} $f: \mathcal{T} \rightarrow \mathbb{R}$. In what follows, we will generally use $W$ to refer to a random variable with law $\mu$, and given a vertex $i$ we will use the symbols $w_i, W_i$ to refer to the weight of this vertex. 
\\\\
We start with an initial tree $T_{-1}$ consisting of a single vertex $-1$ with weight $w_{-1} \in \supp{(\mu)}$. Then, given $T_i$, the model proceeds recursively as follows:
 \begin{enumerate}[(i)]
 \item Sample a vertex $j$ from $T_i$ with probability $\frac{f(N(j))}{Z_i},$
  where $Z_i : = \sum_{k=0}^{i}f(N(j))$ is the \textit{partition function} associated with the process. \item Form $T_{i+1}$ by adding the edge $(j, i+1)$, and assigning $i+1$ weight $W_{i+1}$ sampled independently from $\mu$.
 \end{enumerate}

In this paper, we will focus of two general classes of fitness functions: 
\begin{enumerate}
\item[1.] (Multiplicative model) $f(N(v)) := h_0(W_v) \cdot \sum_{u \sim v} h_1(W_u)$ where $h_0, h_1: [0,1] \to \mathbb{R}_+$ are two positive real-valued functions; 
\item[2.] (Additive model) $f(N(v)) : = a W_v + b\sum_{u \sim v} W_u + c$, for some $a,b,c \geq 0$. 
\end{enumerate}

The key to our analysis is the almost sure convergence of $\frac{1}{n} Z_n$ to a positive constant. To show this we will couple the above process with a P\'olya urn scheme that has
finitely many colours.   

\section{Almost sure convergence of the partition function} 

\subsection{A Coupling Argument}

In this subsection, we    couple the tree process with a vector valued \textit{urn process}. For $L \in \mathbb{N}$, the urn process $(\mathcal{U}^{(n)})_{n \in \mathbb{N}}$ is an evolving $L(L+2)$ dimensional vector, with entries \[\left\{\mathcal{U}^{(n)}_{(i,j)}, : \; (i,j) \in [L+1] \times [L] \cup [L] \times [L+1] \right\}.\] Associated with the urn process is a vector of positive \textit{activities} \[\mathbf{a} = \left\{a_{(i,j)} : \; (i,j) \in [L+1] \times [L] \cup [L] \times [L+1] \right\}.\] 
We partition the interval $[0, 1]$ into equal intervals of length $h = 1/L$. For $j = 1, \ldots, L$, let $I_j$ denote the $j$th interval, that is, $I_{j} := ((j-1)\cdot h,j\cdot h]$. Moreover, let $m_j := \Prob{I_j}$, and for any vertex $v$ set $r(v) = j$ if the weight $w_v \in I_j$. Finally, we set $\omega_i := i \cdot h$.
 Then, the activities are defined as follows:
\begin{align} \label{eq:act-def}
    a_{(i,j)} = \begin{cases}
    \ho{\omega_i} \hoo{\omega_j}, & \text{if } i,j < L+1; \\ 
    \ho{\omega_1} \hoo{\omega_{j}}, & \text{if } i = L+1, j < L+1; \\
    \ho{\omega_j}\hoo{\omega_{L}} + \hoo{\omega_j}(\ho{\omega_{L}} - \ho{\omega_{1}}), & \text{if $i< L+1, j=L+1$.}
    \end{cases}
\end{align}
We also define $\gamma_{i,j} := \frac{\ho{\omega_{i-1}}\hoo{\omega_{j-1}}}{\ho{\omega_{i}} \hoo{\omega_{j}}}$, for $i, j = 1, \ldots,L$, $\gamma_{L+1,j} := \frac{\ho{\omega_{0}}\hoo{\omega_{j-1}}}{\ho{\omega_{1}} \hoo{\omega_{j}}}$ for $j = 1, \ldots, L$
and  \[Z^{(n)}_{P} := \sum_{i,j} a_{(i,j)} \mathcal{U}^{(n)}_{(i,j)}.\] 

\subsubsection{Coupling}
First sample the entire tree, $(T_{n})_{n \in \mathbb{N}_{0}}$; we will use this to define the evolution of the urn process. Let $\eta^{(n)}_{(j)}$ denote the sum of fitnesses of edges $(v,v') \in T_{n}$ such that $r(v') = j$. Also, define \[L^{(n)}_{(j)} : = \sum_{i=1}^{L+1}\gamma_{(i,j)} a_{(i,j)} \mathcal{U}^{(n)}_{(i,j)}.\] Finally, let $Z^{(n)}_{T}$ denote the partition function of the tree. Assume at time $1$,  the tree consists of a pair of vertices $-1, 0$ such that $r(-1) = \ell$ and $r(0) = \ell'$. Then, initialise the urn with balls $(\ell,\ell')$ and $(\ell', \ell)$. Note that, by definition, \[Z^{T}_0 = \hoo{w_{-1}}\ho{w_0} + \hoo{w_{0}}\ho{ w_{-1}} \leq \hoo{\omega_{\ell}}\ho{\omega_{\ell'}} + \hoo{\omega_{\ell'}}\ho{\omega_{\ell}} = Z^{P}_0.\] Moreover, by the definition of $\gamma_{(\ell',\ell)}$ we have \[\eta^{(0)}_{(\ell)} = \hoo{w_{-1}}\ho{w_0} \geq \gamma_{(\ell',\ell)} \hoo{\omega_{\ell}}\ho{\omega_{\ell'}} = L^{(0)}_{( \ell)},\] and similarly $\eta^{(0)}_{(\ell')} \geq L^{(0)}_{(\ell')}$.
\\
Assume inductively, that after $n$ transitions in the urn, we have $\eta^{(n)}_{(j)} \geq L^{(n)}_{(j)}$ for each $j \leq L$ and moreover $Z^{(n)}_{T} \leq Z^{(n)}_{P}$. Let $s$ be the vertex sampled in the tree in the $(n+1)$st step, and assume that $r(s)=j$, and $r(n+1) = k$. Then, for the $(n+1)$th step in the urn Sample a uniform random variable $U_{n+1}\in [0,1]$:
\begin{enumerate}
    \item  If $U_{n+1} \leq \frac{ L^{(n)}_{(j)}Z^{(n)}_{T}}{\eta^{(n)}_{(j)} Z^{(n)}_{P}}$, add balls of type $(j,k)$ and $(k,j)$ to the urn. 
    \item Otherwise, add balls of type $(L+1,k), (k,L+1)$.
\end{enumerate}
 
 Note that, after this step, in Case 1, we have \[\eta^{(n+1)}_{(j)} = \eta^{(n)}_{(j)} + \hoo{w_{s}}\ho{w_{n+1}} \geq L^{(n)}_{(j)} + \gamma_{(k,j)} \hoo{\omega_{j}}\ho{\omega_{k}} = L^{(n+1)}_{(j)},\] \[\eta^{(n+1)}_{(k)} = \eta^{(n)}_{(k)} + \hoo{w_{n+1}}\ho{w_{s}} \geq L^{(n)}_{(k)} + \gamma_{(j,k)} \hoo{\omega_{k}}\ho{\omega_{j}} = L^{(n+1)}_{(k)}\]
 and 
 \begin{align*}
 & Z^{(n+1)}_{T} = Z^{(n)}_{T} + \hoo{w_{s}}\ho{w_{n+1}} + \hoo{w_{n+1}}\ho{w_{s}} \\ & \hspace{3cm} \leq Z^{(n)}_{P} +  \hoo{\omega_{j}}\ho{\omega_{k}} + \hoo{\omega_{k}}\ho{\omega_{j}} = Z^{(n+1)}_{P}.
 \end{align*}
 Moreover, in 
 Case 2, we have \[\eta^{(n+1)}_{(j)} \geq \eta^{(n)}_{(j)} \geq  L^{(n+1)}_{(j)} = L^{(n)}_{(j)},\] \[\eta^{(n+1)}_{(k)} = \eta^{(n)}_{(k)} + \hoo{w_{n+1}}\ho{w_{s}} \geq L^{(n)}_{(k)} + \gamma_{(L+1,j)} \hoo{\omega_{k}}\ho{\omega_{1}} = L^{(n+1)}_{(k)},\] and moreover 
 
 \begin{align*}
 &Z^{(n+1)}_{T} = Z^{(n)}_{T} + \hoo{w_{s}}\ho{w_{n+1}} + \hoo{w_{n+1}}\ho{w_{s}} 
 \\
 &\hspace{3cm} \leq Z^{(n)}_{P} + \hoo{\omega_{L}}\ho{\omega_{k}}
+ h^{+}(\omega_{k})h^{-}(\omega_{L}) = Z^{(n+1)}_{P}.
 \end{align*} Thus, we may recurse the coupling. 

\begin{lem}
The urn process $\left(\mathcal{U}^{(n)}\right)_{n \in \mathbb{N}_{0}}$ is distributed like a P\'olya urn process with activities defined as in equation \eqref{eq:act-def}, and replacement matrix $M$ such that
\begin{align}
    M_{(i,j),(i',j')} = \begin{cases}
    \gamma_{(i',j')} a_{(i',j')} m_i, & \text{if } j = j';\\
    \gamma_{(i',j')} a_{(i',j')} m_j, & \text{if } i = j'; \\
    \left(1 - \gamma_{(i',j')}\right)m_{j}, & \text{if } i = L+1; \\
    \left(1 - \gamma_{(i',j')}\right)m_{i},
    & \text{if } j = L+1; \\
    0 & \text{otherwise.}
    \end{cases}
\end{align}
\end{lem}

\begin{proof}
Given $\mathcal{U}^{(n)}$, the probability of adding balls of type $(j,k)$ and $(k,j)$ for $j, k \leq L$ is 
\[
m_k \frac{\eta^{(n)}_{(j)}}{Z^{(n)}_{T}} \times \frac{ L^{(n)}_{(j)}Z^{(n)}_{T}}{\eta^{(n)}_{(j)} Z^{(n)}_{P}}
 +m_j \frac{\eta^{(n)}_{(k)}}{Z^{(n)}_{T}} \times \frac{ L^{(n)}_{(k)}Z^{(n)}_{T}}{\eta^{(n)}_{(k)} Z^{(n)}_{P}}
 = m_{k}\frac{L^{(n)}_{(j)}}{Z^{(n)}_{P}} + m_j 
 \frac{L^{(n)}_{(k)}}{Z^{(n)}_P},
\]
which is same transition probability as in the P\'olya urn scheme. Moreover, the probability of adding balls of type $(L+1, k), (k,L+1)$ is 
\begin{align*}
&m_k\sum_{j=1}^{L} \left(1 - \frac{L^{(n)}_{(j)}Z^{(n)}_{T}}{\eta^{(n)}_{(j)} Z^{(n)}_{P}} \right) \frac{\eta^{(n)}_{(j)}}{Z^{(n)}_{T}}
= m_{k} \left(1 - \sum_{j=1}^{L} \frac{L^{(n)}_{(j)}}{Z^{(n)}_{P}}\right), 
\end{align*}
as required. 
\end{proof}

\subsection{P\'olya urn approximation} 
In this section, we will describe a finite approximation of the process $(\T_n)_{n\geq 0}$ which is a generalised 
P\'olya urn scheme with balls of finitely many \textit{types}. Each \textit{type} is associated with an \textit{activity}, so that a ball of a certain type is chosen with probability proportional to its activity.
Fix $L > 0$, and partition the interval $[0, 1]$ into equal intervals of length $h = 1/L$. For $j = 1, \ldots, L$, let $I_j$ denote the $j$th interval, that is, $I_{j} := ((j-1)\cdot h,j\cdot h]$. Moreover, let $m_j := \Prob{I_j}$, and for any vertex $v$ set $r(v) = j$ if the weight $w_v \in I_j$. Finally, we set $\omega_i := i \cdot h$. 

Let $[L] := \{1,\ldots, L\}$. In the multiplicative model, we consider a P\'olya urn where balls have types corresponding to ordered pairs $(i,j)$ for $i,j \in [L]$. Moreover, we introduce balls of type $(L+1, i)$ for $i \in [L]$, so that in total, there are $L(L+1)$ types. The activity of the ball $(i,j)$, denoted $a_{ij}$ is defined as follows:
\begin{align} \label{eq:act-mult-def}
    a_{i,j} = \begin{cases} \ho{\omega_i} \cdot \hoo{\omega_j}, & \text{if } i < L + 1 \\
    \ho{1} \hoo{\omega_j}, & \text{otherwise.} 
    \end{cases}
\end{align}
The dynamics for the associated P\'olya urn are as follows: at the $n$th step, a ball $(i,j)$ is selected with probability proportional to $a_{i,j}$. It is then returned to the urn, and with probability $p_{i,j} := \frac{a_{i-1,j-1}}{a_{i,j}}$ two balls of type $(\cdot, j)$ and $(j, \cdot)$ are added to the urn; such that balls of type $(k,j), (j,k)$ are added with probability $m_k$. 

The replacement rule for the P\'olya urn is as follows: 

We partition $[0,1]$ into equal intervals of length $h = 1/L$, where $L \in \mathbb{N}$. Eventually, we will let 
$h\to 0$ by letting $L\to \infty$. In particular, let $I_j = ((j-1)\cdot h,j\cdot h]$, for $j=1,\ldots, L$. 
Let $m_j = \int_{I_j} \dd \mu (x)$; note that this the probability that a given vertex with will have fitness in 
$I_j$.  Furthermore, for any vertex $v$, we let $r(v) = i$, if $W_v \in I_i$. Also, we set $f_i = i \cdot h$.

Let $[L]= \{1,\ldots, L\}$. 
For the multiplicative model, the P\'olya urn will consist of balls which are the ordered pairs 
$(i,j)$, for $i,j = 1,\ldots, L$. Moreover, we introduce balls of colour $(L+1, i)$, for $i \in [L]$.
The purpose of these balls is for the sub-urn restricted to colours $(i,j)$, with $i,j \in [L]$ to provide 
a stochastic lower bound on $(\T_n)_{n\geq 0}$. 

Hence, overall the scheme uses $(L+1)L$ colours.  The activity of ball $(i,j)$ will be denoted by $a_{i,j}$. 
Its value in the two different models is set to be as follows. 
For the multiplicative model, we set  
\begin{equation} \label{eq:activity_mult_def} 
a_{i,j} =\begin{cases}  f_i \cdot f_j, & \mbox{if $i< L+1$} \\
f_j, & \mbox{if $i=L+1$}
\end{cases}. 
\end{equation}
The replacement matrix is described by the following process. 
During step $n$, a ball $(i,j)$ is selected with probability proportional to $a_{i,j}$. Thereafter, 
it is put back into the urn and 
with probability $p_{i,j}=\frac{f_{i-1}}{f_i} \cdot \frac{f_{j-1}}{f_j}$ 
two balls of colour $(k,j)$ and $(j,k)$, respectively, are added with probability $m_k$;
with probability $1 - \frac{f_{i-1}}{f_i} \cdot \frac{f_{j-1}}{f_j} =1-p_{i,j}$ we add two balls of colour $(L+1,k)$, 
with probability $m_k$.  
The definition of $p_{i,j}$ and that of $a_{i,j}$ in~\eqref{eq:activity_mult_def} for $i \leq L+1$ and $j\leq L$ 
imply that 
\begin{equation} \label{eq:activities_rec_id}
a_{i-1,j-1} = p_{i,j} a_{i,j}. 
\end{equation}

For the additive model, the P\'olya urn will also consist of balls which are ordered pairs $(i,j)$, 
for $i,j = 1,\ldots, L$ with activity
\begin{equation} \label{eq:activity_add_nondiag_def} 
 a_{i,j} = b \cdot f_j,
 \end{equation}
 and balls which we denote by $\widehat{(i,i)}$, for $i \in [L]$, that have 
 activity 
\begin{equation} \label{eq:activity_add_diag_def}
a_i = a \cdot f_i +c.
\end{equation}
The replacement matrix is described by the following process. 
During step $n$, ball $(i,j)$ is selected with probability proportional to $a_{i,j}$; thereafter, it is put 
back into the urn together with three balls of colour $(k,j)$, $(j,k)$ and $\widehat{(k,k)}$, respectively, with probability $m_k$. 
A ball $\widehat{(i,i)}$ is selected with probability proportional to $a_i$ and it is put back to the urn 
together with three balls of colour $(k,j)$, $(j,k)$ and $\widehat{(k,k)}$, respectively, with probability $m_k$. 

In both models, a ball of colour $(i,j)$ represents a \textit{directed} edge form a vertex with fitness in $I_i$ to another vertex with fitness in $I_j$. 
In the additive model, a ball of colour $\widehat{(i,i)}$ represents a \textit{directed loop} on a vertex whose 
fitness is in the interval $I_i$. 
Although, the trees we construct are undirected, the direction here represents 
the vertex the new vertex is going to attach to.  In other words, every edge of the tree is \textit{expanded} 
into two directed edges $(i,j), (j,i)$. Moreover, in the additive model, every vertex has also a directed loop 
attached to it, which is represented by a ball of colour $(i,i)$, for every $i=1,\ldots, L$, with activity 
equal to $a_{i,i}$. 
Note that in the additive model: for every vertex $v\in V(\T_n)$
\begin{equation} \label{eq:vertex_weights_add}
a_{r(v)} + \sum_{u: u\sim v} a_{r(v),r(u)} = f(N_n(v)).
\end{equation}
Similarly, in the multiplicative model: 
\begin{equation} \label{eq:vertex_weights_mult}
\sum_{u: u \sim v} a_{r(v),r(u)} = f(N_n(v)). 
\end{equation}

\subsection{The P\'olya scheme for the multiplicative model: eigenvectors and eigenvalues} 
\subsubsection{General Case}
In this subsection, we will consider the replacement matrix of the P\'olya scheme that corresponds to the 
multiplicative model and, in particular, we will determine its leading eigenvalue and the corresponding 
eigenvector. Let $M$ denote the replacement of the P\'olya scheme we use for the multiplicative 
model, so that $M_{(k,\ell),(i,j)} = a_{i,j} \Ex{\xi_{(i,j),(k,\ell)}}$, where 
$\xi_{(i,j),(k,\ell)}$ is the random variable which is the number of balls of type $(k,\ell)$ that are produced when a ball of type $(i,j)$ is selected. Note that $M$ is irreducible, since the associated directed graph is strongly connected: indeed any ball of type $(k,\ell)$ can be produced in two steps from a ball of type $(i,j)$.
Thus, by the Perron-Frobenius theorem, the largest eigenvalue in absolute value is positive and real, and there exists a corresponding 
right eigenvector with positive entries. 
Let $\lambda_{1,L}$ denote the leading eigenvalue and $\U$ denote a right eigenvector corresponding 
to this eigenvalue: $M \cdot \U = \lambda_{1,L} \U$. Moreover, we assume that $\U$ is normalised so that 
\begin{align} \label{eq:normalise}
\sum_{i=1}^{L}\ho{\omega_{i-1}} \left( \sum_{k=1}^{L} a_{(k-1,i-1)}\U_{(k,i)} + a_{(L+1,i-1)} \U_{(L+1,i)}\right) = 1.
\end{align}
Now, note that for $(i,j)$ with $i< L+1$ we have 
$M_{(i,j),(k,\ell)} >0$ if and only if $\ell=j$ or $\ell=i$; otherwise, it is equal to 0. 
In particular, $M_{(i,j), (k,j)} = \gamma_{(k,j)} a_{(k,j)} m_i$ and $M_{(i,j),(k,i)}=\gamma_{(k,i)} a_{(k,i)} m_j$. 
Also, for any $j=1,\ldots, L$, we have $M_{(L+1,j),(k,\ell)}= (1-\gamma_{(k,\ell)}) a_{(k,\ell)}m_j$. 
For $i,j\leq L$, we have 
$\gamma_{(i,j)} a_{(i,j)} = a_{(i-1,j-1)}$, whereas
$\gamma_{(L+1,j)} a_{(L+1,j)} = a_{(L+1,j-1)}$, for $j\leq L$.

With $\mathcal{A}_j : = \sum_{k=1}^{L} a_{(k-1,j-1)} \U_{(k,j)}+a_{(L+1,j-1)}\U_{(L+1,j)}$ we have 
\begin{align} \label{eq:non-dummy}
    \lambda_{1,L} \U_{(i,j)} = m_i \mathcal{A}_j + m_j\mathcal{A}_i.
\end{align}
Moreover, for the colours $(L+1,j)$ and $(j,L+1)$ with $j \in [L]$, we have
\begin{equation*}
\lambda_{1,L} \U_{(j,L+1)} =\lambda_{1,L} \U_{(L+1,j)} = m_j \sum_{k,\ell = 1}^{L} \left(a_{(k,\ell)} - a_{(k-1,\ell-1)}\right) \U_{(k, \ell)} 
+m_j \sum_{k=1}^L a_{(k,L+1)} \U_{(k,L+1)}. 
\end{equation*}
With $\mathcal{E} := \sum_{k,\ell = 1}^{L} \left(a_{(k,\ell)} - a_{(k-1,\ell-1)}\right) \U_{(k, \ell)}$, the above equation simplifies to 
\begin{align} \label{eq:dummy}
  \lambda_{1,L} \U_{(j,L+1)} =  \lambda_{1,L} \U_{(L+1,j)} = m_j \mathcal{E} + m_j \sum_{k=1}^L a_{(k,L+1)} \U_{(k,L+1)}.
\end{align}
Now, multiplying both sides of~\eqref{eq:non-dummy} by $a_{(i-1,j-1)}$ and summing over $i$, we get 
\begin{eqnarray*}  
\lambda_{1,L} \sum_{i=1}^{L} a_{(i-1,j-1)} \U_{(i,j)}  &=& \A_j \sum_{i=1}^{L} m_i a_{(i-1,j-1)}  + 
m_j \sum_{i=1}^{L} a_{(i-1,j-1)} \A_i \nonumber \\
&=& \A_j \hoo{\omega_{j-1}} \sum_{i=1}^L m_i \ho{\omega_{i-1}} + m_j \hoo{\omega_{j-1}} \sum_{i=1}^{L} \ho{\omega_{i-1}}\A_i.  
\end{eqnarray*}
Multiplying~\eqref{eq:dummy} by 
$a_{(L+1,j-1)}$ and adding it to the above
we get 
\begin{align}
    \lambda_{1,L} \A_j &= \A_j \hoo{\omega_{j-1}} \sum_{i=1}^{L} m_i \ho{\omega_{i-1}} + m_j \hoo{\omega_{j-1}} \sum_{i=1}^{L} \ho{\omega_{i-1}} \A_i \nonumber \\
    &+ m_j \ho{\omega_1}\hoo{\omega_{j-1}} 
    \mathcal{E}
    + m_j \ho{\omega_1}\hoo{\omega_{j-1}} \sum_{k=1}^L a_{(k,L+1)} \U_{(k,L+1)}. 
\end{align}
Now, note that by~\eqref{eq:normalise} we have $\sum_{i=1}^{L} \ho{\omega_{i-1}} \A_i = 1$. 
Set $\mu^{(L)}(h_0):=\sum_{i=1}^{L} m_i 
\ho{\omega_{i-1}}$ and $\mathcal{D}:= \sum_{k=1}^L a_{(k,L+1)} \U_{(k,L+1)}$.
By re-arranging the above display, we get
\begin{align}
\A_j = \frac{m_j \hoo{\omega_{j-1}}}{\lambda_{1,L} - \hoo{\omega_{j-1}}\mu^{(L)}(h_0)} \left(1 + \ho{\omega_1}\left(\mathcal{E} + \mathcal{D} \right)\right).
\end{align}
Now, multiplying both sides by $\ho{\omega_{j-1}}$ and summing over $j \in [L]$, the left hand side is $1$ by~\eqref{eq:normalise} and
\[
1 = \sum_{j=1}^{L} \frac{m_j \ho{\omega_{j-1}} \hoo{\omega_{j-1}}}{\lambda_{1,L} - \hoo{\omega_{j-1}} \mu^{(L)}(h_0)} \left(1 + \ho{\omega_1}\left(\mathcal{E} + \mathcal{D} \right)\right).
\]
Now, let $\lambda_1$ be the root of the equation 
\[
\frac{\ho{W}\hoo{W}}{\lambda_1 - \hoo{W} \E{\ho{W}}}.
\]
We now show that $\lambda_{1,L} \to \lambda_1$ as $L \to \infty$.
\begin{clm} \label{clm:lambda_conv} If $\lambda_1\geq \hoo{1} \E{\ho{W}}$, then  
$$ \lim_{L \to \infty} \lambda_{1,L} = \lambda_1. $$ 
\end{clm}
\begin{proof} 
Recall first that $\lambda_{1,L}$ is the root of~\eqref{eq:S}. 
Furthermore, by~\eqref{eq:eigenvector}, since $\A_i>0$, for all $i$, it must be the 
case that $\lambda_{1,L}/ \mu^{(L)} (h_0) > \hoo{1-h}$. 
But note that $\lim_{L\to \infty} \mu^{(L)} (h_0) = \E{\ho{W}}$.  
Hence, 
\begin{equation} \label{eq:lim_bound_lambda}
\liminf_{L \to \infty} \lambda_{1,L} \geq \hoo{1-h}\E{\ho{W}}.
\end{equation}

We need to show that 
\begin{equation} \label{eq:error_term}
\lim_{L\to \infty} \ho{\omega_1}\left(\mathcal{E} + \sum_{k=1}^L a_{(k,L+1)} \U_{(k,L+1)}\right) = 0.
\end{equation}
Note that 
\begin{equation}\label{eq:1st_ub}
 \sum_{k, \ell=1}^{L} (k+\ell - 1)\U_{(k,\ell)} \leq 2 L \sum_{k,\ell=1}^L \U_{(k,\ell)}. 
\end{equation}

Now, we will bound $\sum_{k,\ell=1}^L \U_{(k,\ell)}$ and, in particular, since $h=1/L$, it suffices to show that 
$h \sum_{k,\ell=1}^L \U_{(k,\ell)} \to 0$ as $L \to \infty$.
Summing~\eqref{eq:non-dummy} over $i,j=1,\ldots L$, we deduce that  
\begin{equation} \label{eq:sum_u_rewrite}  \lambda_{1, L} \sum_{k,\ell=1}^L \U_{(k,\ell)}  = 2 \sum_{\ell =1}^{L} \A_\ell. 
\end{equation}
\ti{Shouldn't there be a multiple of $\lambda_{1,L}$ on the left hand side?}
Now, since $f_{\ell -1} > L^{-1/2}$ for $\ell \geq L^{1/2}$, we can write 
\begin{eqnarray} 
\sum_{\ell =1}^{L} \A_\ell &=& L^{1/2} \sum_{\ell \geq L^{1/2}} \frac{1}{L^{1/2}} \A_\ell + 
\sum_{\ell < L^{1/2}} \A_\ell  \leq  L^{1/2} \sum_{\ell \geq L^{1/2}} f_{\ell-1} \A_\ell + 
\sum_{\ell < L^{1/2}} \A_\ell \nonumber \\ 
&\leq&  L^{1/2} \sum_{\ell =1}^L f_{\ell-1} \A_\ell + 
\sum_{\ell < L^{1/2}} \A_\ell  =  L^{1/2}+ 
\sum_{\ell < L^{1/2}} \A_\ell . \label{eq:S_bound}
\end{eqnarray}
Now, by~\eqref{eq:A_j}, we have 
$$ \sum_{\ell < L^{1/2}} \A_\ell  < \sum_{\ell < L^{1/2}} \frac{m_\ell f_{\ell-1}}{\lambda_{1,L} - f_{\ell-1}\| W \|_{1,L}} 
(1+4 h \sum_{\ell=1}^L \A_\ell). $$
But by~\eqref{eq:lim_bound_lambda}, and since $f_{\ell-1} < L^{-1/2}$, for $\ell < L^{1/2}$,  
it follows that for $L$ sufficiently large 
$$\ \frac{1}{\lambda_{1,L} - f_{\ell-1}\| W \|_{1,L}} \leq 2.$$
Hence, 
$$ \sum_{\ell < L^{1/2}} \A_\ell  < 2\sum_{\ell < L^{1/2}} m_\ell f_{\ell-1}
\left(1+4 h \sum_{\ell=1}^L \A_\ell \right). $$
Substituting this into~\eqref{eq:S_bound} we get 
\begin{eqnarray*}
\sum_{\ell =1}^{L} \A_\ell \left( 1- 8h \sum_{\ell < L^{1/2}} m_\ell f_{\ell-1} \right) \leq 
L^{1/2}+ 2\sum_{\ell < L^{1/2}} m_\ell f_{\ell-1}\leq L^{1/2}+ 2 \| W\|_{1,L}.
\end{eqnarray*} 
Since $h=1/L$, we thus conclude that 
$$h\sum_{\ell =1}^{L} \A_\ell = O(L^{-1/2}).  $$
Using the identity~\eqref{eq:sum_u_rewrite} and the upper bound~\eqref{eq:1st_ub}, we deduce 
\eqref{eq:error_term}.
Thus, \eqref{eq:S} becomes
$$1=  \sum_{j=1}^L \frac{m_j f_{j-1}^2 }{\lambda_{1,L} - f_{j-1} \| W \|_{1,L} } (1+ o_L(1)). $$
Also, by~\eqref{eq:lambda_def} we have 
\begin{equation*}
1=  \sum_{j=1}^L \frac{m_j f_{j-1}^2 }{\lambda_{1} - f_{j-1} \| W \|_{1,L} } (1+ o_L(1)).
\end{equation*}
 Hence, $\lim_{L\to \infty} \lambda_{1,L}  =\lambda_1$.
\end{proof}

In this subsection, we will consider the replacement matrix of the P\'olya scheme that corresponds to the 
multiplicative model and, in particular, we will determine its leading eigenvalue and the corresponding 
left-eigenvector. Let $M$ denote the replacement of the P\'olya scheme we use for the multiplicative 
model. In particular, $M_{(k,\ell),(i,j)} = a_{i,j} \Ex{\xi_{(i,j),(k,\ell)}}$, where 
$\xi_{(i,j),(k,\ell)}$ is the random variable which is the number of balls of colour $(k,\ell)$ that are produced 
when a ball of colour $(i,j)$ is selected. 

Note that $M$ is irreducible (since the associated directed graph is strongly connected). 
%
Thus, by the Perron-Frobenius theorem, the largest eigenvalue in absolute value is positive and real, and there exists a corresponding 
right eigenvector with positive entries. 
Let $\lambda_{1,L}$ denote the leading eigenvalue and $\U$ denote a right eigenvector corresponding 
to this eigenvalue: $M \cdot \U = \lambda_{1,L} \U$.

In what follows, we will give an implicit expression for $\lambda_{1,L}$. 
First, note that for $(i,j)$ with $i< L+1$ we have 
$M_{(i,j),(k,\ell)} >0$ if and only if $\ell=j$ or $\ell=i$; otherwise, it is equal to 0. 
In particular, $M_{(i,j), (k,j)} = p_{k,j} a_{k,j} m_i$ and $M_{(i,j),(k,i)}=p_{k,i} a_{k,i} m_j$. 
Also, for any $j=1,\ldots, L$, we have $M_{(L+1,j),(k,\ell)}= (1-p_{k,\ell}) m_j$.  

Thus, for colour $(i,j)$ with $i< L+1$ we have 
\begin{eqnarray} 
\lambda_{1,L} \U_{(i,j)} &=& \sum_{(k,\ell)} M_{(i,j),(k,\ell)} \U_{(k,\ell)} =m_i \sum_{k=1}^{L+1} p_{k,j} a_{k,j} \U_{(k,j)} +
m_j \sum_{k=1}^{L+1} p_{k,i} a_{k,i} \U_{(k,i)} \nonumber \\
&\stackrel{\eqref{eq:activities_rec_id}}{=}& 
m_i \sum_{k=1}^{L+1} a_{k-1,j-1} \U_{(k,j)} +
m_j \sum_{k=1}^{L+1}a_{k-1,i-1} \U_{(k,i)}.
\end{eqnarray}
We set $\A_{i} :=\sum_{k=1}^{L+1} a_{k-1,i-1} \U_{(k,i)}$, whereby the above equation 
becomes: 
\begin{equation}  \label{eq:non-dummy}
\lambda_{1,L} \U_{(i,j)} = m_i \A_j  + m_j \A_i .
\end{equation}
For colour $(L+1,j)$ with $j=1,\ldots, L$, we have 
\begin{equation*}
\lambda_{1,L} \U_{(L+1,j)} =  m_j \cdot \sum_{k,\ell=1}^L(1-p_{k,\ell}) a_{k,\ell}  \U_{(k,\ell)} 
\stackrel{\eqref{eq:activities_rec_id}}{=}m_j \sum_{k,\ell=1}^{L} (a_{k,\ell} - a_{k-1,\ell-1}) \U_{(k,\ell)}. 
\end{equation*}

But $a_{k,\ell} - a_{k-1,\ell-1} = f_k f_\ell - f_{k-1} f_{\ell-1} = k\ell h^2 - (k-1)(\ell -1)h^2 =(k+\ell -1)h^2$. 
Thus,
\begin{equation} \label{eq:dummy} 
\lambda_{1,L} \U_{(L+1,j)} = m_j h^2 \sum_{k, \ell=1}^{L} (k+\ell - 1)\U_{(k,\ell)}.
\end{equation}

We multiply both sides of \eqref{eq:non-dummy} by $a_{i-1,j-1}$ getting
\begin{equation*} 
\lambda_{1,L} a_{i-1,j-1}\U_{(i,j)} = m_i a_{i-1,j-1}\A_j  + m_j a_{i-1,j-1} \A_i.
\end{equation*}
Next, we sum both sides over $i$ and get 
\begin{eqnarray*}  
\lambda_{1,L} \sum_{i=1}^{L} a_{i-1,j-1} \U_{(i,j)}  &=& \A_j \sum_{i=1}^{L} m_i a_{i-1,j-1}  + 
m_j \sum_{i=1}^{L} a_{i-1,j-1} \A_i \nonumber \\
&=& \A_j f_{j-1} \sum_{i=1}^L m_i f_{i-1} + m_j f_{j-1} \sum_{i=1}^{L} f_{i-1}\A_i.  
\end{eqnarray*}

Similarly, we multiply both sides of \eqref{eq:dummy} by $a_{L+1,j}$ and get 
\begin{equation*} 
\lambda_{1,L} a_{L,j-1} \U_{(L+1,j)} = f_{j-1} m_j h^2 \sum_{k, \ell}^{L} (k+\ell - 1)\U_{(k,\ell)}. 
\end{equation*}
Adding the above two equations we get 
\begin{equation}\label{eq:intermediate_sum}  
\lambda_{1,L} \A_j = \A_j f_{j-1} \sum_{i=1}^L m_i f_{i-1} + m_j f_{j-1} \sum_{i=1}^{L} f_{i-1}\A_i
 +f_{j-1} m_j h^2 \sum_{k, \ell}^{L} (k+\ell - 1)\U_{(k,\ell)}. 
\end{equation}
Let $S = \sum_{\ell =1}^L \A_\ell$ and $M=\sum_{\ell=1}^L f_{\ell-1} \A_\ell$, $\| W \|_{1,L} =  \sum_{i=1}^L f_{i-1} m_i $. 

Setting $M=1$ and rearranging~\eqref{eq:intermediate_sum}, we obtain: 
\begin{equation} \label{eq:A_j}
\A_j = \frac{m_j f_{j-1}}{\lambda_{1,L} - f_{j-1} \| W \|_{1,L} } 
+ h^2\frac{f_{j-1} m_j \sum_{k, \ell=1}^{L} (k+\ell - 1)\U_{(k,\ell)}}
{\lambda_{1,L} - f_{j-1} \| W \|_{1,L}}.
\end{equation}
We multiply both sides of~\eqref{eq:A_j} by $f_{j-1}$ and sum over $j=1,\ldots, L$: the left-hand side is equal to $M=1$ and
\begin{eqnarray} 
1&=&  \sum_{j=1}^L \frac{m_j f_{j-1}^2 }{\lambda_{1,L} - f_{j-1} \| W \|_{1,L} } 
+ h^2\sum_{k, \ell=1}^{L} (k+\ell - 1)\U_{(k,\ell)} \sum_{j=1}^L \frac{f_{j-1} m_j}{\lambda_{1,L} - f_{j-1} \| W \|_{1,L}} 
\nonumber \\
&=& \sum_{j=1}^L \frac{m_j f_{j-1}^2 }{\lambda_{1,L} - f_{j-1} \| W \|_{1,L} } 
\left(1+ h^2 \sum_{k, \ell=1}^{L} (k+\ell - 1)\U_{(k,\ell)}  \right). \label{eq:S}
\end{eqnarray}
Furthermore, under the assumption that $M=1$, the substitution of~\eqref{eq:A_j} into \eqref{eq:non-dummy} yields: 
\begin{eqnarray*} 
\lambda_{1,L} \U_{(i,j)} &=& m_i m_j \left( 
\frac{f_{i-1}}{\lambda_{1,L} - f_{i-1} \| W \|_{1,L}} 
+\frac{f_{j-1}}{\lambda_{1,L} - f_{j-1} \| W \|_{1,L}} \right) \\ 
& &\ \ + h^2 m_i m_j \sum_{k, \ell=1}^{L} (k+\ell - 1)\U_{(k,\ell)}  \left( 
\frac{f_{i-1}}{\lambda_{1,L} - f_{i-1} \| W \|_{1,L}} 
+ \frac{f_{j-1}} {\lambda_{1,L} - f_{j-1} \| W \|_{1,L}} \right) \\
&=& m_i m_j \left( \frac{f_{i-1}}{\lambda_{1,L} - f_{i-1} \| W \|_{1,L}} 
+\frac{f_{j-1}}{\lambda_{1,L} - f_{j-1} \| W \|_{1,L}} \right)
\left(1+ h^2 \sum_{k, \ell=1}^{L} (k+\ell - 1)\U_{(k,\ell)}  \right).
\end{eqnarray*}

Now, recall that $\| W\|_1 = \Ex{W} = \int_0^1 x\dd \mu (x)$ and let $\lambda_1$ be the root of the following 
equation: 
\begin{equation} \label{eq:lambda_def}
1 = \int_0^1 \frac{x^2}{\lambda_1 - x \| W\|_1} \dd \mu(x).
\end{equation}
We now show that as $L\to \infty$, we have $\lambda_{1,L} \to \lambda_1$. 
\begin{clm} \label{clm:lambda_conv} If $\lambda_1>\| W\|_1$, then  
$$ \lim_{L \to \infty} \lambda_{1,L} = \lambda_1. $$ 
\end{clm}
\begin{proof} 
Recall first that $\lambda_{1,L}$ is the root of~\eqref{eq:S}. 
Furthermore, by~\eqref{eq:eigenvector}, since $\U_{(i,j)}>0$, for all $i,j$, it must be the 
case that $\lambda_{1,L}/ \| W\|_{1,L} > 1 - h$. 
But note that $\lim_{L\to \infty} \| W \|_{1,L} = \int_0^1 x \dd \mu(x) = \|W\|_1$.  
Hence, 
\begin{equation} \label{eq:lim_bound_lambda}
\liminf_{L \to \infty} \lambda_{1,L} \geq \| W \|_1.
\end{equation}

We need to show that 
\begin{equation} \label{eq:error_term}
\lim_{L\to \infty} h^2 \sum_{k, \ell=1}^{L} (k+\ell - 1)\U_{(k,\ell)}  = 0.
\end{equation}
Note that 
\begin{equation}\label{eq:1st_ub}
 \sum_{k, \ell=1}^{L} (k+\ell - 1)\U_{(k,\ell)} \leq 2 L \sum_{k,\ell=1}^L \U_{(k,\ell)}. 
\end{equation}

Now, we will bound $\sum_{k,\ell=1}^L \U_{(k,\ell)}$ and, in particular, since $h=1/L$, it suffices to show that 
$h \sum_{k,\ell=1}^L \U_{(k,\ell)} \to 0$ as $L \to \infty$.
Summing~\eqref{eq:non-dummy} over $i,j=1,\ldots L$, we deduce that  
\begin{equation} \label{eq:sum_u_rewrite}  \lambda_{1, L} \sum_{k,\ell=1}^L \U_{(k,\ell)}  = 2 \sum_{\ell =1}^{L} \A_\ell. 
\end{equation}
Now, since $f_{\ell -1} > L^{-1/2}$ for $\ell \geq L^{1/2}$, we can write 
\begin{eqnarray} 
\sum_{\ell =1}^{L} \A_\ell &=& L^{1/2} \sum_{\ell \geq L^{1/2}} \frac{1}{L^{1/2}} \A_\ell + 
\sum_{\ell < L^{1/2}} \A_\ell  \leq  L^{1/2} \sum_{\ell \geq L^{1/2}} f_{\ell-1} \A_\ell + 
\sum_{\ell < L^{1/2}} \A_\ell \nonumber \\ 
&\leq&  L^{1/2} \sum_{\ell =1}^L f_{\ell-1} \A_\ell + 
\sum_{\ell < L^{1/2}} \A_\ell  =  L^{1/2}+ 
\sum_{\ell < L^{1/2}} \A_\ell . \label{eq:S_bound}
\end{eqnarray}
Now, by~\eqref{eq:A_j}, we have 
$$ \sum_{\ell < L^{1/2}} \A_\ell  < \sum_{\ell < L^{1/2}} \frac{m_\ell f_{\ell-1}}{\lambda_{1,L} - f_{\ell-1}\| W \|_{1,L}} 
(1+4 h \sum_{\ell=1}^L \A_\ell). $$
But by~\eqref{eq:lim_bound_lambda}, and since $f_{\ell-1} < L^{-1/2}$, for $\ell < L^{1/2}$,  
it follows that for $L$ sufficiently large 
$$\ \frac{1}{\lambda_{1,L} - f_{\ell-1}\| W \|_{1,L}} \leq 2.$$
Hence, 
$$ \sum_{\ell < L^{1/2}} \A_\ell  < 2\sum_{\ell < L^{1/2}} m_\ell f_{\ell-1}
\left(1+4 h \sum_{\ell=1}^L \A_\ell \right). $$
Substituting this into~\eqref{eq:S_bound} we get 
\begin{eqnarray*}
\sum_{\ell =1}^{L} \A_\ell \left( 1- 8h \sum_{\ell < L^{1/2}} m_\ell f_{\ell-1} \right) \leq 
L^{1/2}+ 2\sum_{\ell < L^{1/2}} m_\ell f_{\ell-1}\leq L^{1/2}+ 2 \| W\|_{1,L}.
\end{eqnarray*} 
Since $h=1/L$, we thus conclude that 
$$h\sum_{\ell =1}^{L} \A_\ell = O(L^{-1/2}).  $$
Using the identity~\eqref{eq:sum_u_rewrite} and the upper bound~\eqref{eq:1st_ub}, we deduce 
\eqref{eq:error_term}.
Thus, \eqref{eq:S} becomes
$$1=  \sum_{j=1}^L \frac{m_j f_{j-1}^2 }{\lambda_{1,L} - f_{j-1} \| W \|_{1,L} } (1+ o_L(1)). $$
Also, by~\eqref{eq:lambda_def} we have 
\begin{equation*}
1=  \sum_{j=1}^L \frac{m_j f_{j-1}^2 }{\lambda_{1} - f_{j-1} \| W \|_{1,L} } (1+ o_L(1)).
\end{equation*}
 Hence, $\lim_{L\to \infty} \lambda_{1,L}  =\lambda_1$.
\end{proof}

\subsection{Coupling} 
We will now describe a coupling between the multiplicative model and the above P\'olya urn scheme, in 
which the former dominates the latter on the balls of type at most $m$. 
The coupling is similar to the one introduced by Borgs et al.~\cite{}. 

$\omega, w$

Let $\U_{(i,j)}(n)$ denote the number of balls of colour $(i,j)$ in the P\'olya scheme with replacement 
matrix $M$ after $n$ steps. 
Furthermore, let $(\widehat{\U}_{(i,j)} , 1 \leq i  \leq L+1,  1\leq j \leq L)$ denote the 
right eigenvector of $M$, rescaled so that $\sum_{i=1}^{L+1} \sum_{j=1}^L a_{i,j} \widehat{\U}_{(i,j)} 
=1$. By Theorem~?? in Janson~\cite{} we have that  
for any $(i,j) \in [L+1]\times [L]$ 
almost surely
$\lim_{n\to \infty} \frac{1}{n} \U_{(i,j)}(n) = \widehat{\U}_{(i,j)}$. 
In the proof of the following proposition we will give a coupling between the two processes such that 
the asymptotic fraction of edges of type $(i,j)$ is at least $\widehat{\U}_{(i,j)}$, for $(i,j) \in [L]\times [L]$. 
Let $\Ve_{(i,j)} (n)$ denote the number of all directed edges in $\overrightarrow{\T}_n$, whose tail has fitness within the interval 
$I_i$ and its head has fitness in $I_j$. 
\begin{prop} 
For all $1\leq i,j\leq L$, we have
$$ \liminf_{n\to \infty} \frac{\Ve_{(i,j)}(n)}{n}\geq \widehat{\U}_{(i,j)}.$$
almost surely. 
\end{prop}
\begin{proof} 
In the process $(\T_n)_{n\geq 0}$, recall that $W_n$ is the fitness of vertex $n$. 
Let $F_n$ be the random variable that is the fitness of the vertex selected from $\T_{n-1}$ during the execution 
of the $n$th step. In other words, $F_n$ is the fitness of the vertex of $\T_{n-1}$ that will be the neighbour of 
vertex $n$. We denote by $\rho_{i,j}(n)$ the probability that

Let $\rho_{(i,j)}(n)$ be the probability that a ball of colour $(i,j)$ will be selected from the urn during 
the execution of the $n$th step. Similarly, let $p_{i,j}(n)$ be the probability that the edge selected from $\overrightarrow{\T}_n$ has tail with fitness in $I_i$ and head with fitness in $I_j$. 
We will provide a coupling of $(\Ve_{(i,j)} (n), i,j=1,\ldots, L)$ with $(\U_{(i,j)} (n), i \in [L+1], j\in [L])$ such that 
for all $n$ with probability 1 we have 
\begin{eqnarray} 
\U_{(i,j)} (n) \leq \Ve_{(i,j)} (n), \mbox{for $i,j =1,\ldots, L$}, \label{eq:coupl_1} \\
\rho_{i,j} (n) \leq p_{i,j} (n), \mbox{for $i,j=1,\ldots, L$}. \label{eq:coupl_2}
\end{eqnarray}
\end{proof}

We couple the two processes as follows:
Suppose, by induction that $\Ve(i,j) \geq U(i,j)$. Moreover, assume that $Z^{(T)}_{n} \leq Z^{(P}_{n}$. We perform the next step in the Polya urn as follows:

\section{Finitely supported measures on $[0,1]$}

Let $\mu$ be a probability measure on $[0,1]$, we will consider point measures on $[0,1]$ that are close 
to $\mu$ in total variation distance. The aim is to study the limiting distribution of the associated P\'olya scheme 
when a point measure $\nu$ approaches $\mu$ in the above sense. 

\subsection{The P\'olya scheme for the multiplicative model: eigenvectors and eigenvalues} 

In this subsection, we will consider the replacement matrix of the P\'olya scheme that corresponds to the 
multiplicative model and, in particular, we will determine its leading eigenvalue and the corresponding 
left eigenvector. 
We shall assume that the fitness of each vertex is sampled from a finitely supported 
measure $\nu$ on $[0,1]$ having support $\supp {\nu} = \{ x_1, \ldots, x_L\}$ with $0< x_1 < \cdots < x_L\leq 1$. A directed edge whose endvertices has fitness $(x_i,x_j)$ has colour $(i,j)$. 

Let $M$ denote the replacement of the P\'olya scheme we use for the multiplicative 
model. In particular, $M_{(k,\ell),(i,j)} = a_{i,j} \Ex{\xi_{(i,j),(k,\ell)}}$, where 
$\xi_{(i,j),(k,\ell)}$ is the random variable which is the number of balls of colour $(k,\ell)$ that are produced 
when a ball of colour $(i,j)$ is selected. 

Note that $M$ is irreducible. Indeed, let $M^t$ denote the $t$th power of $M$. 
We will argue that for any colours $(i,j), (k,\ell)$ there exists $t>0$ such that 
$M^t_{(k,\ell),(i,j)} > 0$.  Indeed, we can take $t=2$: a ball of colour $(k,\ell)$ can be produced in 2 steps
from a ball of colour $(i,j)$, if a ball of colour $(j,k)$ is produced (with probability $\nu_k$) and then a ball
of colour $(k,\ell)$ is produced (with probability $\nu_\ell$).  

Frobenius' theorem states the largest (in absolute value) eigenvalue is positive and real with the corresponding 
right eigenvector being pointwise positive. 
Let $\eig{\nu}$ denote the leading eigenvalue and $\U$ denote a right eigenvector corresponding 
to this eigenvalue; thus, $M \cdot \U = \eig{\nu} \U$.

In what follows, we will give an implicit expression for $\eig{\nu}$ and, furthermore, we will determine 
$\U$. We write the defining equation for $\eig{\nu}$ and $\U$ in a detailed form.
Firstly, note that for $(i,j)$ we have 
$M_{(i,j),(k,\ell)} >0$ if and only if $\ell=j$ or $\ell=i$; otherwise, it is equal to 0. 
To be more specific, $M_{(i,j), (k,j)} = a_{k,j} \nu_i$ and $M_{(i,j),(k,i)}=a_{k,i} \nu_j$. 

Thus, for colour $(i,j)$  we have 
\begin{eqnarray*} 
\eig{\nu} \U_{(i,j)} &=& \sum_{(k,\ell)} M_{(i,j),(k,\ell)} \U_{(k,\ell)} =\nu_i \sum_{k=1}^{L} a_{k,j} \U_{(k,j)} +
\nu_j \sum_{k=1}^{L} a_{k,i} \U_{(k,i)}. 
\end{eqnarray*}
We set $\A_j :=\sum_{k=1}^{L} a_{k,j} \U_{(k,j)}$.
\begin{equation}  \label{eq:non-dummy}
\eig{\nu} \U_{(i,j)} = \nu_i \A_j  + \nu_j \A_i .
\end{equation}

We multiply both sides of \eqref{eq:non-dummy} by $a_{i,j}$ getting
\begin{equation*} 
\eig{\nu} a_{i,j}\U_{(i,j)} = 
a_{i,j} \nu_i \A_j + a_{i,j} \nu_j \A_i.
\end{equation*}

Recall that $a_{i,j} = h_0(x_i) h_1(x_j)$. 
Adding the above equation over $i \in [L]$, we get 
\begin{equation}\label{eq:intermediate_sum}  
\eig{\nu} \A_j = \A_j h_1(x_j) \sum_{i=1}^L \nu_i h_0(x_i) + \nu_j h_1(x_j) \sum_{i=1}^{L} h_0(x_i) \A_i. 
\end{equation}

Let $N = \sum_{\ell =1}^L \A_\ell$ and $S=\sum_{\ell=1}^L h_0(x_\ell) \A_\ell$, 
$\E[\nu]{h_0} =  \sum_{i=1}^L h_0(x_i) \nu_i $. 

Rearranging~\eqref{eq:intermediate_sum}, we obtain: 
\begin{equation} \label{eq:A_j}
\A_j = \frac{\nu_j h_1 (x_j)}{\eig{\nu} - h_1(x_j) \E[\nu]{h_0} }\cdot S.
\end{equation}
We multiply both sides of~\eqref{eq:A_j} by $h_0(x_j)$ and sum over $j=1,\ldots, L$: the left-hand side is equal to 
$S$ and
\begin{eqnarray*} 
S&=&  
 \sum_{j=1}^L \frac{\nu_j h_0(x_j) h_1(x_j)}{\eig{\nu} - h_1(x_j) \E[\nu]{h_0} }\cdot S. \label{eq:S}
\end{eqnarray*}
Hence, 
\begin{equation} \label{eq:S} 
\sum_{j=1}^L \frac{\nu_j h_0(x_j) h_1(x_j)}{\eig{\nu} - h_1(x_j)\E[\nu]{h_0}}=1.
\end{equation}
Furthermore, the substitution of~\eqref{eq:A_j} into \eqref{eq:non-dummy} yields: 
\begin{eqnarray} 
\eig{\nu} \U_{(i,j)} &=& 
 \nu_i \nu_j \left( \frac{h_1(x_i) }{\eig{\nu} -h_1( x_i) \E[\nu]{h_0}} 
+\frac{h_1(x_j)}{\eig{\nu} - h_1(x_j) \E[\nu]{h_0}} \right)\cdot S. \nonumber \\
& & \label{eq:eigenvector}
\end{eqnarray}
Multiplying both sides by $a_{i,j} = h_0(x_i) h_1(x_j)$ and summing over $i,j= 1,\ldots, L$. 
\begin{eqnarray*}
\eig{\nu} \sum_{i,j=1}^L a_{i,j} \U_{(i,j)} &=& \left( \E[\nu]{h_1} \E[\nu] {\frac{h_0 h_1}{\eig{\nu} -h_1 \E[\nu]{h_0}}} + 
\E[\nu]{h_0} \E[\nu] {\frac{h_1^2}{\eig{\nu} - h_1 \E[\nu]{h_0}}} \right) \cdot S \nonumber \\
&\stackrel{\eqref{eq:S}}{=}& 
 \left( \E[\nu]{h_1}+ 
\E[\nu]{h_0} \E[\nu] {\frac{h_1^2}{\eig{\nu} - h_1 \E[\nu]{h_0}}} \right) \cdot S.
\end{eqnarray*}
Note that $\sum_{i,j=1}^L a_{i,j} \U_{(i,j)} =\sum_{j=1}^L \A_j = N$. Normalising $\U$ so that $N=1$, we get 
\begin{equation} \label{eq:S_expression}
\eig{\nu} =  \left( \E[\nu]{h_1}+ 
\E[\nu]{h_0} \E[\nu] {\frac{h_1^2}{\eig{\nu} - h_1 \E[\nu]{h_0}}} \right) \cdot S.
\end{equation}
Setting $T_{h_0,h_1}^{(\nu)}:= \E[\nu]{h_1}+ 
\E[\nu]{h_0} \E[\nu] {\frac{h_1^2}{\eig{\nu} - h_1 \E[\nu]{h_0}}}$, \eqref{eq:eigenvector} implies that under 
the normalisation $N=1$ we have 
\begin{equation} \label{eq:eigenvector_pure} 
\U_{i,j} = \nu_i \nu_j \left( \frac{h_1(x_i) }{\eig{\nu} -h_1( x_i) \E[\nu]{h_0}} 
+\frac{h_1(x_j)}{\eig{\nu} - h_1(x_j) \E[\nu]{h_0}} \right) \cdot (T_{h_0,h_1}^{(\nu)} )^{-1}.
\end{equation}
We define the measure on the $\mathscr{B} ([0,1]^2)$ 
$$ U^{(\nu)}(\cdot) =\left( \iint_{\cdot} 
 \left( \frac{h_1(x) }{\eig{\nu} -h_1( x) \E[\nu]{h_0}} 
+\frac{h_1(y)}{\eig{\nu} - h_1(y) \E[\nu]{h_0}} \right) \dd \nu (x) \dd \nu (y) \right)\frac{\eig{\nu}}{T_{h_0,h_1}^{(\nu)}}. 
 $$

We also define the measure $\Gamma_n^{(\nu)}$ on $\mathscr{B}([0,1]^2)$ as 
$$ \Gamma_n^{(\nu)} (\cdot ) = \frac{1} {n}\cdot \sum_{i,j=1}^{L}\delta((x_i,x_j)\in \cdot).$$
Theorem~?? from~\cite{} implies that almost surely
$$ \Gamma_n \to U^{(\nu)}, $$ 
weakly as $n\to \infty$.

\section{The behaviour of $U^{(\nu)}$ in the vicinity of $\mu$.}
Now, let $\mu$ be a measure on $\mathscr{B}([0,1])$. We will consider the limiting behaviour of 
$U^{(\nu)}$ as $\nu$ approaches $\mu$ in total variation distances. Our only assumption on $\mu$ 
is that $h_1^{-1} (\{1\})$ contains no atom. 

If $\E[\mu]{h_0}<\infty$, we define a normalisation 
$\mu \mapsto \wh{\mu}$ such that  
$\wh{\mu} ( \cdot )= \frac{1}{\E[\mu]{h_0}} \int_\cdot h_0(x) \dd \mu (x)$.

Suppose that the following condition holds:
\begin{equation} \label{eq:cond_ass} 
\int_0^1 \frac{ h_1(x)}{1- h_1(x)} \dd \wh{\mu} (x) < 1.
\end{equation}

We will first show that $\eig{\nu}/\E[\nu]{h_0}$ is close to 1 provided that $\nu$ is close enough to $\mu$. 
For any $S \subseteq [0,1]$, we denote by $h_1^{-1} (S) = \{x\in [0,1] \ : \ h_1(x) \in S \}$, that is,  
the \textit{pullback} of $S$ under $h_1$. 
Recall that our assumption is that $h_1^{-1} (\{1\})$ contains no atom. Thus,  
$\mu ( h_1^{-1} ((1-\eps,1]) ) \downarrow 0$ as $\eps \downarrow 0$. 
Furthermore, the following holds. 
\begin{clm} \label{clm:sup_h1}
For any $0< \eps <1$, 
$\sup_{x_i \in \supp{\nu}} \{ h_1 (x_i)\} > 1- \eps$, if $d_{TV}(\nu, \mu) <  \mu ( h_1^{-1} ((1-\eps,1] ) )$. 
\end{clm}
\begin{proof}
Indeed, if $\sup_{x_i \in \supp{\nu}} \{ h_1 (x_i)\} \leq 1- \eps$, then 
$\supp{\nu} \cap  h_1^{-1} ((1-\eps,1]) = \varnothing$. In this case, 
$d_{TV} (\nu, \mu) \geq \mu ( h_1^{-1} ((1-\eps,1]) )$.
\end{proof}
For convenience, we will set $f(\lambda ,x) = \frac{h_1(x)}{\lambda - h_1(x)}$; thus $f: \R\times [0,1] \to \R$.  
Hence, the condition in~\eqref{eq:cond_ass}  can be written as 
\begin{equation} \label{eq:cond_ass_rev} \int_0^1 f(1, x) \dd \wh{\mu} (x) <1. 
\end{equation} 
We also set $\eign{\nu} = \eig{\nu}/\E[\nu]{h_0}$. 

Furthermore, with this notation we can write  for any $B,B' \in \mathscr{B}([0,1])$
\begin{eqnarray}
\lefteqn{\Xi^{(\nu)}(B\times B') :=\int_{B\times B'}h_0 (x) h_1(y)  \dd U^{(\nu)} ((x,y)) } \nonumber \\
&& 
=\frac{1}{T_{h_0,h_1}^{(\nu)}} \left(
\E[\nu]{h_1 \ind{B'}} \int_{B} f(\eign{\nu}, x) \dd \wh{\nu} (x)  
+ \E[\nu]{h_0 \ind{B}} \int_{B'} h_1(y) f(\eign{\nu}, y) \dd \wh{\nu}(y) \right). \nonumber
\end{eqnarray}
(Note also that $\Xi^{(\nu)} ([0,1]^2) = 1$.)

If 
\begin{equation} \label{eq:non_cond_ass} 
\int_0^1 \frac{ h_1(x)}{1- h_1(x)} \dd \wh{\mu} (x) \geq 1,
\end{equation}
then, we let $\lamu \geq 1$ be the unique solution to 
$$ \int_0^1 \frac{ h_1(x)}{\lamu- h_1(x)} \dd \wh{\mu} (x) = 1.$$
If~\eqref{eq:cond_ass} holds, then we set $\lamu=1$. 

We now define the measure $\Lm$ on $\mathscr{B} ([0,1]^2)$ to be such that for any $B, B' \in \mathscr{B} ([0,1])$
$$ \Lm (B\times B') = \frac{1}{T_{h_0,h_1}^{(\mu)}} \left( \E[\mu]{h_1 \ind{B'}} \int_{B} f(\lamu, x) \dd \wh{\mu} (x)  
+ \E[\mu]{h_0 \ind{B}} \int_{B'} h_1(y) f(\lamu, y) \dd \wh{\nu}(y) \right).$$
where $T_{h_0,h_1}^{(\mu)} = \E[\mu]{h_1} + \E[\mu]{h_0} \int_0^1 h_1 (y) f(\lamu,y) \dd \wh{\mu}(y)$.
Note that 
\begin{eqnarray}  
\Lm ([0,1)\times [0,1]) &=& \frac{1}{T_{h_0,h_1}^{(\mu)}} \left( \E[\mu]{h_1} \int_0^1 f(\lamu, x) \dd \wh{\mu} (x)+ \E[\mu]{h_0} \int_0^1 h_1 (y) f(\lamu,y) \dd \wh{\mu}(y) \right) \nonumber  \\
&=& \begin{cases} 
< 1 & \mbox{if~\eqref{eq:cond_ass} holds} \\
= 1 & \mbox{if~\eqref{eq:non_cond_ass} holds}
\end{cases}. \label{eq:transition}
\end{eqnarray}

We will show that for any open sets $B,B' \subset [0,1]$
\begin{equation} \label{eq:weak_conv} 
\liminf_{\nu \to \mu} \Xi^{(\nu)} (B \times B') \geq \Lm (B\times B').
\end{equation}
Firstly as $h_0, h_1$ are continuous and bounded (or bounded Lipschitz) on $[0,1]$
we have 
\begin{equation} \label{eq:terms_conv}
\lim_{\nu \to \mu} \E[\nu]{h_1 \ind{B'}} = \E[\mu]{h_1 \ind{B'}} \ \mbox{and} \ 
\lim_{\nu \to \mu} \E[\nu]{h_0 \ind{B}} = \E[\mu]{h_0 \ind {B'}}.
\end{equation}

Now, we will show that 
\begin{eqnarray} 
& &\liminf_{\nu \to \mu}  \int_{B} f(\eign{\nu}, x) \dd \wh{\nu} (x)  \geq  \int_{B} f(\lamu, x) \dd \wh{\mu} (x)  \label{eq:ints_conv_I} \\
&&\ \mbox{and} \ 
\liminf_{\nu \to \mu} \int_{B'} h_1(y) f(\eign{\nu}, y) \dd \wh{\nu}(y) \geq \int_{B'} h_1(y) f(\lamu, y) \dd \wh{\nu}(y).
\label{eq:ints_conv_II}
\end{eqnarray}
We will need the fact that $\eign{\nu}$ converges to $\lamu$ as $\nu \to \mu$ - we postpone its proof until later.
\begin{clm} \label{clm:eig_lim_super} We have 
$$\lim_{\nu \to \mu} \left| \eign{\nu}-\lamu \right| = 0. $$
\end{clm}
Recall that $\lamu \geq 1$. Hence for any $\eps>0$
$$ \liminf_{\nu \to \mu}  \int_{B} f(\eign{\nu}, x) \dd \wh{\nu} (x)  \geq 
 \liminf_{\nu \to \mu}  \int_{B} f(\lamu+\eps, x) \dd \wh{\nu} (x). $$
 But for any $\eps> 0$, the function $f(\lamu + \eps, x)$ is bounded and continuous with respect to $x\in B$. 
 So in fact, 
 $$\liminf_{\nu \to \mu}  \int_{B} f(\lamu+\eps, x) \dd \wh{\nu} (x) = \lim_{\nu \to \mu}  \int_{B} f(\lamu+\eps, x) \dd \wh{\nu} (x) = \int_{B} f(\lamu+\eps, x) \dd \wh{\mu} (x).$$
But for all $x\in[0,1]$,  $\ind{B}(x) f(\lamu+\eps, x) \uparrow \ind{B}(x) f(\lamu,x)$ as $\eps \downarrow 0$. 
As $f \geq 0$, the monotone convergence theorem is applicable, whereby
$$ \lim_{\eps \downarrow 0} \int_{B} f(\lamu+\eps, x) \dd \wh{\mu} (x) = \int_{B} f(\lamu,x) \dd \wh{\mu}.$$ 
Hence,~\eqref{eq:ints_conv_I} follows. The same argument applies also to the proof of~\eqref{eq:ints_conv_II}.

Finally, we need to prove that 
\begin{equation}\label{eq:T_lims} \lim_{\nu \to \mu} T_{h_0,h_1}^{(\nu)} = T_{h_0,h_1}^{(\mu)}. \end{equation}
We already have the limits in~\eqref{eq:terms_conv}; so it suffices to show that 
$$ \lim_{\nu \to \mu} \int_0^1 h_1 (y) f(\eign{\nu},y) \dd \wh{\nu} (y) =  
\int_0^1 h_1 (y) f(\lamu, y) \dd \wh{\mu} (y). $$ 
By~\eqref{eq:ints_conv_II}, it only suffices to show that 
$$ \limsup_{\nu \to \mu} \int_0^1 h_1 (y) f(\eign{\nu},y) \dd \wh{\nu} (y) =  
\int_0^1 h_1 (y) f(\lamu, y) \dd \wh{\mu} (y).$$
By Claim~\ref{clm:sup_h1}, for any $\eps >0$ we have 
\begin{eqnarray} 
\lefteqn{ \limsup_{\nu \to \mu} \int_0^1 h_1 (y) f(\eign{\nu},y) \dd \wh{\nu} (y) \leq} \nonumber \\
&&  \limsup_{\nu \to \mu} \int_{h_1^{-1} ([0,1-\eps^{1/2}])} h_1 (y) f(\lamu - \eps ,y) \dd \wh{\nu} (y) 
 +  \limsup_{\nu \to \mu} \int_{h_1^{-1} ((1-\eps^{1/2},1])} h_1 (y) f(\eign{\nu},y) \dd \wh{\nu} (y).  \nonumber \\
&& \label{eq:limsup}
\end{eqnarray}
The function $f(\lamu -\eps,x)$ is bounded and continuous with respect to $x \in h_1^{-1} ([0,1-\eps^{1/2}])$. 
Thus, 
\begin{eqnarray} 
\lefteqn{ \limsup_{\nu \to \mu} \int_{h_1^{-1} ([0,1-\eps^{1/2}])} h_1 (y) f(\lamu - \eps ,y) \dd \wh{\nu} (y) } \nonumber \\
&=&  \lim_{\nu \to \mu} \int_{h_1^{-1} ([0,1-\eps^{1/2}])} h_1 (y) f(\lamu - \eps ,y) \dd \wh{\nu} (y) \nonumber \\
&=&   \int_{h_1^{-1} ([0,1-\eps^{1/2}])} h_1 (y) f(\lamu - \eps ,y) \dd \wh{\mu} (y). \nonumber 
\end{eqnarray}
Now, for $y \in h_1^{-1} ([0,1-\eps^{1/2}])$ we bound the difference: 
\begin{eqnarray}
\frac{1}{h_1(y)} (f(\lamu -\eps, y) - f (\lamu , y)) &=& \frac{\lamu -h_1(y) - (\lamu - \eps) + h_1(y)}{(\lamu -\eps - h_1(y)) (\lamu - h_1(y))} \nonumber \\
&=& \frac{\lamu  - (\lamu - \eps)}{(\lamu -\eps - h_1(y)) (\lamu - h_1(y))}  \nonumber \\
&=& \frac{\eps}{(\lamu -\eps - h_1(y)) (\lamu - h_1(y))}. \nonumber
 \end{eqnarray}
But since $h_1(y) \leq 1- \eps^{1/2}$, 
$$\lamu -\eps - h_1(y) \geq \eps^{1/2} - \eps \stackrel{\eps < 1/4}{>} \frac{1}{2} \eps^{1/2}. $$
Thereby, for any $y \in h_1^{-1} ([0,1-\eps^{1/2}])$
$$f(\lamu -\eps, y) - f (\lamu , y) \leq 2\eps^{1/2} \frac{h_1(y)}{ \lamu - h_1(y)}. $$
This implies that for any $\eps < 1/4$
\begin{eqnarray*}
\int_{h_1^{-1} ([0,1-\eps^{1/2}])} h_1 (y) f(\lamu - \eps ,y) \dd \wh{\nu} (y)  &\leq& 
(1+2\eps^{1/2}) \int_{h_1^{-1} ([0,1-\eps^{1/2}])} h_1 (y) f(\lamu,y) \dd \wh{\nu} (y)  \\
&\leq& (1+2\eps^{1/2})  \int_0^1 h_1 (y) f(\lamu,y) \dd \wh{\nu} (y).
 \end{eqnarray*}
So for any such $\eps$
\begin{equation} \label{eq:sup_part_I}
\limsup_{\nu \to \mu} \int_{h_1^{-1} ([0,1-\eps^{1/2}])} h_1 (y) f(\lamu - \eps ,y) \dd \wh{\nu} (y)  \leq 
 (1+2\eps^{1/2})  \int_0^1 h_1 (y) f(\lamu,y) \dd \wh{\nu} (y). 
 \end{equation}
Now, we will show that 
\begin{equation} \label{eq:sup_part_II}
\lim_{\eps \downarrow 0} 
\limsup_{\nu \to \mu} \int_{h_1^{-1} ((1-\eps^{1/2},1])} h_1 (y) f(\eign{\nu},y) \dd \wh{\nu} (y) = 0.
\end{equation}
Since $h_1(y) \leq 1$ and $f$ is non-negative, it suffices to show that 
\begin{equation} \label{eq:sup_part_II_new}
\lim_{\eps \downarrow 0} 
\limsup_{\nu \to \mu} \int_{h_1^{-1} ((1-\eps^{1/2},1])} f(\eign{\nu},y) \dd \wh{\nu} (y) = 0.
\end{equation}
In turn, since $\int_0^1 f(\eign{\nu},y) \dd \wh{\nu} (y) =1$, we can show that 
\begin{equation} \label{eq:sup_part_II_compl}
\lim_{\eps \downarrow 0} 
\liminf_{\nu \to \mu} \int_{h_1^{-1} ([0,1-\eps^{1/2}])} f(\eign{\nu},y) \dd \wh{\nu} (y) \geq 1.
\end{equation}
But by Claim~\ref{clm:eig_lim_super}, $\limsup_{\nu \to\mu} \eign{\nu} \leq \lamu +\eps$. Hence, 
$\eign{\nu} \leq \lamu+\eps$ for all $\nu$ that are close enough to $\mu$. 
So  
$$\liminf_{\nu \to \mu} \int_{h_1^{-1} ([0,1-\eps^{1/2}])} f(\eign{\nu},y) \dd \wh{\nu} (y) \geq 
\liminf_{\nu \to \mu} \int_{h_1^{-1} ([0,1-\eps^{1/2}])} f(\lamu + \eps,y) \dd \wh{\nu} (y).$$
But $\lamu \geq 1$, whereby $\lamu + \eps > 1$. So $f(\lamu + \eps , y)$ is bounded and 
continuous in $h_1^{-1} ([0,1-\eps^{1/2}])$. Hence, 
$$ \liminf_{\nu \to \mu} \int_{h_1^{-1} ([0,1-\eps^{1/2}])} f(\lamu + \eps,y) \dd \wh{\nu} (y) = 
\int_{h_1^{-1} ([0,1-\eps^{1/2}])} f(\lamu + \eps,y) \dd \wh{\mu} (y).$$
But $f(\lamu + \eps ,y) \ind{h_1^{-1} ([0,1-\eps^{1/2}])} (y) \uparrow f(\lamu, y)$, as $\eps \downarrow 0$, for 
any $y \in[0,1]$. Since $f$ is non-negative, the monotone convergence theorem implies that 
$$\lim_{\eps \downarrow 0} \int_{h_1^{-1} ([0,1-\eps^{1/2}])} f(\lamu + \eps,y) \dd \wh{\mu} (y) = 
\int_0^1 f(\lamu, y) \dd \wh{\mu} (y) = 1. $$
Thus,~\eqref{eq:sup_part_II_compl} follows and, thereby,~\eqref{eq:sup_part_II_new} and~\eqref{eq:sup_part_II}
as well.

We conclude that by~\eqref{eq:sup_part_I} and~\eqref{eq:sup_part_II}, the righthand side 
of~\eqref{eq:limsup} converges to $ \int_0^1 h_1 (y) f(\lamu,y) \dd \wh{\nu} (y)$ as $\eps \downarrow 0$. 
This concludes the proof of~\eqref{eq:T_lims}. 

We now finish by giving the proof of Claim~\ref{clm:eig_lim_super}.
\begin{proof}[Proof of Claim~\ref{clm:eig_lim_super}] 
Suppose that this is not the case and, in particular, that there exists $\eps>0$ and a sequence of 
 point measures $(\nu^{(m)})_{m \in \mathbb{N}}$ such that $\nu^{(m)} \to \mu$ as $m \to \infty$ for which 
 either $\frac{\eig{\nu^{(m)}}}{\E{\nu^{(m)}}{h_0}} > \lambda^{(\mu)} + \eps$,  for any $m$ sufficiently large 
 or $\frac{\eig{\nu^{(m)}}}{\E{\nu^{(m)}}{h_0}} < \lambda^{(\mu)} - \eps$,  for any $m$ sufficiently large.
 
 Suppose that the former holds. In this case, ~\eqref{eq:S} implies that for $m$ sufficiently large
$$ \frac{1}{\E{\nu^{(m)}}{h_0}}\cdot \sum_{j=1}^L \frac{\nu_j^{(m)} h_0(x_j) h_1(x_j)}{\lambda^{(\mu)}+\eps - h_1(x_j)}>1.$$
In other words, for any such $m$
$$\int_0^1 \frac{h_1(x)}{\lambda^{(\mu)} + \eps -h_1(x)} \dd \wh{\nu^{(m)}} (x) = \int_0^1 
f(\lambda^{(\mu)} + \eps, x) \dd \wh{\nu^{(m)}} (x) >1.$$
 Since $\lambda^{(\mu)}\geq 1$, the function $f(\lambda^{(\mu)} + \eps, x)$ 
 is bounded and continuous and therefore as $m\to \infty$
 $$ \int_0^1f(\lambda^{(\mu)} + \eps, x) \dd \wh{\nu^{(m)}} (x) \to 
 \int_0^1f(\lambda^{(\mu)} + \eps, x) \dd \wh{\mu} (x) < 1.
 $$
 which yields a contradiction.
 
 So now suppose that  $\frac{\eig{\nu^{(m)}}}{\E{\nu^{(m)}}{h_0}} < \lambda^{(\mu)} - \eps$ for all $m$ 
 that are sufficiently large. For all such $m$ we have 
 $$ \int_0^1 \frac{h_1(x)}{\lambda^{(\mu)} - \eps -h_1(x)} \dd \wh{\nu^{(m)}} (x) = 
 \int_0^1 f(\lambda^{(\mu)} - \eps, x)  \dd \wh{\nu^{(m)}}< 1.$$
 Suppose first that  $\lambda^{(\mu)} > 1$. Then we may assume that $\eps>0$ is small enough so 
 that $\lambda^{(\mu)} -\eps >1$. 
 Since $h_1 (x) \leq 1$ the function $f(\lambda^{(\mu)} - \eps, x)$ 
 is bounded and continuous and therefore as $m\to \infty$
 $$ \int_0^1 f(\lambda^{(\mu)} - \eps, x) \dd \wh{\nu^{(m)}} (x) \to 
 \int_0^1 f(\lambda^{(\mu)} - \eps, x) \dd \wh{\mu} (x) > 1.
 $$
 which yields a contradiction.
Suppose now that $\lambda^{(\mu)}=1$.  
 But Claim~\ref{clm:sup_h1} together with the assumption that $\lim_{m \to \infty} \nu^{(m)} = \mu$ imply that $\liminf_{m\to \infty} \frac{\eig{\nu^{(m)}}}{\E{\nu^{(m)}}{h_0}} \geq 1$. In other words, 
 for $m$ sufficiently large, $\frac{\eig{\nu^{(m)}}}{\E{\nu^{(m)}}{h_0}} \geq 1 -\eps = \lambda^{(\mu)} -\eps$. 
 \end{proof}
\section{The degree profile of $\T_n$} 

In this section, we will determine the fraction of vertices in $\T_n$ having a certain degree. 
In particular, we will consider the fraction of vertices with a given number of offspring $k\in \N_0$. 
We will denote by $N_k (n)$ the number of such vertices in $\T_n$ and set $p_k(n) = \frac{1}{n} 
\Ex{N_k (n)}$. 

Of particular importance for the asymptotics of $p_k(n)$ is the P\'olya process associated with the neighbourhood
of a particular vertex. More specifically, for a vertex $v$ of fitness $W = x_{J}$ (a random variable sampled from 
$\nu$ independently of everything else) whose parent has fitness $x_I$ consider the P\'olya urn starting with a
single ball of colour $(I,J)$. At each step a ball is selected with probability proportional to its activity, and it is 
put back together with a ball of colour $(I',J)$ with probability $\nu (x_{I'})$. This process describes the sub-urn 
which consists of those balls that represent edges pointing towards $v$. 
We call $v$ the \textit{root} of this process. 
One can view this as a Markov chain $(S_\ell )_\ell$ taking values in $\N_0^{[L]}$. 
The component $S_\ell (i)$ is the number of directed edges/balls of colour $(i,J)$. 
We let $\mathbb{P}^\star_{(J,S)}(\cdot )$ and $\mathbb{E}^\star_{(J,S)}(\cdot)$ denote the probability 
measure and the corresponding expectation operator of this Markov chain starting with $S(i)$ balls 
of colour $(i,J)$, for $i \in [L]$. (So here the central vertex around which 
a neighbourhood is built has fitness $x_J$.) We call this the \textit{offspring process} of vertex $v$.
Furthermore, we call the vector $S_\ell$ the \textit{offspring profile} of $v$ after $\ell$ steps. 

For any $J\in [L]$ and $S \in \N_0^{[L]}$ we define 
\begin{equation} \label{def:weight_s}
Z^\star ((J,S)) := \sum_{i\in [L]} a_{i,J} S (i). 
\end{equation}
We let $Z_\ell^\star$ be the total activity of the balls after $\ell$ steps: 
$Z_\ell^{\star} := Z^\star((J,S_\ell))$. If the process has root of fitness $W$ 
starts with a single ball/directed edge whose other vertex has fitness $W_1$, then 
$Z_\ell^\star$ is distributed as $h_0(W)\sum_{s=1}^\ell h_1 (W_s)$, where $(W_s)_{s \geq 1}$ 
are i.i.d. random variables with common law $\nu$. 
The following is a consequence of the strong law of large numbers.
\begin{prop} \label{prop:as_conv_nbd} 
For all $J \in [L]$ and $S \in \N_0^{[L]}$, $\mathbb{P}^\star_{(J,S)}$-a.s.
\begin{equation} \label{eq:star_process_conc} 
\lim_{\ell \to \infty} \frac{1}{\ell} Z_\ell^{\star} = \E[\nu]{h_0} \E[\nu]{h_1}.
\end{equation}
\end{prop}

The limiting value of $p_k(n)$ is given in terms of the above Markov chain. 
Now, if $\Xi$ is a Borel measure on $[0,1]^2$, let $V(\Xi ) = J$ where $(I,J)\in [0,1]^2$ is random having law $\Xi$.
Also, for $j\in[L]$, we let $\delta_j \in \N_0^{[L]}$ be the unit vector with $\delta_j(i) =1$ if and only if $i=j$. 
\begin{thm} \label{thm:degree_k}
For all $k\geq 0$, we have, with convergence in probability,
\[\lim_{n\to \infty} \frac{1}{n} N_{k} (n) = 
\mathbb{E}_{(J_W, \delta_{V(U^{(\nu)})})}^\star \left( \frac{\lambda}{Z^\star_k+\eig{\nu}}\prod_{\ell=0}^{k-1} \frac{Z^\star_\ell}{Z_\ell^\star+\eig{\nu}} \right) =: p_k\]
with $\eig{\nu}$ is the leading eigenvalue of the replacement matrix $M$. 
\end{thm}
The proof of the above theorem follows the technique introduced in~\cite{FIMS2019}. 

Firstly, note that it is sufficient to consider only vertices that arrive after  step  $\eta n$, where 
$\eta > 0$ is a constant which we will send to 0 eventually.  This will allow us to control the distribution 
of the fitness of any vertex arriving after time $\eta n$ and whose degree we want to keep track of. 
This is the case since when we sample a directed edge, the distribution of its colour will be close to 
$U^{(\nu)}$ provided that $n$ is sufficiently large. 

If $N_k^\eta(n)$ denotes the number of vertices with $k$ offspring in $\T_n$, 
which arrived after time $\eta n$, we have
\[N_k^{\eta}(n) \leq N_k(n) \leq \eta n + N_k^{\eta}(n).\]
Therefore, 
\begin{equation*}  
\lim_{\eta \to 0}\limsup_{n \to \infty} \frac{1}{n} \left| \E[\T_0]{N_k (n)}-  \E[\T_0]{N_k^{\eta}(n)} \right| =0.
\end{equation*}

The rest of this section is thus devoted to proving that, for all $k\geq 0$,
\[\lim_{\eta \to 0} \lim_{n \to \infty} \frac{1}{n}  \E[\T_0]{ N_k^\eta(n)}
= \p_k.\]

\subsection{Asymptotics on the expected value of $N^{\eta}_k(n)$} 

Let $\tau_i$ denote the (stopping) time of the $i$th split. 
We will use a limit theorem about the sequence $(\tau_n )_{n\in \N_0}$ proved by Athreya and Karlin
(Limit theorems for the split times of branching processes, Journal of Mathematics and Mechanics 17 (1967), 257--277). According to Theorem 5 in~\cite{ar:AthKar67} there exists a random variable $W$ such that 
with $\mu=\eig{\nu}/\left(\sum_{i,j\in[L]} a_{i,j} \U_{i,j}\right)$,
 $\mathbb{P}_{\T_0}$-almost surely 
\begin{equation} \label{eq:lim_stopping_times} \lim_{n\to \infty} n \mu e^{-\eig{\nu} \tau_n} = W. \end{equation}
Moreover, $W>0$ $\mathbb{P}_{\T_0}$-almost surely. 

In fact, we will derive asymptotics on $N^{\eta}_{\leq k} (n)$, which counts the number of vertices $i$ with 
$\eta n \leq i \leq n$ having $\dg{i}{n}\leq k$:
$$N^{\eta}_{\leq k} (n) = \sum_{\eta n \leq i \leq n} \ind{\dg{i}{n}\leq k}.$$
We will essentially work on the event of~\eqref{eq:lim_stopping_times}. The next claim follows immediately 
from this. 
\begin{clm} For any $\delta$ there exists $j_0$ such that 
$$\Prob[\T_0] {\left| \tau_j - \frac{1}{\eig{\nu}} \log j - \log (W/\mu) \right| < \delta, \ \forall j\geq j_0} > 1-\delta. $$
\end{clm}
Let $\cE_{j_0,\delta}$ denote the above event.
Then provided that $n > j_0/ \eta$ we have 
\begin{eqnarray} 
\E[\T_0]{N^{\eta}_{\leq k} (n)} &\geq& \E[\T_0]{N^{\eta}_{\leq k} (n) \cdot \ind{\cE_{j_0,\delta}}} \nonumber \\
&=& \sum_{\eta n \leq i \leq n}  \E[\T_0]{\ind{\dg{i}{\tau_n }\leq k} \cdot \ind{\cE_{j_0,\delta}}}  \geq 
 \sum_{\eta n \leq i \leq n}  \E[\T_0]{\ind{\dg{i}{\frac{1}{\eig{\nu}} \log (n/i) + 2\delta} \leq k} \cdot \ind{\cE_{j_0,\delta}}}  \nonumber \\
 &\geq & \sum_{\eta n \leq i \leq n}  \E[\T_0]{\ind{\dg{i}{\frac{1}{\eig{\nu}} \log (n/i) + 2\delta} \leq k}}   
- \sum_{\eta n \leq i \leq n}  \E[\T_0]{1-\ind{\cE_{j_0,\delta}}} \nonumber \\
&\geq & \sum_{\eta n \leq i \leq n}  \E[\T_0]{\ind{\dg{i}{\frac{1}{\eig{\nu}} \log (n/i) + 2\delta} \leq k}}   - \delta n \nonumber \\
&=& \sum_{\eta n \leq i \leq n} \E[\T_0]{  \mathbb{E}( \ind{\dg{i}{\frac{1}{\eig{\nu}} \log (n/i) + 2\delta} \leq k} 
| \mathscr {G}_{\tau_i} )}   - \delta n. \label{eq:low_inter}
\end{eqnarray}
Let $t_{n,i}^+ :=\frac{1}{\eig{\nu}} \log (n/i) + 2\delta$. 
Furthermore, if the ancestor of vertex $i$ has fitness equal to $x_j$ we set $a(i) = j$.
Consider the offspring process starting at $(J_{W_i}, \delta_{a(i)})$ embedded in continuous times. 
Here, a ball of colour $(i,j)$ has Poisson point process associated with it of rate $a_{i,j}$ which is independent 
of everything else.  
We denote by $Z_\ell^{\star\star}$ is the total activity of the balls after the $\ell$th split. 
If $T_\ell$ denotes the time of the $\ell$th split for $\ell \in \N$ and with $T_0=0$,
then for any $\ell \in \N$ the interval $\Delta T_\ell := T_\ell - T_{\ell-1}$ is distributed as $\Exp{Z_\ell^{\star\star}}$, 
an exponentially distributed random variable with parameter $Z_\ell^{\star \star}$.  
We denote by $\mathbb{P}_{(J_{W_i}, \delta_{a(i)})}^{\star \star}$ and 
$\mathbb{E}_{(J_{W_i}, \delta_{a(i)})}^{\star\star}$ the associated probability measure and the 
expectation operator.  
With this notation
$$  \mathbb{E}( \ind{\dg{i}{t_{n,i}^+} \leq k} 
| \mathscr {G}_{\tau_i} ) = \mathbb{E}^{\star\star}_{(J_{W_i}, \delta_{a(i)})} (\ind{\sum_{\ell=0}^k \Delta T_{\ell} 
\geq t_{n,i}^+}).$$
Note that $Z_\ell^{\star \star}$ is distributed as $Z_\ell^{\star}$ in the discrete offspring process 
starting at $(J_{W_i}, \delta_{a(i)})$. So, if $\mathbb{P}_{z_0,\ldots, z_k}$ denotes the probability space of 
$k+1$ independent random variables $Y_\ell \sim \Exp{z_\ell}$, $\ell =0, \ldots, k$, we then have 
\begin{equation} \label{eq:disc-cts}
\mathbb{E}^{\star\star}_{(J_{W_i}, \delta_{a(i)})} (\ind{\sum_{\ell=0}^k \Delta T_{\ell} 
\geq t_{n,i}^+}) = \mathbb{E}^{\star}_{(J_{W_i}, \delta_{a(i)})} 
\left( \mathbb{P}_{Z_0^{\star}, \ldots, Z_k^{\star}} \left( \sum_{\ell=0}^k Y_\ell \geq t_{n,i}^+\right)  \right). 
\end{equation}

Now, observe that for each $j \in \N$, $a(j) = V(\Gamma_i^{(\nu)})$.  But by Proposition~\ref{prop:measure_conv}, 
$\mathbb{P}_{\T_0}$-a.s. $\Gamma_j^{(\nu)} \to U^{(\nu)}$ weakly as $j\to \infty$. 
So the law of $a(j)$ converges to the law of $V(U^{(\nu)})$ as $n\to \infty$. 
Furthermore, the law of $J_{W_j}$ equals the law of $J_W$, where $W$ is a random variable with law $\nu$ that 
is independent of $V(U^{(\nu)})$.  
Hence, $\mathbb{P}_{\T_0}$-a.s. $(J_{W_j}, \delta_{a(j)})$ converges weakly to $(J_W, \delta_{V(U^{(\nu)})})$ 
as $j\to \infty$. 
Thus, $\mathbb{P}_{\T_0}$-a.s. 
$$ \lim_{j\to\infty}\mathbb{E}^{\star}_{(J_{W_j}, \delta_{a(j)})} 
\left( \mathbb{P}_{Z_0^{\star}, \ldots, Z_k^{\star}} \left( \sum_{\ell=0}^k Y_\ell \geq t_{n,i}^+\right)  \right) 
= \mathbb{E}^{\star}_{(J_{W}, \delta_{V(U^{\nu})})} 
\left( \mathbb{P}_{Z_0^{\star}, \ldots, Z_k^{\star}} \left( \sum_{\ell=0}^k Y_\ell \geq t_{n,i}^+\right)  \right).
$$
This implies that for any $\delta$, if $n$ is sufficiently large, then for any $\eta n \leq i \leq n$ we have 
\begin{eqnarray*}
\lefteqn{\E[\T_0]{ \mathbb{E}^{\star}_{(J_{W_i}, \delta_{a(i)})} 
\left( \mathbb{P}_{Z_0^{\star}, \ldots, Z_k^{\star}} \left( \sum_{\ell=0}^k Y_\ell \geq t_{n,i}^+\right)  \right) } =} \\
& &\hspace{3cm} \mathbb{E} \left( \mathbb{E}^{\star}_{(J_{W}, \delta_{V(U^{\nu})})} 
\left( \mathbb{P}_{Z_0^{\star}, \ldots, Z_k^{\star}} \left( \sum_{\ell=0}^k Y_\ell \geq t_{n,i}^+\right)  \right)\right) 
\pm \delta. 
\end{eqnarray*}
Using the lower bound in~\eqref{eq:low_inter}, we get 
\begin{equation}  
\E[\T_0]{N^{\eta}_{\leq k} (n)} \geq \mathbb{E} \left( \mathbb{E}^{\star}_{(J_{W}, \delta_{V(U^{\nu})})} 
\left(
\sum_{\eta n \leq i \leq n}
 \mathbb{P}_{Z_0^{\star}, \ldots, Z_k^{\star}} \left( \sum_{\ell=0}^k Y_\ell \geq t_{n,i}^+\right)  \right)\right) 
 - 2\delta n.
\end{equation} 
Now, we will give asymptotics for the inner sum that is uniform over all possible values of the sequence 
$Z_\ell^{\star}$, $\ell = 0, \ldots, k$. Fix an increasing sequence $0<z_0 < \cdots < z_k$ such that $z_k \leq k \bar{h_1}\bar{h_0}$ and $z_0 > \underline{h_0} \underline{h_1}$. We write 
\begin{eqnarray}
\sum_{\eta n \leq i \leq n}
\mathbb{P}_{z_0, \ldots, z_k} \left( \sum_{\ell=0}^k Y_\ell \geq t_{n,i}^+\right)   = n \cdot 
 \sum_{\eta n \leq i \leq n} \frac{1}{n} \cdot \mathbb{P}_{z_0, \ldots, z_k} \left( \sum_{\ell=0}^k Y_\ell \geq t_{n,i}^+\right). \nonumber 
\end{eqnarray}
But with $x = i/n$ (this should probably go...),
\begin{equation} \label{eq:sum_inter}
\lim_{n\to \infty}\sum_{\eta n \leq i \leq n} \frac{1}{n} \cdot \mathbb{P}_{z_0, \ldots, z_k} \left( \sum_{\ell=0}^k Y_\ell \geq t_{n,i}^+\right) = \int_{\eta}^1  \mathbb{P}_{z_0, \ldots, z_k} 
\left( \sum_{\ell=0}^k Y_\ell \geq \frac{1}{\eig{\nu}} \log \left( \frac{1}{x} \right) + 2\delta \right)
 \dd x .
 \end{equation}
 We now perform a change of variable and set $y= \frac{1}{\eig{\nu}} \log \left(\frac{1}{x}\right)$.
 Thus, $\dd y = -\frac{1}{\eig{\nu}} \cdot \frac{1}{x} \dd x$, whereby $\dd x = -\eig{\nu} x \dd y =- \eig{\nu} 
 e^{-\eig{\nu} y} \dd y$. 
 Also, set $F_{z_0,\ldots, z_k} (y) = \mathbb{P}_{z_0, \ldots, z_k} \left( \sum_{\ell=0}^k Y_\ell \geq y \right)$.
 The above integral thus becomes
 \begin{eqnarray}
& &\int_{\eta}^1  F_{z_0, \ldots, z_k} 
\left( \frac{1}{\eig{\nu}} \log \left( \frac{1}{x} \right) + 2\delta \right)
 \dd x  =- \int_{\frac{1}{\eig{\nu}} \log \left(\frac{1}{\eta}\right) }^{0}  F_{z_0, \ldots, z_k} 
\left( y + 2\delta \right) \left( \eig{\nu} 
 e^{-\eig{\nu} y}  \right) \dd y \nonumber\\
 &=& \int_0^{\frac{1}{\eig{\nu}} \log \left(\frac{1}{\eta}\right) }  F_{z_0, \ldots, z_k} 
\left( y + 2\delta \right) \left(\eig{\nu} 
 e^{-\eig{\nu} y} \right) \dd y \nonumber \\
 &=&  \int_0^{\frac{1}{\eig{\nu}} \log \left(\frac{1}{\eta}\right) }  F_{z_0, \ldots, z_k} 
\left( y + 2\delta \right) \left(
 -e^{-\eig{\nu} y} \right)^\prime \dd y. \label{eq:inter_integral}
\end{eqnarray}
We will perform integration by parts to the above integral. 
If $f_{z_0,\ldots, z_k} (y)$ is the pdf of $\sum_{\ell=0}^k \Exp{z_\ell}$, then note that 
$F_{z_0, \ldots, z_k}^\prime \left( y + 2\delta \right) = - f_{z_0,\ldots, z_k} (y +2\delta)$. 
 Thus, the last integral in~\eqref{eq:inter_integral} becomes
 \begin{eqnarray}
\lefteqn{\int_0^{\frac{1}{\eig{\nu}} \log \left(\frac{1}{\eta}\right) }  F_{z_0, \ldots, z_k} 
\left( y + 2\delta \right) \left(
 e^{-\eig{\nu} y} \right)^\prime \dd y =} \nonumber \\ 
& & \left[ -f_{z_0,\ldots, z_k} (y +2\delta) e^{-\eig{\nu} y}\right]_0^{\frac{1}{\eig{\nu}} \log \left(\frac{1}{\eta}\right)}
- \int_0^{\frac{1}{\eig{\nu}} \log \left(\frac{1}{\eta}\right)}  f_{z_0,\ldots, z_k} (y +2\delta) e^{-\eig{\nu} y }\dd y.
\end{eqnarray}
Now, we take limits of the above as $\eta \downarrow 0$ and $\delta \downarrow 0$. Note that 
$f_{z_0,\ldots, z_k} (y) = 0$, for $y\leq 0$ and moreover it is bounded and continuous on $\R$. 
Hence, 
$$\lim_{\delta \downarrow 0} \lim_{\eta \downarrow 0} \left[ -f_{z_0,\ldots, z_k} (y +2\delta) e^{-\eig{\nu} y}\right]_0^{\frac{1}{\eig{\nu}} \log \left(\frac{1}{\eta}\right)} =1.$$
Moreover, 
\begin{eqnarray}
\lim_{\delta \downarrow 0} \lim_{\eta \downarrow 0} \int_0^{\frac{1}{\eig{\nu}} \log \left(\frac{1}{\eta}\right)}  f_{z_0,\ldots, z_k} (y +2\delta) e^{-\eig{\nu} y }\dd y& =& 
\lim_{\delta \downarrow 0}
\int_0^{\infty}  f_{z_0,\ldots, z_k} (y +2\delta) e^{-\eig{\nu} y }\dd y  \nonumber \\
&=&\int_0^\infty f_{z_0,\ldots, z_k} (y) e^{-\eig{\nu} y }\dd y. \label{eq:Laplace_int}
\end{eqnarray}
where the last equality follows from the dominated convergence theorem. 
Now, this last integral is the Laplace transform of $\sum_{\ell=0}^k \Exp{z_\ell}$ evaluated at $\eig{\nu}$:
$$ \int_0^\infty f_{z_0,\ldots, z_k} (y) e^{-\eig{\nu} y }\dd y =  \prod_{\ell=0}^k \frac{z_\ell}{z_\ell  + \eig{\nu}}.$$
We thus conclude that 
\begin{eqnarray}\lim_{\delta \downarrow 0} \lim_{\eta \downarrow 0} 
\lim_{n\to \infty}\sum_{\eta n \leq i \leq n} \frac{1}{n} \cdot \mathbb{P}_{z_0, \ldots, z_k} \left( \sum_{\ell=0}^k Y_\ell \geq t_{n,i}^+\right) = 1 - \prod_{\ell=0}^k \frac{z_\ell}{z_\ell  + \eig{\nu}}.
\end{eqnarray}
Hence, by the dominated convergence theorem we finally get:
\begin{eqnarray}
& &\lim_{\eta \downarrow 0} 
\lim_{n\to \infty} \frac{1}{n} \E[\T_0]{N_{\leq k}^{\eta} (n)} \nonumber \\
&\geq&
\lim_{\delta \downarrow 0} \lim_{\eta \downarrow 0} 
\lim_{n\to \infty} \frac{1}{n}\left( \mathbb{E} \left( \mathbb{E}^{\star}_{(J_{W}, \delta_{V(U^{\nu})})} 
\left( \sum_{\eta n \leq i \leq n}
 \mathbb{P}_{Z_0^{\star}, \ldots, Z_k^{\star}} \left( \sum_{\ell=0}^k Y_\ell \geq t_{n,i}^+\right)  \right)\right) 
 - 2\delta n \right) \nonumber \\
  &\geq&
  \mathbb{E} \left( \mathbb{E}^{\star}_{(J_{W}, \delta_{V(U^{\nu})})} 
\left( \lim_{\delta \downarrow 0} \lim_{\eta \downarrow 0} 
\lim_{n\to \infty} \frac{1}{n} \left(\sum_{\eta n \leq i \leq n}
 \mathbb{P}_{Z_0^{\star}, \ldots, Z_k^{\star}} \left( \sum_{\ell=0}^k Y_\ell \geq t_{n,i}^+\right)  \right)
 - 2\delta n \right)\right) \nonumber \\
&=&   \mathbb{E} \left( \mathbb{E}^{\star}_{(J_{W}, \delta_{V(U^{\nu})})} 
\left( 1 - \prod_{\ell=0}^k \frac{Z_\ell^\star}{Z_\ell^{\star}  + \eig{\nu}}\right) \right).
\end{eqnarray}

For the upper bound the argument is similar and we will only sketch it. 
We let $t_{n,i}^- = \frac{1}{\eig{\nu}} \log (n/i) - 2\delta$. 
\begin{eqnarray} 
\E[\T_0]{N^{\eta}_{\leq k} (n)} &\leq& \E[\T_0]{N^{\eta}_{\leq k} (n) \cdot \ind{\cE_{j_0,\delta}}} + 
n \E[\T_0]{\ind{\cE_{j_0,\delta}^c}} 
\nonumber \\
&\leq & \sum_{\eta n \leq i \leq n}  \E[\T_0]{\ind{\dg{i}{\tau_n }\leq k} \cdot \ind{\cE_{j_0,\delta}}} +\delta n 
\leq 
 \sum_{\eta n \leq i \leq n}  \E[\T_0]{\ind{\dg{i}{t_{n,i}^-} \leq k} \cdot \ind{\cE_{j_0,\delta}}}  +\delta n \nonumber \\
 &\leq & \sum_{\eta n \leq i \leq n}  \E[\T_0]{\ind{t_{n,i}^- \leq k}}   - \delta n \nonumber \\
&=& \sum_{\eta n \leq i \leq n} \E[\T_0]{  \mathbb{E}( \ind{\dg{i}{t_{n,i}^-} \leq k} 
| \mathscr {G}_{\tau_i} )}   + \delta n. \label{eq:upper_inter}
\end{eqnarray}
The sum is estimated as in~\eqref{eq:sum_inter}; the only difference is that $t_{n,i}^+$ has been 
replaced by $t_{n,i}^-$.
We can thus conclude that 
$$ \lim_{\eta \downarrow 0} 
\lim_{n\to \infty} \frac{1}{n} \E[\T_0]{N_{\leq k}^{\eta} (n)} \leq 
\mathbb{E} \left( \mathbb{E}^{\star}_{(J_{W}, \delta_{V(U^{\nu})})} \left( 1 - \prod_{\ell=0}^k \frac{Z_\ell^\star}
{Z_\ell^\star  + \eig{\nu}}\right) \right).
 $$
Turning back to $N_k^{\eta} (n)$, note that for $k\geq 1$  we have 
$N_k^\eta (n) = N_{\leq k}^\eta (n) - N_{\leq k-1}^\eta (n)$.  
Hence, we deduce that  
\begin{eqnarray*}
\lim_{\eta \downarrow 0} \lim_{n\to \infty} \frac{1}{n} \E[\T_0]{N_{k}^{\eta} (n)} &=& 
\mathbb{E} \left( \mathbb{E}^{\star}_{(J_{W}, \delta_{V(U^{\nu})})}  \left(
\prod_{\ell=0}^{k-1} \frac{Z_\ell^\star}{Z_\ell^\star  + \eig{\nu}} - \prod_{\ell=0}^{k} \frac{Z_\ell^\star}{Z_\ell^\star  + \eig{\nu}} \right)\right) \\ 
&=&
\mathbb{E} \left( \mathbb{E}^{\star}_{(J_{W}, \delta_{V(U^{\nu})})} \left( 
\frac{\eig{\nu}}{Z_k^\star + \eig{\nu}}\prod_{\ell=0}^{k-1} \frac{Z_\ell^\star}{Z_\ell^\star  + \eig{\nu}} \right)\right).
\end{eqnarray*}
Also, $N_0^{\eta}(n) = N_{\leq 0}^{\eta}(n)$ and therefore
$$ \lim_{\eta \downarrow 0} \lim_{n\to \infty} \frac{1}{n} \E[\T_0]{N_{0}^{\eta} (n)} = \mathbb{E} \left( \mathbb{E}^{\star}_{(J_{W}, \delta_{V(U^{\nu})})} \left( \frac{\eig{\nu}}{Z_0^\star + \eig{\nu}} \right) \right).$$
}

\section{Introduction} 
\subsection{Background}
Complex networks appearing in areas as diverse as the internet, social networks and telecommunications are well known for their ubiquitous, non-trivial properties; in particular, they often have a scale free (power law) degree distribution, and display a small or ultra-small world phenomenon (having diameter of logarithmic or double logarithmic order with respect to the size of the network). In their seminal paper, Albert and Barab\'{a}si in \cite{Barabasi509} (later studied rigorously in \cite{bollobaspreferential,mori-trees-2002}) observed that these properties emerged naturally in a model where vertices arrive one at a time, and display a ``preference'' to popular vertices - more precisely, connect to existing vertices with probability proportional to their degree. In the case where the newly arriving vertex connects to a single existing vertex, this gives rise to a well-known model of random trees that has been studied under various names: first under the name \textit{ordered recursive tree} by Prodinger and Urbanek in \cite{prodingerurbanek}, \textit{nonuniform recursive trees} by Szyma\'{n}ski in \cite{szymanski}, random \textit{plane oriented recursive trees} in \cite{mahmoud92,mahmoudetal93}, random \textit{heap ordered recursive trees} \cite{chen1994} and \textit{scale-free trees} \cite{bollobaspreferential,szaboalava2002,bollobasriordan04}. Various other modifications of this model have also been studied, including the case that vertices are chosen according to a \textit{super-linear} function of their degree in \cite{Oliveira-spencer}, or indeed any positive function of the degree \cite{rudas} (assuming a certain condition is satisfied). In \cite{holmgren-janson}, the latter model is generalised to arbitrary non-negative functions of the degree and is referred to as \textit{generalised preferential attachment}. 
\\\\
Whilst the preferential attachment model is successful in reproducing the properties of complex networks, it is generally the earlier arriving vertices that are more likely to have higher degrees, since (informally) they have more time to acquire new neighbours, which in turn reinforces the growth of their degree. (Indeed, a result of \cite{dereich-morters-sublinear} shows that, from a  certain time point onward, the vertex with maximal degree remains fixed in this model.) In contrast, in real world models it is often newly arriving nodes that quickly acquire a large number of links (for example, in the world wide web). Motivated by this, in \cite{bianconibarabasi2001}, Bianconi and Barab\'{a}si introduced their well-known  model (also called \textit{preferential attachment with multiplicative fitness}). There, vertices arrive one at a time, and, upon arrival, each vertex is equipped with a random weight sampled independently from a fixed distribution. At each time-step, the newly arriving vertex $u$ connects to an existing vertex $v$ with probability proportional to the product of the weight of $v$ and its degree. Thus, the random weight may be interpreted as a measure of the intrinsic ``attractiveness'' of a vertex. Bianconi and Barab\'{a}si postulated the emergence of an interesting dichotomy in this model which they called \textit{Bose-Einstein condensation} (motivated by similar phenomena in statistical physics): under a certain critical condition on the weight distribution, a positive proportion of all the edges in tree accumulate around vertices of maximum weight. This dichotomy was first proved rigorously by Borgs et al. in  \cite{borgs-chayes} in the case that the weight distribution is supported on an interval, and absolutely continuous with respect to Lebesgue measure (however, they note that other classes of weight distribution are possible). They also showed that in this model, the degree distribution of vertices with a given weight follows a power law, with exponent depending on the weight of the vertex. A similar condensation phenomenon was observed in a variant of this model by Dereich in \cite{dereich-unfolding}, and later, in a more general, robust setting (in the sense that the results apply to wide variety of model specifications) in \cite{dereich-ortgiese}. 
\\\\
Two other similar models are the \textit{preferential attachment with additive fitness} introduced by Erg\"{u}n and Rodgers in \cite{ergun}, where newly arriving vertices now connect to existing vertices with probability proportional to the sum of their weight and degree, and the \textit{weighted recursive tree} introduced in \cite{wrt1}. In \cite{wrt2delphin}, S\'{e}nizergues showed that the preferential attachment with additive fitness (with deterministic weights) is equal in distribution to a particular weighted random recursive tree with random weights. In addition,  Lodewijks and Ortgiese~in \cite{bas,bas2} uncovered an interesting dichotomy in the maximal degrees of these models, in a robust, evolving graph setting. In \cite{asympt-gen}, the second author studied a model incorporating the weighted recursive tree as well as  preferential attachment trees with both additive and multiplicative fitness: here at each time-step vertex with weight $w$ and degree $k$ is chosen with probability proportional to $g(w)(k-1) + h(w)$, where $g,h$ are non-negative, measurable functions. In this case, the dynamics of the model depend on $h$ in an non-trivial way: under a certain critical condition on the weight distribution, $g$ and $h$ condensation occurs, but does not occur if $h$ takes large enough values on certain parts of its domain. 
\\\\
In the case of evolving trees, many of the above models describe the family tree of associated continuous time branching processes (often \textit{Crump-Mode-Jagers} or multitype branching processes), and this perspective has offered some interesting insights into the evolution of these models. For example, the preferential attachment tree of Albert and Barab\'{a}si was actually first described in the context of evolution by Yule in \cite{Yule} and in the context of language by Simon in \cite{Simon}. In addition, the \textit{condensation} phenomenon observed by Bianconi and Barab\'{a}si was first studied in a similar, yet simpler manner, in the context of evolution by Kingman in \cite{kingman-1978}. Later, the results of \cite{Oliveira-spencer, rudas, holmgren-janson, athreya-arka-sethuraman-2008} have all exploited the connection to branching processes to derive results related to more general preferential attachment models, and in \cite{asympt-gen,bhamidi} in relation to inhomogeneous models with a `fitness' component. Often, the associated branching process with the discrete time model is known as the \textit{continuous time embedding}, or \textit{Arthreya-Karlin embedding}, based on pioneering work by Arthreya and Karlin in \cite{at_embedding} who applied this approach in the context of P\'{o}lya urns. As shown in \cite{holmgren-janson, bhamidi}, when studying `local' properties such as degrees of vertices, one can observe that the continuous time embedding is a Crump-Mode-Jagers branching process, and apply the results of \cite{nerman_81}, whilst when studying properties such as the height (which is the same order of magnitude as the diameter), one can apply the results of \cite{kingman1975} and an argument of Pittel \cite{pittel}.
\\\\
In \cite{dereich-mailler-morters}, the authors studied condensation in models of reinforced branching processes that generalise the continuous time embedding of the Bianconi-Barab\'{a}si model, showing that the condensation is \textit{non-extensive}: whilst a positive proportion of edges in the family tree of the process accumulate around vertices of maximal weight, the maximal degree of the tree remains sub-linear. In addition, in \cite{vdh-aging-mult-fitness-2017}, the authors studied another generalisation of the continuous time embedding of the Bianconi-Barab\'{a}si model, incorporating `aging' effects, and applying this to the study of citation networks; they demonstrated a dichotomy between degree distributions having power law and exponential tails based on the aging parameter. 
\\\\
There are a number of other interesting variations of inhomogeneous preferential attachment models. In  \cite{jonathan-co-existing},  Jordan studies a model of preferential attachment where vertices belong to two types, and new vertices connect to one according to an additive fitness mechanism, and the other via a multiplicative fitness. Geometric models have also been considered in \cite{ jonathan-geometric-preferential-attachment}: here, new vertices are equipped with a location in a metric space, and connect to existing vertices with probability proportional to the product of their degree, and a positive function (called an attractiveness function) of the distance between them. In \cite{ jonathan-geometric-preferential-attachment}, the authors demonstrate a dichotomy (depending on the attractiveness function) between behaviour according to the model of Albert and Barab\'{a}si, and a well known geometric model known as the on line nearest neighbour model. 
\\\\
Inhomogeneous models have also been studied in the context of models with \textit{choice} in \cite{jonathan-freeman-extensive-cond,jonathan-haslegrave-location-based}, with the appearance of more fascinating condensation phenomena. In this model vertices are equipped with weights, at each time step $r$ vertices are chosen with probability proportional to their degree, and out of these $r$ vertices, a random vertex is chosen as the neighbour of the new-coming vertex (where the probability distribution may depend on the weights of the vertices). In \cite{jonathan-freeman-extensive-cond}, the authors showed that, in the case that the maximal weight vertex is chosen, \textit{extensive condensation} may occur, that is, under a critical condition on the weight distribution, a positive proportion of edges accumulate around the vertex of maximal degree. In addition, in \cite{jonathan-haslegrave-location-based}, the authors showed that in certain cases, with random choice rules, the distribution of edges with endpoint having certain weight converges weakly to a \textit{random measure} where \textit{multiple condensation} can occur with positive probability (that is, positive proportions of edges accumulate around vertices of multiple weights). In addition, they showed that multiple condensation cannot occur when deterministic choice rules are used, and there exist phase transitions for condensation occurring with probability $0$ or $1$.

\subsection{Preferential Attachment Trees with Neighbourhood Influence}

As we discussed above, a number of preferential attachment mechanisms which incorporate inhomogeneity have been considered. However, models where the attachment mechanism depends on the \textit{weights} of the \textit{neighbours} of a vertex have received far less attention. 
In this direction, the authors in \cite{dynamical-simplices} recently incorporated \textit{higher-dimensional interactions} into this notion of preferential attachment, studying a model of evolving simplicial complexes. They proved convergence in probability of the limiting degree distribution to a limiting value, depending on a \textit{companion Markov process} that tracks the evolution of the neighbourhood of a given vertex. In this paper, we study a simplified version of that model, which involves evolving \textit{trees}; as a result, we are able to derive stronger statements. 
\\\\
More precisely, we consider a model of \textit{weighted directed trees} $(\cT_{n})_{n \in \mathbb{N}_0}$; these are labelled directed trees, where vertices have real valued weights associated to them. Let $\mathbb{T}$ denote the set of all such weighted trees, and given a tree $\cT \in \mathbb{T}$ and a vertex $j \in \cT$, let $N^{+}(j, \cT)$ be the weighted tree consisting of $j$ and all of its \textit{out-neighbours}. 
In order to define the model, we will require a probability measure $\mu$, which, without loss of generality is supported on a subset of an interval $[0,\wmax]$, for some $\wmax > 0$ and a \textit{fitness function} $f: \mathbb{T} \rightarrow \mathbb{R}_{+}$.  
\\\\
In the model we consider, we start with an initial tree $\cT_{0}$ consisting of a single vertex with random weight $W_0$ sampled from $\mu$.
Then, given $\cT_i$, the model proceeds recursively as follows:
 \begin{enumerate}[(i)]
 \item Sample a vertex $j$ from $\cT_i$ with probability $\frac{f(N^+(j, \cT_i))}{\cZ_i},$
  where $\cZ_i : = \sum_{k=0}^{i}f(N^+(k,\cT_i))$ is the \textit{partition function} associated with the process. \item Form $\cT_{i+1}$ by adding the edge $(j, i+1)$, and assigning vertex $i+1$ weight $W_{i+1}$ sampled independently from $\mu$.
 \end{enumerate}

In this paper, we define $f$ so that 
\begin{equation} \label{eq:fitness-function}
f(N^+(v,T)) = h(W_{v}) + \sum_{(v,u)\in E(T)} g(W_{v}, W_{u}),
\end{equation}
where $h: [0,\wmax] \rightarrow [0,\infty)$ and $g: [0,\wmax] \times [0,\wmax] \rightarrow [0, \infty)$ are bounded and measurable. To ensure that the evolution of the model is well-defined, in all of our results we condition on $W_0$ satisfying $h(W_0) > 0$ (which we assume is an event that has positive probability).
\begin{figure}[H] 
\captionsetup{width=.8\linewidth}
\centering
\begin{center} 
 \begin{tikzpicture}[scale=0.9]
\ImageNode[label={0:}]{A}{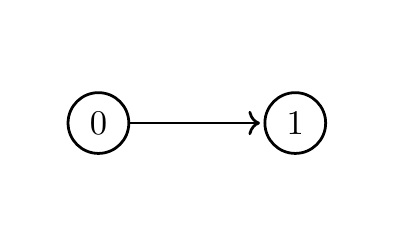}
\ImageNode[label={180:},right=of A]{B}{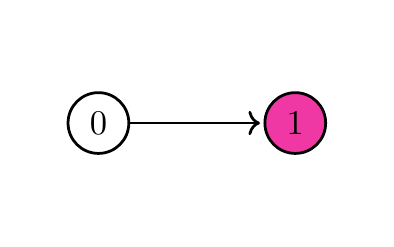}
\ImageNode[label={180:},right=of B]{C}{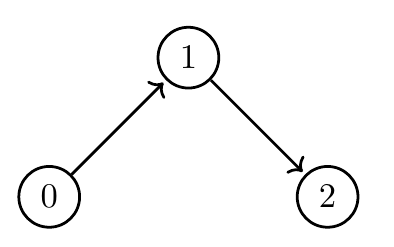}
\draw[-{Latex[length=5mm, width = 5mm]},
    arrowblue,
    line width=5pt
  ]
  (A) -- (B);
  \draw[-{Latex[length=5mm, width = 5mm]},
    arrowblue,
    line width=5pt
  ]
  (B) -- (C);
 \end{tikzpicture}
\end{center}
\centering
\caption{A sample transition from $\cT_{1}$ to $\cT_{2}$. In $\cT_1$, $0$ is chosen with probability proportional to $f(N^{+}(0,\cT_1)) = h(W_0) + g(W_0,W_1)$, while $1$ is chosen with probability proportional to $f(N^{+}(1,\cT_1)) = h(W_1)$. In this evolution, $1$ is chosen, so the newcomer $2$ arrives as an out-neighbour of $1$.}
\label{fig:tree}
\end{figure}

\begin{rmq} \label{rem:fitness-form}
The form of the fitness function in Equation~\eqref{eq:fitness-function} is sufficiently general to encompass some existing models. In the case where $g$ and $h$ are a single constant, we obtain the classic preferential attachment tree of Albert and Barab\'{a}si. The case $g(x,y) = h(x) = x$ is the Bianconi-Barab\'{a}si model, whilst the case $g(x,y) \equiv 1, h(x) = x$ is the preferential attachment tree with additive fitness. Finally, the case $g(x,y) = g'(x)$, for some bounded measurable function of a single variable is the generalised preferential attachment with fitness model studied by the second author in~\cite{asympt-gen}.
\end{rmq}

\begin{rmq} \label{rmq:branching}
One may interpret $(\cT_{n})_{n \in \mathbb{N}_0}$ in the context of reinforced branching processes as follows: we begin with an \emph{individual} $0$ belonging to its own  \emph{family} that reproduces after an exponentially distributed amount of time, with parameter $h(W_0)$. We say that the \emph{ancestral weight} of the family is $W_0$. Then, recursively, when a birth event occurs in the $i$th family, with ancestral weight $W_i$, a new individual with random weight $W$ joins the $i$th family, reproducing after an $\text{Exp}(g(W_i,W))$ amount of time; and simultaneously, an individual of weight $W$ begins its own family, with ancestral weight $W$. The out-neighbourhood of a vertex $i$ in the tree $\cT_{n}$ (including the vertex $i$ itself) then represents individuals in the $i$th family in the branching process, at the time of the $n$th birth event.
\end{rmq}
\begin{rmq} \label{rmq:branching 2}
One can extend the model from the previous remark further by supplanting it with constants $0 \leq \beta, \gamma \leq 1$, so that when a birth event occurs, independently with probability $\beta$, an individual with random weight $W$ joins the $i$th family, and with probability $\gamma$, an individual with random weight $W'$ (also sampled from $\mu)$ initiates its own family with ancestral weight $W'$. While not immediately clear from the way we have defined the model, our methods also extend to this case - this link becomes clearer when viewing individuals as ``loops'' and ``edges'' in a P\'{o}lya urn similar to Urn~E (see Figure~\ref{fig:urn-e} below). In this extended model, the case $g(x,y) = h(x) = x$ (and this terminology) was introduced in \cite{dereich-mailler-morters}, as a stochastic analogue of the model of Kingman \cite{kingman-1978}. 
\end{rmq}

Let $\mathscr{B}$ denote the Borel $\sigma$-algebra on $[0,\wmax]$, and $\mathscr{B} \otimes \mathscr{B}$ the product $\sigma$-algebra on $[0,\wmax] \times [0,\wmax]$.
In this paper, we will generally be concerned with studying following main quantities: 
\begin{enumerate}
    \item Given $A \in \mathscr{B}\otimes \mathscr{B}$, the quantity $\Xitwo(A,n)$ denotes the number of edges $(v,v')$ in the tree $\mathcal{T}_{n}$ such that $(W_v, W_{v'}) \in A$, that is,
    \begin{align} \label{eq:xi-def}
    \Xitwo(A,n) := \sum_{(v,v') \in \mathcal{T}_{n}} \mathbf{1}_{A}(W_{v},W_{v'});
    \end{align}
    \item Given $B \in \mathscr{B}$, the quantity $N_{\geq k}(B,n)$ denotes the number of vertices $v$ in the tree $\cT_n$ with out-degree at least $k$ and weight $W_v \in B$, that is, 
    \begin{align} \label{eq:deg-dist}
        N_{\geq k}(B,n) := \sum_{v \in \mathcal{T}_{n}: \deg^{+}(v, \cT_{n}) \geq k} \mathbf{1}_{B}(W_v).
    \end{align}
    \item For $B \in \mathscr{B}$, we also define $\Xi(B,n)$, so that
\begin{align} \label{eq:xi-def-2}
    \Xi(B,n) := \sum_{(v,v') \in \mathcal{T}_{n}} \mathbf{1}_{B}(W_v) = \Xi^{(2)}(B\times[0,\wmax], n)
\end{align}
(where the latter equality is in the almost sure sense).
\end{enumerate}

\subsubsection{Notation} \label{subsub:some-more-notation}
We denote by $\mathbb{N}_{0} := \mathbb{N} \cup \{0\}$ - i.e. the natural numbers including $0$. Also, in general in this paper, $W$ refers to a generic $\mu$ distributed random variable on a probability space $(\Omega, \mathbb{F}, \mathbb{P})$ taking values in the measure space $([0,\wmax],\mathscr{B})$ and $\mathbb{E}[\cdot]$ will denote expectation with respect to this random variable. In addition, we will require a probability space with an infinite sequence $W_0, W_1, W_2, \ldots$ of random variables, which are independent and identically distributed; we view these (abusing notation slightly) as belonging to the product space $(\mathbf{\Omega}, \mathbf{F}, \boldsymbol{\mathbb{P}}) := (\prod_{i \in \mathbb{N}_{0}} (\Omega_{i}, \mathbb{F}_{i}, \mathbb{P}_i)$.
For brevity, $\mathbb{E}[\cdot]$ will also denote expectations with respect to random variables on this product space. 

 In addition, for $s \in \mathbb{N}$, we denote by $[s]$ the set $\left\{1, \ldots, s\right\}$. In addition, for $\ell \in \mathbb{N}$, we denote by $[s]^{\ell}$ the $\ell$-fold Cartesian product $[s] \times \cdots \times [s]$.  Given a set $S \subset \mathcal{S}$, we denote by $S^{c}$ the complement of this set, and (if $\mathcal{S}$ has a topology made clear from context), we denote by $\overline{S}$ the topological closure of $S$. We also denote the indicator function associated with $S$ by $\mathbf{1}_{S}$. Finally, we will introduce some extra notation related to Section~\ref{sec:non-cond} in \ref{subsub:extra-not}.

\subsection{Main Results} \label{subsec:main-results}
The results in this paper depend on two sets of conditions; intuitively one set of conditions describes the `non-condensation' regime, whilst the other describes the `condensation' regime. 
\subsubsection{The Non-Condensation Regime}
The first main conditions are the following: recalling $g$ and $h$ as defined in \eqref{eq:fitness-function}, assume 
\hypertarget{c1}{} 
\begin{itemize}
    \item [\hyperlink{c1}{\textbf{C1}}]
    There exists some $\lambda^* > \tilde{g}^{*}$ such that
    \begin{align} \label{eq:cond-c1}
    \E{\frac{h(W)}{\lambda^* - \tilde{g}(W)}} = 1, 
    \end{align}
    where $\tilde{g}(x) := \E{g(x,W)}$ and $\tilde{g}^{*} := \E{\sup_{x \in [0,w^*]} g(x,W)}$.
    We call $\lambda^{*}$ the \textit{Malthusian parameter} of the process.
    \item [\hyperlink{c1}{\textbf{C2}}] For some $\maxurntwo > 0, N \in \mathbb{N}$, there exist measurable functions $\phi^{(i)}_{j}:[0,\wmax] \rightarrow [0, \maxurntwo]$, $j = 1,2$, $i \in [N]$, and a bounded continuous function $\kappa: [0,\maxurntwo]^{2N} \rightarrow \mathbb{R}_{+}$ such that
    \begin{equation} \label{eq:admissable-g}
        g(x,y) =  \kappa\left(\phi^{(1)}_{1}(x),\ldots,\phi^{(N)}_{1}(x), \phi^{(1)}_{2}(y), \ldots, \phi^{(N)}_{2}(y)\right).
    \end{equation}
\end{itemize}

\begin{rmq} \label{rem:function-form}
We expect similar results under the weaker hypothesis that $g$ and $h$ are measurable and bounded rather than Condition~\hyperlink{c1}{\textbf{C2}}. However, this condition still allows many ``reasonable'' choices of bounded measurable functions $g$. This includes the models mentioned in Remark~\ref{rem:fitness-form}, the case where $g$ is continuous, as well as functions of the form $g(x,y) = \alpha \phi_{1}(x) + \beta \phi_{2}(y)$ or $g(x,y) = \phi_{1}(x)\phi_{2}(y)$, where $\phi_{1}, \phi_{2}$ are bounded and measurable and $\alpha,\beta \geq 0$.  
\end{rmq}

 Our first theorem concerns the partition function of the process
\begin{thm} \label{th:conv-parti-c1}
Assume Conditions~\hyperlink{c1}{\textbf{C1}} and \hyperlink{c1}{\textbf{C2}}. Then we have 
\[\lim_{n\to\infty} \frac{\cZ_{n}}{n} \rightarrow \lambda^{*}\]
almost surely, where $\cZ_{n}$ and $\lambda^{*}$ respectively denote the partition function and Malthusian parameter of the process.
\end{thm}

Define $\psi(x) = h(x)/(\lambda^{*} - \tilde{g}(x))$, denote by $\psi_{*}\mu$ the pushforward measure of $\mu$ under $\psi$ - i.e. the measure such that for $A \in \mathscr{B}$
\begin{align} \label{eq:pushforward-def}
(\psi_{*}\mu)(A) = \E{\frac{h(W)}{\lambda^{*} - \tilde{g}(W)} \mathbf{1}_{A}(W)}. 
\end{align}

\begin{thm} \label{th:weak-conv-edge}
Assume Conditions~\hyperlink{c1}{\textbf{C1}} and \hyperlink{c1}{\textbf{C2}}. Then, with $\Xitwo(\cdot, n)$ as defined in Equation~\eqref{eq:xi-def}, we have
\begin{align} \label{eq:edge-conv-first}
     \frac{\Xitwo(\cdot, n)}{n} \rightarrow
     (\psi_{*}\mu \times \mu )(\cdot),
\end{align}
almost surely, in the sense of weak convergence. (Here $\psi_{*}\mu \times \mu$ denotes the product measure of $\psi_{*}\mu$ and $\mu$ on $([0,\wmax]^2, \mathscr{B}\otimes\mathscr{B})$.)
\end{thm}

We include the proofs of Theorems~\ref{th:conv-parti-c1} and \ref{th:weak-conv-edge} in Section~\ref{sec:non-cond} in~\ref{sub:conv-parti-proof} and \ref{sub:weak-conv-edge-proof}. We also prove theorems related to the degree distribution. In order to describe this result, we first describe a \textit{companion process} $(S_{i}(w))_{i \geq 0}$ that describes the evolution of the \textit{fitness} of a vertex with weight $w$ as its neighbourhood changes. First, let $W_1, W_2, \ldots$ be independent $\mu$-distributed random variables and let $w \in [0, \wmax]$. We then define the random process $(S_{i}(w))_{i \geq 0}$ inductively so that
\begin{align} \label{eq:def-comp-process}
S_0(w) := h(w);  \quad S_{i+1}(w) := S_{i}(w) + g(w, W_{i+1}), \; i \geq 0.
\end{align}
Recall from Section \ref{subsub:some-more-notation}, $\mathbb{E}[\cdot]$ denotes expectation with respect to the path of $S_{i}(W_0)$ (i.e., expectations with respect to the product measure involving the terms $W_0, W_1, W_2 \ldots$). We then have the following theorem:

\begin{thm} \label{th:degree-dist}
Assume Conditions~\hyperlink{c1}{\textbf{C1}} and \hyperlink{c1}{\textbf{C2}}. 
Then, for any $B \in \mathscr{B}$, we have 
\begin{align} \label{eq:tail-dist-gen-geom}
    \lim_{n \to \infty} \frac{N_{\geq k}(B,n)}{n} = \E{\prod_{i=0}^{k-1}\left(\frac{S_i(W_0)}{S_i(W_0) + \lambda^{*}}\right)\mathbf{1}_{B}(W_0)},
\end{align}
almost surely.
\end{thm}

We prove Theorem~\ref{th:degree-dist} in Subsection~\ref{subsec:proof-of-theorem-deg-dist}. 

\begin{rmq} \label{rem:geom}
One may interpret the right hand side of Equation~\eqref{eq:tail-dist-gen-geom} as the probability of a sequence of at least $k$ consecutive heads before a first tail when, sampling $W_0$ at random, and flipping the $i$th coin heads with probability proportional to $S_{i-1}(W_0)$.   
\end{rmq}

Theorem~\ref{th:degree-dist} allows us to deduce results about the distribution of edges with tail having certain weight, ($\Xi(\cdot, n)$ as defined in Equation~\eqref{eq:xi-def-2}). First we require the following lemma, which may be of independent interest:

\begin{lem} \label{lem:conv-sum}
Let $(S_{i}(w))_{i\geq 0}$ denote the process defined in \eqref{eq:def-comp-process} in terms of bounded, measurable functions $g,h$, suppose $\tilde{g}(x) := \E{g(x,W)}$ and $\tilde{g}_{+} = \sup_{x \in [0,\wmax]} \tilde{g}(x)$. Then, for any $w \in [0,\wmax]$, and $\lambda \geq \tilde{g}_{+}$
we have 
\begin{equation} \label{eq:gen-geom-sum}
\sum_{k=1}^{\infty} \E{\prod_{i=0}^{k-1}\left(\frac{S_i(w)}{S_i(w) + \lambda}\right)} = \frac{h(w)}{\lambda - \tilde{g}(w)},
\end{equation}
where the right hand side is infinite if $g(w) = \tilde{g}_{+}$ and 
$\lambda = \tilde{g}_{+} = \tilde{g}(w)$. In particular,
\[\sum_{k=1}^{\infty} \E{\prod_{i=0}^{k-1}\left(\frac{S_i(w)}{S_i(w) + \lambda}\right)\mathbf{1}_{B}(W_0)} = \E{\frac{h(W)}{\lambda - \tilde{g}(W)} \mathbf{1}_{B}(W_0)}.\]
\end{lem}

As the proof of this lemma detracts from the main techniques used in this paper, we delay its proof to the appendix, in Subsection~\ref{subsec:proof-of-lem:convsum}. 

\begin{rmq}
One may interpret Equation~\eqref{eq:gen-geom-sum} as a generalisation of the classic geometric series formula: if we set $g(x,y) \equiv 0$, and $q := h(w)/(h(w) + \lambda)$, the left hand side of \eqref{eq:gen-geom-sum} is $\sum_{i=1}^{\infty} q^{i} =\frac{h(w)}{h(w) + \lambda} =  \frac{q}{1-q}$. Indeed, as Remark~\ref{rem:geom} shows, one may interpret the left hand side as the expected value of a generalised geometrically distributed random variable.
\end{rmq}

Lemma~\ref{lem:conv-sum} allows us to strengthen the weak convergence result of Theorem~\ref{th:weak-conv-edge} to setwise convergence.
 
\begin{thm} \label{thm:conv-edge-dist}
Assume Condition~\hyperlink{c1}{\textbf{C1}}. Then, for any set $A \in \mathscr{B}$ we have
\begin{align*}
\frac{\Xi(A,n)}{n} \rightarrow (\psi_{*}\mu)(A),
\end{align*}
almost surely. In other words, the random probability measure $\Xi(\cdot,n)/n$ converges almost surely setwise to the limiting measure $\psi_{*}\mu$.
\end{thm}

\begin{rmq}
The setwise convergence in Theorem~\ref{thm:conv-edge-dist} is stronger than the usual weak convergence results appearing in the literature (for example in \cite{borgs-chayes,
dereich-mailler-morters}). However, as the limiting measure is absolutely continuous with respect to $\mu$, and hence almost surely with respect to the measures $\Xi(\cdot,n)$, one might expect to improve this to total variation convergence. Indeed, this is the result obtained in the simplified model first analysed by Kingman in \cite{kingman-1978} (Kingman describes the non-condensation regime as the ``democratic'' regime). 
\end{rmq}

\subsubsection{The Condensation Regime}
In this paper, we are able to describe a ``condensation'' result; 
we first make precise what ``condensation'' means.
\begin{df}
Suppose we are given a $\mu$-null set $S \subseteq [0,\wmax]$ and let $\Xi(\cdot, n)$ be as in \eqref{eq:xi-def}. We say that \textit{condensation} occurs around the set $S$, if for some nested collection of sets $(S_{\eps})_{\eps \geq 0}$,  \footnote{That is, a collection of sets such that if $\eps_1 < \eps_2$, $S_{\eps_1} \subseteq S_{\eps_2}$.} 
with $S_{\eps} \downarrow S$ as $\eps \to 0$ we have 
\[
\lim_{\eps \to 0} \lim_{n \to \infty} \frac{\Xi(S_{\eps}, n)}{n} > 0,
\]
with positive probability. 
\end{df}

\begin{rmq}
Informally, condensation means that, in the limit of the random measure $\Xi(\cdot, n)/n$, the set $S$ acquires more mass than one `would expect'. Indeed, if we swap limits, 
\[
\lim_{n \to \infty} \lim_{\eps \to 0}  \frac{\Xi(S_{\eps}, n)}{n} = \lim_{n \to \infty} \frac{\Xi(S, n)}{n} = 0,
\]
almost surely, since $\mu(S) = 0$.
\end{rmq}

Our main assumptions are now as follows:
\hypertarget{d}{} 
\begin{itemize}
    \item [\hyperlink{d}{\textbf{D1}}]   We have 
    \begin{align} \label{eq:cond-c2}
    \E{\frac{h(W)}{\tilde{g}^{*} - \tilde{g}(W)}} < 1.
    \end{align}
    \item [{\hyperlink{d}{\textbf{D2}}}] The function $g$ satisfies Condition~\hyperlink{c1}{\textbf{C2}}. 
    \item[\hyperlink{d}{\textbf{D3}}] There exists a (maximal) set of points $\mathcal{M} \subseteq \supp{(\mu)}$, such that, for any $x^{*} \in \mathcal{M}$, \[\max_{p \in [0,\wmax]} g(p,W) = g(x^{*},W) \quad \mathbb{P}-\text{a.s.}\] We denote by $x^{*}$ a generic point in $\mathcal{M}$. 
    \item [\hyperlink{d}{\textbf{D4}}] For all $\eps > 0$ sufficiently small, and a measurable function $u_{\eps}: [0,\wmax] \rightarrow \mathbb{R}_{+}$ with $\lim_{\eps \to 0} u_{\eps} = 0$ pointwise, we have 
    \begin{equation} \label{eq:dominating-set}
    \mathcal{M}_{\eps}:=\left\{x: \Prob{g(x^{*}, W) - g(x,W) < u_{\eps}(W)}=1 \right\} = \left\{x: \Prob{g(x^{*}, W) - g(x,W) < u_{\eps}(W)} > 0 \right\}.
    \end{equation}
    Under this assumption, we have $\mu(\mathcal{M}_{\eps}) > 0$.
\end{itemize}

\begin{rmq}
Note that, by the measurability of $g(\cdot,q)$ for any $q \in [0,\wmax]$, the function \[p \mapsto \esssup_{q \in [0,\wmax]} \left\{g(x^{*}, q) - g(p,q) - u_{\eps}(q)\right\}\]
is also measurable (see, e.g. Theorem~4.7.1., \cite{bogachev}). This ensures that the set $\mathcal{M}_{\eps} \in \mathscr{B}$. 
\end{rmq}

\begin{ex} \label{ex:product-example} 
In the case that $g(x,y) = \phi_{1}(x)\phi_{2}(y)$ for bounded, measurable $\phi_{1}, \phi_2$, if $\phi_{1}(x)$ is maximised on a set $\mathcal{M}$ and $\phi_{2}(y) > 0$ $\mu$-a.e., for $\eps > 0$ and $x^{*} \in \mathcal{M}$ we may take $u_{\eps} = \eps \cdot \phi_2$ and 
\begin{align*}
\mathcal{M}_{\eps}:= \left\{x: \phi_{1}(x^{*})\phi_{2}(W) -  \phi_{1}(x)\phi_{2}(W) < \eps \phi_{2}(W) \right\}
= \left\{x: \phi_{1}(x^{*}) - \phi_{1}(x) < \eps \right\}.
\end{align*}
A condition that guarantees that this set has positive measure is assuming continuity of $\phi_1$ at some point $x^{*} \in \mathcal{M}$, as this implies that $\mathcal{M}_{\eps}$ is a neighbourhood of $x^{*}$.
\end{ex}

\begin{rmq}
Conditions~\hyperlink{d}{\textbf{D1}} and \hyperlink{d}{\textbf{D2}} may be interpreted as analogues of Conditions~\hyperlink{c1}{\textbf{C1}} and \hyperlink{c1}{\textbf{C2}} in the condensation regime. One may regard $\mathcal{M}$ from~\textbf{D3} as a ``dominating set'', in the sense that $\mathbb{P}$-a.s., upon arrival of a new vertex into its neighbourhood, 
the change of the fitness of any vertex is at most the  change of the fitness of a vertex with weight with weight in $\mathcal{M}$. Condition~\hyperlink{d}{\textbf{D4}} ensures that this ``dominating property'' is captured by sets $\mathcal{M}_{\eps}$ of positive measure. Indeed the right hand side of \eqref{eq:dominating-set} implies that the change of the fitness of any vertex with weight in $\mathcal{M}_{\eps}^{c}$ is at most the  change of the fitness of a vertex having weight in $\mathcal{M}_{\eps}$. Note that $\mathcal{M}_{\eps} \downarrow \mathcal{M}$ as $\eps \to 0$. This accounts for the formation of the condensate in Theorem~\ref{thm:condensation}, since $\tilde{g}$ is maximised on $\mathcal{M}$, by \hyperlink{d}{\textbf{D1}} it must be the case that $\mu(\mathcal{M}) = 0$. 
\end{rmq}



\begin{thm} \label{th:conv-parti-d1}
Assume Conditions~\hyperlink{d}{\textbf{D1-D4}}. Then we have \[\lim_{n \to \infty} \frac{\cZ_{n}}{n} \rightarrow \tilde{g}^{*} = g(x^{*})\]
almost surely.
\end{thm}

\begin{thm} \label{thm:condensation}
Assume Conditions~\hyperlink{d}{\textbf{D1-D4}}. Then, for any $A\in \mathscr{B}$ such that, for $\eps > 0$ sufficiently small $A \cap \mathcal{M}_{\eps} = \varnothing$, we have
\begin{align} \label{eq:cond-portmanteau}
\frac{\Xi(A,n)}{n} \rightarrow (\psi_{*}\mu)(A),
\end{align}
almost surely. In addition, 
\begin{align} \label{eq:condensate-mass}
    \lim_{\eps \to 0} \lim_{n \to \infty} \frac{\Xi(\mathcal{M}_{\eps},n)}{n} = 1 - (\psi_{*} \mu)([0,\wmax]) > 0,
\end{align}
so that condensation occurs around $\mathcal{M}$. 
\end{thm}

\begin{rmq}
As the condensation occurs around the ``dominating set'' $\mathcal{M}$, in the context of reinforced branching processes (see Remarks~\ref{rmq:branching} and \ref{rmq:branching 2}), one may interpret this is families with maximum reinforced `fitness' (in this context meaning the ability to produce offspring quickly) acquiring a positive proportion of individuals in the population in the limit. This has an interesting interpretation in the context of evolution.
\end{rmq}

We have the following corollary:

\begin{cor} \label{cor:cor-seven}
Assume Conditions~\hyperlink{d}{\textbf{D1-D4}}, and the sets $\mathcal{M}_{\eps}$ in~\hyperlink{d}{\textbf{D4}} are such that $\overline{\mathcal{M}}_{\eps} \downarrow \mathcal{M}$ as $\eps \to 0$ (recalling $\overline{\mathcal{M}}_{\eps}$ denotes the topological closure of $\mathcal{M}_{\eps}$). Also, suppose that $\mathcal{M} = \{x^{*}\}$, and define the measure $\Pi(\cdot)$ such that, for $B \in \mathscr{B}$ 
\[
\Pi(B) = (\psi_{*}\mu)(B) + \left(1 - (\psi_{*}\mu)([0,\wmax])\right) \delta_{x^{*}}(B). 
\]
Then, 
\[
\frac{\Xi(\cdot,n)}{n} \rightarrow \Pi(\cdot) \quad \text{almost surely},
\]
in the sense of weak convergence.
\end{cor}

\begin{ex} 
In the case that $g(x,y) = \phi_{1}(x)\phi_{2}(y)$ for a bounded, continuous function $\phi_{1}$ and bounded measurable function $\phi_2$, if $\phi_{1}(x)$ is maximised at a unique point $x^{*}$ and $\phi_{2}(y) > 0$ $\mu$-a.e., we may take $u_{\eps}$ and $\mathcal{M}_{\eps}$ as defined in Example~\ref{ex:product-example}. Indeed, in this case 
\begin{align*}
\overline{\mathcal{M}}_{\eps}= 
 \left\{x: \phi_{1}(x^{*}) - \phi_{1}(x) \leq  \eps \right\},
\end{align*}
so that $\overline{\mathcal{M}}_{\eps} \downarrow \{x^{*}\}$ as $\eps \to 0$.
\end{ex}

Finally, we have the following extension of Theorem~\ref{th:degree-dist}:

\begin{thm} \label{thm-deg-dist-2}
Assume Conditions~\hyperlink{d}{\textbf{D1-D4}}. Then, for any $B \in \mathscr{B}$, we have
\[
\lim_{n \to \infty} \frac{N_{\geq k}(B,n)}{n} = \E{\prod_{i=0}^{k-1}\left(\frac{S_i(W)}{S_i(W) + \tilde{g}^{*}}\right) \mathbf{1}_{B}(W)},
\]
almost surely.
\end{thm}

\subsection{Discussion}
In this subsection, we provide an informal discussion of some of the implications of our main results. 
\subsubsection{Power-Law Degrees}
First note that by Theorem~\ref{th:degree-dist}, if $N_{k}(B,n)$ denotes the number of vertices with degree $k$ and weight belonging to $B$ at time $n$, then almost surely
\begin{align}
\lim_{n \to \infty} \frac{N_{k}(B,n)}{n} = \lim_{n \to \infty} \left(\frac{N_{\geq k}(B,n)}{n} - \frac{N_{\geq k+1}(B,n)}{n} \right) = \E{\frac{\lambda^{*}}{S_{k}(W)+ \lambda^{*}}\prod_{i=0}^{k-1}\left(\frac{S_i(W)}{S_i(W) + \lambda^{*}}\right) \mathbf{1}_{B}(W)}.
\end{align}
 
Now by the strong law of large numbers, one would expect (at least asymptotically), $S_{i}(W) \sim h(W) + i\tilde{g}(W)$, and thus it is natural to expect
\[
\lim_{n \to \infty} \frac{N_{k}(B,n)}{n} \sim \E{\frac{\lambda^{*}}{k\tilde{g}(W) + \lambda^{*}}\prod_{i=0}^{k-1} \left(\frac{h(W) + i\tilde{g}(W)}{h(W) + i\tilde{g}(W) + \lambda^{*}}  \right)}. 
\]
Now, if we approximate the product on the right hand side as a ratio of gamma functions, and noting that by Stirling's approximation (as $k\to \infty$),
\[
\frac{\Gamma(k)}{\Gamma(k+a)} = (1 + O(1/k))k^{-a},
\]
we thus expect that
\[\lim_{n \to \infty} \frac{N_{k}(B,n)}{n} \sim \E{k^{- \left(1 + \lambda^{*}/\tilde{g}(W)\right)} \mathbf{1}_{B}(W)}. \]
Thus, informally, this model displays a degree distribution of vertices with a given weight satisfies a power law that depends on the weights of the vertices. Noting also that $\lambda^{*}/\tilde{g}(W) > 1$, the exponent of this power law is larger than $2$. A similar analysis can be applied to the condensation regime by applying Theorem~\ref{thm-deg-dist-2}. Finally, note that these arguments can be made rigorous if the function $g$ is independent of its second argument - i.e., if $g(x,y) = g'(x)$ for some function $g': [0,\wmax] \rightarrow [0, \infty)$ - see Section 2 of \cite{asympt-gen}. 

\subsubsection{The Growth of the Neighbourhood 
of Fixed Vertex}
In the following proposition, we let $f_{v}(n)=f(N^+ (v,\mathcal{T}_n))$ denote the fitness (as defined in Equation~\eqref{eq:fitness-function}) of a vertex labelled $v\in \mathbb{N}_{0}$, with weight $w_{v}$ in the tree at time $n$. In addition, let $(R_{i})_{i \geq v}$ denote the filtration generated by the tree process $(\mathcal{T}_{i})_{i \geq v}$. Next, set 
\[
M_{v}(n) := \frac{f_{v}(n)}{\prod_{s=v}^{n-1}\left( \frac{\cZ_{s} + \tilde{g}(w_{v}) }{\cZ_{s}}\right)}.
\]

\begin{prop} \label{prop:informal-martingale}
For any vertex $v$, $(M_{v}(n))_{n \geq v}$ is a martingale with respect to the filtration $(R_{i})_{i\geq v}$. 
\end{prop}

\begin{proof}
Using the definition of the process, for $n \geq v$ we compute
\begin{align*}
\E{f_{v}(n+1)|R_{n}} & =  \frac{f_v(n)}{\cZ_n}\left(f_{v}(n) + \tilde{g}(w_v)\right) + \left(1 - \frac{f_{v}(n)}{\cZ_n}\right)f_{v}(n) 
\\ &= f_{v}(n)\left(\frac{\cZ_{n} + \tilde{g}(w_v)}{\cZ_n}\right).
\end{align*}
The result follows from the definition of $(M_v(n))_{n\geq v}$.
\end{proof}

Now, here we note two things: first, if $\deg^{+}_{v}(t)$ denotes the out-degree of vertex $v$ at time $t$, then we expect $f_{v}(t) \sim \deg^{+}_{v}(t)$ (in fact, by applying Wald's lemma, one can show $\E{f_{v}(t)} = h(w_v) + \E{\deg^{+}_{v}(t)}\tilde{g}(w_v)$). Second, by Theorems~\ref{th:conv-parti-c1} and \ref{th:conv-parti-d1}, we expect $\cZ_{i} \sim \lambda^{*} i$ and $\tilde{g}^{*} i$ in the non-condensation and condensation regimes respectively. Thus, we expect
\[
\deg^{+}_{v}(t) \sim \prod_{s=v}^{t-1}\left( \frac{\cZ_{s} + \tilde{g}(w_{v}) }{\cZ_{s}}\right) \sim \begin{cases}
t^{\tilde{g}(w_v)/\lambda^{*}}, & \text{under Conditions~\hyperlink{c1}{\textbf{C1}} and \hyperlink{c1}{\textbf{C2}}}; \\
t^{\tilde{g}(w_v)/\tilde{g}^{*}}, & \text{under Conditions~\hyperlink{d}{\textbf{D1-D4}}}.
\end{cases}
\]
Therefore, in the non-condensation regime, we expect each individual vertex to grow like $t^{\tilde{g}(w_v)/\lambda^{*}} \leq  t^{\tilde{g}^{*}/\lambda^{*}} < t$, whereas, in the condensation regime, vertices with weight $w_v$ such that $g(w_v)$ is closer and closer to $\tilde{g}^{*}$ grow at a rate closer and closer to linearity with respect to the size of the network. Note that to turn this argument into a rigorous result in terms of $\E{\deg^{+}_{v}(t)}$, one requires L1 convergence of the martingale in Proposition~\ref{prop:informal-martingale}. 
\subsection{Overview and Techniques}
\subsubsection{Overview}
In Section~\ref{sec:non-cond} we prove results about the model related to the non-condensation regime. We first review some background theory about \textit{P\'{o}lya urns} in Subsection~\ref{subsec:gen-polya}, and then, the resulst of Subsection~\ref{subsec:urn-e} are used in order to prove Theorems~\ref{th:conv-parti-c1} and \ref{th:weak-conv-edge} in Subsections~\ref{sub:conv-parti-proof} and~\ref{sub:weak-conv-edge-proof} respectively. Next, the results of Subsection~\ref{subsec:urn-d} are used to prove Theorems~\ref{th:degree-dist} and \ref{thm:conv-edge-dist} in \ref{subsec:proof-of-theorem-deg-dist} and \ref{subsec:proof-of-thm-conv-edge-dist}. In Section~\ref{sec:condensation} we extend the previous results to the condensation regime, proving Theorems~\ref{th:conv-parti-d1} and \ref{thm:condensation}, Corollary~\ref{cor:cor-seven} and Theorem~\ref{thm-deg-dist-2} in \ref{subsec:proof-of-conv-parti-d1}, \ref{subsec:proof-of-condensation}, \ref{subsec:cor-seven} and \ref{subsec-proof-of-deg-dist-2} respectively. We prove Lemma~\ref{lem:conv-sum} in the Appendix, in Subsection~\ref{subsec:proof-of-lem:convsum}.   
\subsubsection{Techniques}
This paper generalises the  techniques used in \cite{borgs-chayes} for the study of the Bianconi-Barab\'{a}si model - using a \textit{P\'{o}lya urn approximation}. However, the generalisation of this model to bounded measurable functions $h$, functions $g$ satisfying Condition~\hyperlink{c1}{\textbf{C2}}, and the possibility of arbitrary weight distributions lead to technical challenges, somewhat analogous to those arising from using a measure-theoretic approach to integration as opposed to the Riemmann integral. Applying this approach to studying the degree distribution in the case of uncountably supported weight distributions also appears to be novel. In extending the results to the condensation regime we apply a similar coupling to that used in \cite{asympt-gen}.  
\\\\
One might imagine that many of the results here may follow easily from an application of the theory of Crump-Mode-Jagers branching processes (for example as in \cite{dereich-mailler-morters}). However, the dependence between offspring distributions of a parent and its offspring means that the classic theory is not immediately applicable. This in turn raises the question of whether one can develop a theory of C-M-J branching processes with \textit{dependencies}. 

\section{The Non-Condensation Regime} \label{sec:non-cond}
\subsection{Generalised P\'{o}lya urns} \label{subsec:gen-polya}
Generalised P\'{o}lya urns are a well studied family of stochastic processes representing the composition of an \textit{urn} containing balls with certain \textit{types}. If $\mathscr{T}$ denotes the set of possible types, associated to a ball of type $t \in \mathscr{T}$ is a non-negative \textit{activity} $\mathbf{a}(t)$, which depends on the type. The process then evolves in discrete time so that, at each time-step, a ball of type $t$ is sampled at random from the urn with probability proportional to its activity $\mathbf{a}(t)$, and replaced with a number of different coloured balls according to a (possibly random) \textit{replacement rule}. 
\\\\
In the case that $\mathscr{T}$ is finite, the configuration of the urn after $n$ replacements may be  represented as a \textit{composition vector} $(X_{n})_{n \in \mathbb{N}_{0}}$ with entries labelled by type, and the activities encoded in an \textit{activity vector} $\mathbf{a}$. In this vector, the $i$th entry corresponds to the number of balls of type $i \in \mathscr{T}$. Let $(\xi_{ij})_{i,j \in \mathscr{T}}$ be the matrix whose $ij$th component denotes the random number of balls of colour $j$ added, if a ball of colour $i$ is drawn, and (following the notation of Janson in \cite{janson_urns}) define the matrix $A$ such that $A_{ij} := a_j \E{\xi_{ji}}$. The (expected) evolution of the urn in the $(n+1)$st step, may therefore be obtained by applying the matrix $A$ to the composition vector $X_{n}$. A type $i \in \mathscr{T}$ is said to be \textit{dominating} if, for any $j \in \mathscr{T}$, it is possible to obtain a ball of type $j$ starting with a ball of type $i$. If we write $i \sim j$ for the equivalence relation where $i \sim j$ if it is possible to obtain $j$ starting from a ball of type $i$, and vice versa. This partitions the types into equivalence \textit{classes}. A class $\mathscr{C} \subseteq \mathscr{T}$ is \textit{dominating} if, for every $i \in \mathscr{C}$, $i$ is dominating. Moreover, the eigenvalues of $A$ may be obtained by the restriction of $A$ to its classes; we say an eigenvalue belongs to a \textit{dominating class} if it is an eigenvalue of the restriction of $A$ to this class. Finally, we say that the urn, or the matrix $A$, is \textit{irreducible} if there is only one dominating class (note the difference when compared to irreducible matrices in the context of Markov chains: here it is possible for diagonal entries to be negative). Now, assume the following conditions are satisfied:
\begin{itemize}
    \item [(A1)] For all $i,j \in \mathscr{T}$, $\xi_{ij} \geq 0$ if $i \neq j$ and $\xi_{ii} \geq -1$. 
    \item[(A2)] For all $i,j \in \mathscr{T}$, $\E{\xi_{ij}^{2}} < \infty$.
    \item[(A3)] The largest real eigenvalue $\lambda_{1}$ of $A$ is positive.
    \item[(A4)] The largest real eigenvalue $\lambda_{1}$ is simple.
    \item[(A5)] We start with at least one ball of a dominating type.
    \item[(A6)] $\lambda_{1}$ belongs to the dominating class. 
\end{itemize}

The following is a well known result of Janson from 2004
(building on previous work by by Athreya and Karlin, see, for example, Proposition 2 in \cite{at_embedding} and Theorem 5 of \cite{split_times}): 
\begin{thm}[\cite{janson_urns}, Theorem 3.16] \label{thm:polyurnconvergence}
Assume Conditions (A1)-(A6), and suppose that $v_1$ denotes the right eigenvector, corresponding to the leading eigenvalue $\lambda_{1}$ of $A$, normalised so that $\mathbf{a}^{T} v_1 = 1$. Then, we have
\[\frac{X_n}{n} \xrightarrow{n \rightarrow \infty} \lambda_1 v_1, \]
almost surely, conditional on essential non-extinction (i.e. non-extinction of balls of dominating type). 
\end{thm}
In addition, the following lemma by Janson provides convenient criteria for satisfying (A1)-(A6):

\begin{lem}[\cite{janson_urns}, Lemma~2.1] \label{lem:jans-irred}
If $A$ is irreducible, (A1) and (A2) hold, $\sum_{j\in\mathscr{T}}\E{\xi_{ij}} \geq 0$ for all $i \in \mathscr{T}$, with the inequality being strict for some $i \in \mathscr{T}$, then (A1) - (A6) are satisfied and essential extinction does not occur. 
\end{lem}

\subsubsection{Analysing the Tree using P\'{o}lya Urns}
The idea behind analysing the distribution of edges with a given weight, and the degree distribution in this model, is to consider two different types of P\'{o}lya urns, which we call \textit{Urn E} and \textit{Urn D} respectively. We illustrate  the evolution of both these urns below. Recall, Figure~\ref{fig:tree} illustrates a possible evolution of a step of the process $(\mathcal{T}_{i})_{i \in \mathbb{N}_{0}}$; Figures~\ref{fig:urn-e} and \ref{fig:urn-d} illustrate the corresponding steps in Urn E and Urn D. 

In Urn E, we consider a generalised P\'{o}lya urn with balls of two types: singletons $x$, and tuples $(x,y)$, corresponding to 
`edges' and `loops'. A ball of type $(x,y)$ has activity $g(x,y)$ and a ball of type $x$ has activity $h(x)$. At each step, if a ball of activity $x$ or $(x,y)$ is selected, we introduce a new ball of random type $W$, and a ball of type $(x,W)$. In relation to the evolving tree, this corresponds to the event that a vertex of weight $x$ has been sampled in the subsequent step. 

\begin{figure}[H]
\captionsetup{width=.8\linewidth}
\begin{center} 
 \begin{tikzpicture}[scale=1]
\ImageNode[label={0:}]{A}{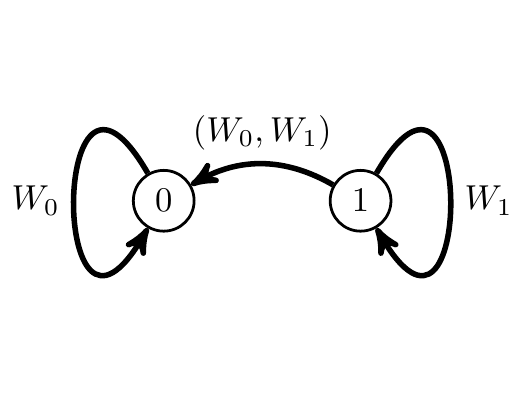}
\ImageNode[label={180:},right=of A]{B}{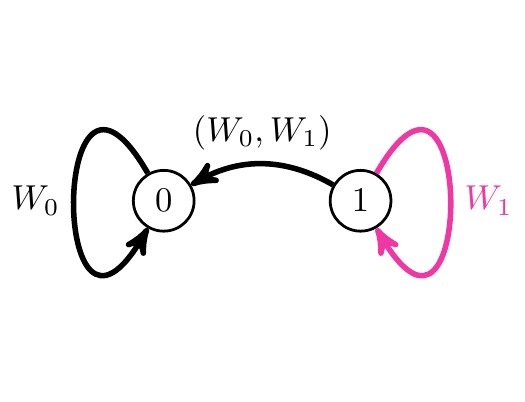}
\ImageNode[label={180:},right=of B]{C}{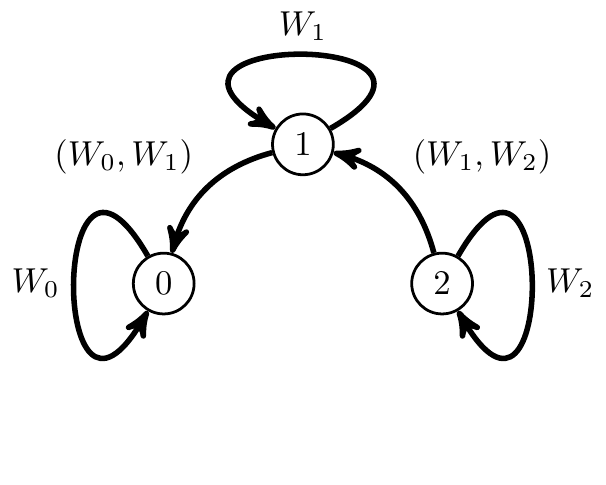}
\draw[-{Latex[length=5mm, width = 5mm]},
    arrowblue,
    line width=5pt
  ]
  (A) -- (B);
  \draw[-{Latex[length=5mm, width = 5mm]},
    arrowblue,
    line width=5pt
  ]
  (B) -- (C);
 \end{tikzpicture}
\end{center}
\centering
\caption{The evolution of the tree from $\cT_1$ to $\cT_2$ from Figure~\ref{fig:tree} viewed as a transition in Urn E. The event vertex $1$ is selected may be interpreted as the event that the `loop' $W_1$ is selected in the P\'{o}lya urn - and thus the arrival of the vertex $2$ corresponds to the arrival of the `loop' $W_2$ and the `edge' $(W_1,W_2)$ in the P\'{o}lya urn.}
\label{fig:urn-e}
\end{figure}

In Urn D, we consider a generalised P\'{o}lya urn with balls of types corresponding to tuples of varying lengths. 
A ball of type $(x_0, \ldots, x_k)$ has activity $h(x_0) + \sum_{i=1}^{k} g(x_0, x_i)$, and at each step, if a ball this type is selected, we remove it and introduce new balls of random type $W$, and a ball of type $(x_0, \ldots, x_k,W)$. In relation to the evolving tree, this corresponds to the event that a vertex of weight $x_0$, with neighbours (listed in order of arrival) having weights $x_1, \ldots, x_{k}$, has been sampled when proceeding to the subsequent step. 

\begin{figure}[H]
\captionsetup{width=.8\linewidth}
\begin{center} 
 \begin{tikzpicture}[scale=0.9]
\ImageNode[label={0:}]{A}{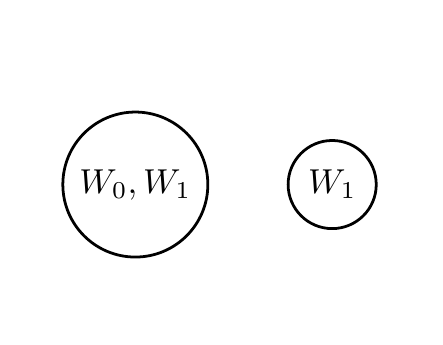}
\ImageNode[label={180:},right=of A]{B}{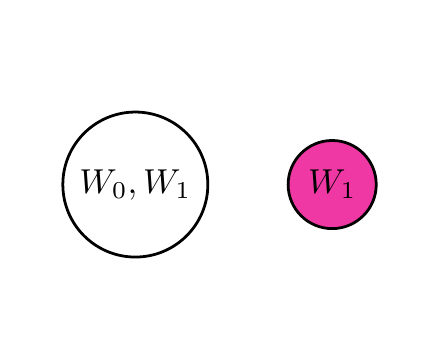}
\ImageNode[label={180:},right=of B]{C}{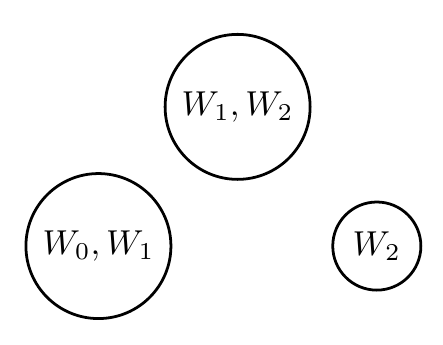}
\draw[-{Latex[length=5mm, width = 5mm]},
    arrowblue,
    line width=5pt
  ]
  (A) -- (B);
  \draw[-{Latex[length=5mm, width = 5mm]},
    arrowblue,
    line width=5pt
  ]
  (B) -- (C);
 \end{tikzpicture}
\end{center}
\centering
\caption{The evolution of the tree from $\cT_1$ to $\cT_2$ from Figure~\ref{fig:tree} viewed as a transition in Urn D. The event vertex $1$ is selected may be interpreted as the event that the ball $W_1$ is selected in the P\'{o}lya urn - and thus the arrival of the vertex $2$ corresponds to the addition of the balls $W_2$ and $(W_1,W_2)$ (the latter representing the addition of vertex $2$ into the neighbourhood of vertex $1$).}
\label{fig:urn-d}
\end{figure}

Note that, in the manner we have described Urns E and D, the set of possible types may be infinite: the measure $\mu$ may have infinite support so that $W$ may take on infinite values, and the neighbourhoods of vertices (in Urn D) may be infinite. Whilst there is some theory related to infinite type P\'{o}lya urns within the framework of measure-valued P\'{o}lya processes (see, for example, \cite{maivil}),  these results are often non-trivial to apply in practice - see, for example, pages 14-21 of \cite{dynamical-simplices}. As a result, we instead approximate these infinite urns with urns of finitely many types - enough to approximate the sigma algebras generated by $W, g(W,W')$ and $h(W)$, where $W, W'$ are i.i.d random variables sampled according to $\mu$. In Subsection~\ref{subsec:urn-e} we apply this analysis to Urn~E, and in Subsection~\ref{subsec:urn-d} we apply it to Urn~D. We first introduce some extra notation specific to this section.

\subsubsection{Some More Notation and Terminology} \label{subsub:extra-not}
In order to apply the finite P\'{o}lya urn theory, given a set of types $\mathscr{T}$, we denote by $\mathbb{V}_{\mathscr{T}}$ the \textit{free vector space} over the field $\mathbb{R}$ generated by $\mathscr{T}$ (i.e. the vector space where vectors are indexed by the elements of $\mathscr{T}$). We will generally view an urn with types $\mathscr{T}$ as a stochastic process taking values in $\mathbb{V}_{\mathscr{T}}$. In addition we will generally identify vectors $\mathbf{v} \in \mathbb{V}_{\mathscr{T}}$ interchangeably with functions $\mathbf{v}: \mathscr{T} \rightarrow \mathbb{R}$. Thus, for $x \in \mathscr{T}$, $\mathbf{v}(x)$ denotes the entry of the vector corresponding to $x$, and for $\mathbf{v}_1, \mathbf{v}_2 \in \mathbb{V}_{\mathcal{B}}$, we have $(\mathbf{v}_1\mathbf{v}_2)(x) = \mathbf{v}_1(x) \mathbf{v}_2(x)$. For $x \in \mathscr{T}$, we define $\delta_{x} \in \mathbb{V}_{\mathscr{T}}$ such $\delta_{x}(y) = 1$ if $y = x$ and $0$ otherwise.
\\\\
For a Borel measurable set $S \subseteq \mathbb{R}$, and a finite set $\mathcal{A}$ of Borel measurable subsets of $S$, we say that $\mathcal{A}$ forms a \textit{good partition} of $S$ if, given any two nonempty sets $A_{i}, A_{j} \in \mathcal{A}$, $A_{i} \cap A_{j} \neq \varnothing \implies A_i = A_j$, and $\bigcup_{i=1}^{s} A_{i} = S$. Note that, given two good partitions  $\mathcal{A}_1, \mathcal{A}_2$ of $S$, the set 
\begin{equation} \label{eq:refined-partition}
\left\{A_1 \cap A_2: A_1 \in \mathcal{A}_1, A_2 \in \mathcal{A}_2\right\}
\end{equation} 
also forms a good partition of $S$. In addition, if $\mathcal{A}$ is a good partition of $S$, we say that $\mathcal{A}'$ forms a \textit{refined good partition} (often we will just write \textit{refined partition}) of $\mathcal{A}$, if, for any $A' \in \mathcal{A}'$ there exists $A \in \mathcal{A}$ such that $A' \subseteq{A}$. The following lemma (which is well-known) justifies the use of the word `refined'. 

\begin{lem} \label{lem:refined-partitions-partition}
Suppose $\mathcal{A}$ is a good partition of a set $S$, and $\mathcal{A}'$ is a refined partition of $\mathcal{A}$. Then, for any set $A \in \mathcal{A}$, there exist sets $X_1, \ldots, X_{s} \in \mathcal{A}'$ such that $A = \bigcup_{i=1}^{s} X_{i}$. In particular, $\{X_{i}\}_{i \in [s]}$ forms a good partition of $A$.
\end{lem}

\begin{proof}
For $A \in \mathcal{A}$, define the sub-family $\mathcal{X}:= \left\{A' \in \mathcal{A}': A' \subseteq A \right\}. $
Suppose $U := \left(\bigcup_{X \in \mathcal{X}} X\right) \neq A$. Then, there exists $x \in A \setminus U$, and since $\mathcal{A'}$ partitions $S$, $x \in V'$, for some set $V' \in \mathcal{A}'$ with $V' \not \subseteq A$. But then, since $\mathcal{A}'$ is a refined partition of $\mathcal{A}$, $V' \subseteq V$ for some $V \in \mathcal{A}$.
But then, this implies that either $V \cap A \neq \varnothing$, contradicting the fact that $\mathcal{A}$ is a good partition of $S$, or $V = A$, contradicting the fact that $V' \not \subseteq A$.
\end{proof}

\subsection{Urn~E} \label{subsec:urn-e}
In this subsection we will refer to
Conditions~\hyperlink{c1}{\textbf{C1}} and \hyperlink{c1}{\textbf{C2}}. We will analyse the process under these conditions by coupling the tree process $(\cT_{n})_{n \in \mathbb{N}_{0}}$ with P\'{o}lya urn processes, parametrised by $m \in \mathbb{N}$. These may be interpreted as finite approximations of Urn~E. 
 Now, for each $x \in \mathbb{R}$ and $m \in \mathbb{N}$ we define a good partition of interval $[0,x]$ into into $2^{m}$ intervals (a \textit{dyadic partition}): set
\begin{align*}
    \mathcal{D}^{m}_{1}(x) := [0,2^{-m}x], \quad \text{ and } \quad \mathcal{D}^{m}_{i}(x) := ((i-1)\cdot2^{-m}x, i\cdot2^{-m}x], \; i \in [2^{m}] \setminus \{1\}. 
\end{align*}
For $i \in [2^m]$, we also denote the closure of $\mathcal{D}^{m}_{i}(x)$ by $\overline{\mathcal{D}}^{m}_{i}(x)$, so that   
\[
\overline{\mathcal{D}}^{m}_{i}(x) = [(i-1) \cdot 2^{-m}x, i \cdot 2^{-m} x].
\]
Supposing $h:[0,\wmax] \rightarrow \mathbb{R}_{+}$ takes values in $[0, h_{\max{}}]$, and recalling the functions $\phi^{(j)}_{1}, \phi^{(j)}_{2}, j \in [N]$ from Condition~\hyperlink{c1}{\textbf{C2}}, for each $i \in [2^{m}]$, $j \in [N]$ and $k \in [2]$, we set 
\begin{align*}
\mathcal{H}^{m}_{i} := h^{-1}\left(\mathcal{D}^{m}_{i}(h_{\max})\right)
\quad \text{and } 
    \Phi^{m}_{k}(i,j) := \left(\phi^{(j)}_{k}\right)^{-1}\left(\mathcal{D}^{m}_{i}(\maxurntwo)\right). 
\end{align*}
By the measurability assumptions on the functions $\phi^{(j)}_{k}$ and $h$, for each $i \in [2^{m}]$ we have $\mathcal{H}^{m}_{i}, \Phi^{m}_{i}(j,k) \in \mathscr{B}$, and thus, the collections of sets $\left\{\mathcal{H}^{m}_{i}\right\}_{i \in [2^m]}$ and  
$\left\{\Phi^{m}_{k}(i,j)\right\}_{i \in [2^m]}$ form good partitions of $[0,\wmax]$. We now split the latter family of sets to form a refined partition: for $\mathbf{i} = (i_1, \ldots, i_N), \mathbf{j} = (j_1, \ldots, j_N)  \in [2^m]^{N}$, if we set
\begin{align} \label{eq:partition-x-y-coordinate}
\nonumber
\Phi^{m}_{1}(\mathbf{i}) &= \Phi^{m}_{1}(i_1,1) \cap \Phi^{m}_{1}(i_2,2) \cap \cdots \cap \Phi^{m}_{1}(i_N,N)   \quad \text{ and, } \\
\Phi^{m}_{2}(\mathbf{j}) &= \Phi^{m}_{2}(j_1,1) \cap \Phi^{m}_{2}(j_2,2) \cap \cdots \cap \Phi^{m}_{2}(j_N,N), 
\end{align}
by iteratively applying Equation~\eqref{eq:refined-partition}, the families of sets $\left\{\Phi^{m}_{1}(\mathbf{i})\right\}_{\mathbf{i}\in [2^m]^{N}}$ and $\left\{\Phi^{m}_{2}(\mathbf{j})\right\}_{\mathbf{j}\in [2^m]^{N}}$ also form good partitions of $[0,\wmax]$. Now, given $\mathbf{v} = (v_1, \ldots, v_N) \in [2^m]^{N}$, set 
\[
\overline{\mathcal{D}}^{m}_{\mathbf{v}}(\maxurntwo) := \overline{\mathcal{D}}^{m}_{v_1}(\maxurntwo) \times \overline{\mathcal{D}}^{m}_{v_2}(\maxurntwo) \times \cdots \times \overline{\mathcal{D}}^{m}_{v_N}(\maxurntwo),
\]
and observe that, given $\mathbf{i}, \mathbf{j} \in [2^m]^N$, the construction of the sets in Equation~\eqref{eq:partition-x-y-coordinate} are such that
$(x,y) \in \Phi^{m}_{1}(\mathbf{i}) \times \Phi^{m}_{2}(\mathbf{j})$ implies that 
\[
\left(\phi^{(1)}_{1}(x), \ldots, \phi^{(N)}_{1}(x), \phi^{(1)}_{2}(y), \ldots, \phi^{(N)}_{2}(y)\right) \in \overline{\mathcal{D}}^{m}_{\mathbf{i}}(\maxurntwo) \times \overline{\mathcal{D}}^{m}_{\mathbf{j}}(\maxurntwo) 
\]
Now, recalling the function $\kappa:[0,\maxurntwo]^{2N} \rightarrow [0, g_{\max}]$ from Condition~\hyperlink{c1}{\textbf{C2}}, for each $\mathbf{i}, \mathbf{j} \in [2^{m}]^{N}$, by continuity on the compact set $\overline{\mathcal{D}}^{m}_{\mathbf{i}}(\maxurntwo) \times \overline{\mathcal{D}}^{m}_{\mathbf{j}}(\maxurntwo)$, for $(x,y) \in \Phi^{m}_{1}(\mathbf{i}) \times \Phi^{m}_{2}(\mathbf{j})$  we have
\begin{align} \label{eq:min-kappa}
\nonumber \kappa\left(\phi^{(1)}_{1}(x), \ldots, \phi^{(N)}_{1}(x), \phi^{(1)}_{2}(y), \ldots, \phi^{(N)}_{2}(y)\right) &\geq \inf_{\mathbf{u},\mathbf{v} \in \overline{\mathcal{D}}^{m}_{\mathbf{i}}(\maxurntwo) \times \overline{\mathcal{D}}^{m}_{\mathbf{j}}(\maxurntwo)}  \left\{\kappa(\mathbf{u},\mathbf{v}) \right\}
\\ &= \min_{\mathbf{u},\mathbf{v} \in \overline{\mathcal{D}}^{m}_{\mathbf{i}}(\maxurntwo) \times \overline{\mathcal{D}}^{m}_{\mathbf{j}}(\maxurntwo)} \left\{ \kappa(\mathbf{u},\mathbf{v}) \right\} =: \kappa^{-}(\mathbf{i}, \mathbf{j}), 
\end{align}
and likewise, 
\begin{align} \label{eq:max-kappa}
\nonumber \kappa\left(\phi^{(1)}_{1}(x), \ldots, \phi^{(N)}_{1}(x), \phi^{(1)}_{2}(y), \ldots, \phi^{(N)}_{2}(y)\right) &\leq \sup_{\mathbf{u},\mathbf{v} \in \overline{\mathcal{D}}^{m}_{\mathbf{i}}(\maxurntwo) \times \overline{\mathcal{D}}^{m}_{\mathbf{j}}(\maxurntwo)}  \left\{\kappa(\mathbf{u},\mathbf{v}) \right\}
\\ &= \max_{\mathbf{u},\mathbf{v} \in \overline{\mathcal{D}}^{m}_{\mathbf{i}}(\maxurntwo) \times \overline{\mathcal{D}}^{m}_{\mathbf{j}}(\maxurntwo)} \left\{ \kappa(\mathbf{u},\mathbf{v}) \right\} =: \kappa^{+}(\mathbf{i}, \mathbf{j}). 
\end{align}
Now, set
\begin{align*}
&    g^{-}(x,y) := \sum_{\mathbf{i},\mathbf{j} \in [2^m]^N} \kappa^{-}(\mathbf{i}, \mathbf{j}) \mathbf{1}_{\Phi^{m}_{1}(\mathbf{i}) \times \Phi^{m}_{2}(\mathbf{j})}(x,y), \quad \text{and } \quad  g^{+}(x,y) := \sum_{\mathbf{i},\mathbf{j} \in [2^m]^{N}} \kappa^{+}(\mathbf{i}, \mathbf{j}) \mathbf{1}_{\Phi^{m}_{1}(\mathbf{i}) \times \Phi^{m}_{2}(\mathbf{j})}(x,y);
\end{align*}
and
\begin{align*}
    h^{-}(x) := \sum_{i=1}^{2^m} (i-1) \cdot 2^{-m}h_{\max} \mathbf{1}_{\mathcal{H}_i}(x), \quad h^{+}(x) := \sum_{i=1}^{2^m} i \cdot 2^{-m}h_{\max} \mathbf{1}_{\mathcal{H}_i}(x).
\end{align*}
One should interpret these functions as \textit{lower} and \textit{upper} approximations to $g$ and $h$, indeed, by construction, we now have the following lemma:

\begin{lem} \label{lem:unif-conv}
We have $g^{-} \uparrow g$, $h^{-} \uparrow h$,  $g^{+} \downarrow g$ and $h^{+} \downarrow h$ uniformly, as $m \rightarrow \infty$. 
\end{lem}

\begin{proof}
We prove the statements regarding $h^{-}$ and $g^{-}$; the others follow analogously (in the case of $g^{+}$ using Equation~\eqref{eq:max-kappa} instead of \eqref{eq:min-kappa}). Since the sets $(\mathcal{H}_{i}^{m})_{i \in [2^{m}]}$ form a good partition of $[0,\wmax]$, for each $m \in \mathbb{N}$, given $x \in [0, \wmax]$, we have $x \in \mathcal{H}^{m}_{j}$ for some $j \in [2^{m}]$, and thus
\[
h^{-}(x) = (j-1)\cdot2^{-m}h_{\max} \leq h(x) \leq h^{-}(x) + 2^{-m}h_{\max}.
\]
The convergence result for $h^{-}$ follows. 
Now, note that by uniform continuity of $\kappa$ on the compact set $[0, \maxurntwo]^{2N}$, for $\eps > 0$, let $M$ be sufficiently large so that for all $\mathbf{u}, \mathbf{v} \in [0, \maxurntwo]^{2N}$ 
\begin{equation} \label{eq:uniform-cont}
\|\mathbf{u} - \mathbf{v}\| < \sqrt{2N}\cdot 2^{-M}\maxurntwo \quad \implies \quad  |\kappa(\mathbf{u}) - \kappa(\mathbf{v})| < \eps.
\end{equation}
Now, for any $m > M$, given $(x,y) \in [0,\wmax] \times [0,\wmax]$, there exists a unique set $\Phi^{m}_{1}(\mathbf{i}) \times \Phi^{m}_{2}(\mathbf{j})$ containing $(x,y)$, which implies that 
\[
\left(\phi^{(1)}_{1}(x), \ldots, \phi^{(N)}_{1}(x), \phi^{(1)}_{2}(y), \ldots, \phi^{(N)}_{2}(y)\right) \in \overline{\mathcal{D}}^{m}_{\mathbf{i}}(\maxurntwo) \times \overline{\mathcal{D}}^{m}_{\mathbf{j}}(\maxurntwo).
\]
Thus, for each $j \in [N]$, combining this equation with the definition of $\kappa^{-}(\mathbf{i}, \mathbf{j})$ from~\eqref{eq:min-kappa}, we have 
\[
\kappa^{-}(\mathbf{i}, \mathbf{j}) \leq \kappa\left(\phi^{(1)}_{1}(x), \ldots, \phi^{(N)}_{1}(x), \phi^{(1)}_{2}(y), \ldots, \phi^{(N)}_{2}(y)\right) \leq \kappa^{-}(\mathbf{i}, \mathbf{j}) + \eps,
\]
and thus 
\[
g^{-}(x,y) \leq g(x,y) \leq g^{-}(x,y) + \eps.
\]
The result now follows. 
\end{proof}

Now, using the good partitions $\left\{\mathcal{H}_{i}^{m}\right\}_{i \in [2^{m}]}$,  $\left\{\Phi^{m}_{1}(\mathbf{i})\right\}_{\mathbf{i} \in [2^{m}]^{N}}$, $\left\{\Phi^{m}_{2}(\mathbf{j})\right\}_{\mathbf{j} \in [2^{m}]^{N}}$ and $\left\{\mathcal{D}^{m}_{i}(\wmax)\right\}_{i \in [2^{m}]}$, we will form an even more refined partition, which we will use as the ``building blocks'' of the evolution of the P\'{o}lya urn approximations. For each $m$, define the good partition $\mathscr{I}^{m}$ such that
\begin{align} \label{eq:building-blocks}
    \mathscr{I}^m := \bigg\{I \in \mathscr{B}: I = \mathcal{H}_{p}^m \cap \mathcal{D}^{m}_{q}(\wmax) \cap \Phi^{m}_{1}(\mathbf{i}) \cap \Phi^{m}_{2}(\mathbf{j}), \; p,q\in [2^{m}], \mathbf{i}, \mathbf{j} \in [2^m]^{N}\bigg \}.
\end{align}
Intuitively, this family of sets is such that the finite 
$\sigma$-algebra $\sigma(\mathscr{I}^{m})$, is ``fine enough" to approximate $\mathscr{B}$, and also capture the behaviour of $g$ and $h$. Observe that, for $m_1 < m_2$, $\mathscr{I}^{m_2}$ is a refined partition of $\mathscr{I}^{m_1}$. 

Suppose $\left|\mathscr{I}^m\right| = \dimurn$; then we label the sets in $\mathscr{I}^{m}$ arbitrarily as $\left(\mathcal{I}^{m}_{i}\right)_{i\in [\dimurn]}$. Now, for each $(x,y) \in \mathcal{I}^{m}_{i} \times \mathcal{I}^{m}_{j}$, $g^{-}(x,y)$ and $g^{+}(x,y)$ are constant, depending only on $(i,j)$, and likewise, for each $x \in \mathcal{I}^{m}_{\ell}$, $h^{-}(x)$ and $h^{+}(x)$ are constant, depending on $\ell$. Motivated by this, for each $(i,j) \in [\dimurn] \times [\dimurn]$, we define the following quantities:
\begin{align}
    g_{\min{}}(i,j) := g^{-}(x,y), \quad g_{\max{}}(i,j) := g^{+}(x,y), \quad (x,y) \in \mathcal{I}^{m}_{i} \times \mathcal{I}^{m}_{j}, 
\end{align}
and likewise, for each $\ell \in [\dimurn]$, we define 
\begin{align}
    h_{\min{}}(\ell) := h^{-}(x), \quad h_{\max{}}(\ell) := h^{+}(x), \quad x \in \mathcal{I}^{m}_{\ell}, 
\end{align}
We also set 
\begin{align}
    r(x) := \sum_{i=1}^{\dimurn} i \mathbf{1}_{\mathcal{I}^{m}_{i}}(x),
\end{align}
so that $r(x) = i$ if $x \in \mathcal{I}^{m}_{i}$. In addition, set 
\begin{align} \label{eq:var-defs}
\nonumber p_{i}^{m} := \mu\left(\mathcal{I}^{m}_{i}\right), i \in [\dimurn], \quad g^{*}(j) := \max_{i \in [\dimurn]} \left\{g_{\max{}}(i,j)\right\}, \\ 
\tilde{g}_{-}(i) := \sum_{j=1}^{\dimurn} p^{m}_{j} g_{\min{}}(i,j), \quad \tilde{g}_{+}(i) := \sum_{j=1}^{\dimurn} p^{m}_{j} g_{\max{}}(i,j), \quad \text{ and } \quad 
\tilde{g}^{*}_{+} := \sum_{j=1}^{\dimurn} p^{m}_{j} g^{*}(j).
\end{align}
Recall that $\tilde{g}(x) = \E{g(x,W)}$, and note that $\tilde{g}_{-}(r(x)) = \E{g^{-}(x,W)}$, $\tilde{g}_{+}(r(x)) = \E{g^{+}(x,W)}$ and $\tilde{g}^{*}_{+} = \E{\max_{x \in [0,\wmax]} g^{+}(x,W)}$.
Then, observe that by Lemma~\ref{lem:unif-conv} and dominated convergence, $\tilde{g}_{-}(r(x)) \uparrow \tilde{g}(x)$, $\tilde{g}_{+}(r(x)) \downarrow \tilde{g}(x)$ and
\[
\tilde{g}^{*}_{+} \downarrow \E{\sup_{x \in [0, \wmax]}g(x,W)} = \tilde{g}^{*}, \quad \text{as $m \rightarrow \infty$.}
\]

\subsubsection{The Urn}
We are now ready to define the urn process $(\mathcal{U}_{n})_{n \in \mathbb{N}_{0}}$.
For $i \in \mathbb{N}$, set
\[[\dimurn]^{i} := [\dimurn] \times [\dimurn] \cdots \times [\dimurn] = \left\{(u_0, \ldots u_{i-1}): u_0, \ldots, u_{i-1} \in [\dimurn]\right\}, \]
and
\begin{align*}
\mathcal{B} := [\dimurn] \cup [\dimurn]^{2} \cup \left(\{\dimurn+1\} \times [\dimurn]\right);
\end{align*} 
this will represent the set of types in Urn E. We now define parameters $\bfgamm$ such that, for $x \in \left[\dimurn\right] \cup [\dimurn] \times [\dimurn]$, 
\begin{align} \label{eq:bfgamm-def}
\bfgamm(x) = \begin{cases} \frac{\gmin{i,j}}{\gmax{i,j}}, & x = (i,j) \in [\dimurn]^{2}, \gmax{i,j} > 0; \\
\frac{\hmin{i}}{\hmax{i}}, & x = i \in [\dimurn], \hmax{i} > 0; \\
0, & \text{otherwise.}
\end{cases}
\end{align}

Then, we define the urn process $(\mathcal{U}_{n}^{m})_{n \in \mathbb{N}_{0}}$ as the urn process with \textit{activities} $\mathbf{a}$ such that
\begin{align} \label{eq:act-def}
\mathbf{a}(x) = \begin{cases}
\gmax{i,j} & \text{if $x = (i,j)$, $i,j \in [\dimurn]$}\\
g_{\max{}}^{*}(j) & \text{if $x = (i,j)$, $i = \dimurn+1, j \in [\dimurn]$}\\
\hmax{i} & \text{if $x = i \in [\dimurn]$};
\end{cases}
\end{align}
and a \textit{replacement matrix} $M$ such that, for $x,x' \in \mathbb{V}_{\mathcal{B}}$, 
\begin{align*}
    M_{x', x} = \begin{cases}
    (\bfgamm \mathbf{a})(x) p^{m}_{\ell}, & \text{if } x' = (i, \ell), x \in (\{i\} \times [\dimurn]) \cup \{i\}, i, \ell \in [\dimurn];  \\
    (\mathbf{a} - \bfgamm \mathbf{a})(x) p^{m}_{\ell}, & \text{if } x' = \left(\dimurn+1, \ell\right), x \in \mathcal{B}; \\
    \mathbf{a}(x) p^{m}_{\ell}, & \text{if } x' = \ell, x \in \mathcal{B}; \\
    0 & \text{otherwise.}
    \end{cases}
\end{align*}
Note that it is not necessarily the case that $M$ is irreducible: it may be the case that $\mathbf{a}(x) = 0$ for certain $x \in \mathcal{B}$ (this is possible if $\hmax{i} =0$ or $\gmax{i,j} = 0$), or it may be the case that $p^{m}_{\ell} = 0$ for certain choices of $\ell$. We therefore define the following subsets of $\mathcal{B}$: 
\begin{align*}
    \mathscr{U}_1 := \left\{x \in \mathcal{B}: M_{x'x} = 0 \; \forall x' \in \mathcal{B} \right\} = \left\{x \in \mathcal{B}: \mathbf{a}(x) = 0\right\},
\end{align*}
and 
\begin{align*}
    \mathscr{U}_2 := \left\{x' \in \mathcal{B}: M_{x'x} = 0 \; \forall x \in \mathcal{B} \right\}.
\end{align*}
Also assume that $\mathscr{U}_1 \cap \mathscr{U}_2 = \varnothing$; if not, we replace $\mathscr{U}_1$ by $\mathscr{U}_1 \setminus \mathscr{U}_2$.  
We then set $R= \mathcal{B} \setminus (\mathscr{U}_1 \cup \mathscr{U}_2)$, and let $M_{R}$ be the restriction of $M$ to $R$. 
It is easy to check that $M_{R}$ is irreducible, and thus, by Lemma~\ref{lem:jans-irred}, has a unique largest positive eigenvalue $\lambda_{m}$ with corresponding eigenvector $\mathbf{u}_{R}$. But then, writing $M$ in block form (with columns and rows labelled by $R, \mathscr{U}_1, \mathscr{U}_2$) for suitable matrices $A,B, C$, we have
\[
M = \begin{blockarray}{cccc}
R& \mathscr{U}_1 & \mathscr{U}_2 \\
\begin{block}{(ccc)c}
  \bigstrut[t]  M_{R} & 0 & B & R\\
  A & 0 & C & \mathscr{U}_1 \\
  0 & 0 & 0 & \mathscr{U}_2 \bigstrut[b] \\
\end{block}
\end{blockarray}.
\]
Thus, $M$ has the same largest positive eigenvalue, with corresponding right eigenvector given (in block form) by 
\[
\mathbf{u}_{m} = \begin{bmatrix}
           \mathbf{u}_{R} \\
           \left(\lambda_{R}^{-1}\right) A \mathbf{u}_{R} \\
           0
         \end{bmatrix}.
\]
Here, we assume $\mathbf{u}_{m}$ is normalised so that $\mathbf{a} \cdot \mathbf{u}_{m} = 1$.  In addition, (assuming we begin with a single ball $x \in R$), one readily verifies that the restriction of $M$ to $R$ and $\mathscr{U}_1$ satisfies conditions (A1)-(A6) of Subsection~\ref{subsec:gen-polya}. Note also, that at each time-step the probability of adding a ball of type $x \in \mathscr{U}_2$ is $0$ and thus, for each $n \in \mathbb{N}_{0}$, $\mathcal{U}_{n}(x) = 0$ almost surely. Therefore, combining this fact with Theorem~\ref{thm:polyurnconvergence}, we have the following corollary.
\begin{cor} \label{cor:slln-polya-coupling}
With $\mathbf{u}_{m}, \lambda_{m}$ and $R$ as defined above, assuming we begin with a ball $x \in R$, we have 
\begin{align} \label{eq:lim-urn-comp}
    \frac{\mathcal{U}^{m}_n}{n} \xrightarrow{n \to \infty} \lambda_{m} \mathbf{u}_{m}
\end{align}
almost surely. In particular, almost surely
\begin{align} \label{eq:lambda-em}
    \frac{\mathbf{a} \cdot \mathcal{U}^{m}_{n}}{n} \xrightarrow{n \to \infty} \lambda_{m}.
\end{align}
\end{cor}

In the coupling below, the assumption of a ball $x \in R$ is met by the tree process being initiated by a vertex $0$ with weight $W_0$ sampled at random from $\mu$ and satisfying $h(W_0) > 0$.

\subsubsection{Coupling Urn E with the Tree Process}
For $A \in \mathscr{B} \otimes \mathscr{B}$, recall the definition of $\Xitwo(A, n)$ from Equation~\eqref{eq:xi-def}: this is the number of directed edges $(v,v')$ of $\cT_{n}$ where $(W_v,W_{v'}) \in A$.

\begin{prop} \label{prop:coupling-tree-with-urn}
There exists a coupling $((\hat{\mathcal{U}}^{m}_{n})_{m \in \mathbb{N}}, \hat{\cT}_{n})_{n \in \mathbb{N}_{0}}$ of the P\'{o}lya urn processes $\left\{(\mathcal{U}^{m}_{n})_{n \in \mathbb{N}_{0}}, m \in \mathbb{N}\right\}$ and the tree process $(\cT_{n})_{n \in \mathbb{N}_{0}}$ such that, for each $m \in \mathbb{N}$, almost surely (on the coupling space), 
$\hat{\mathcal{U}}^{m}_{0}$ consists of a single ball $\ell \in R$ 
and, in addition, for
$(i,j) \in \left[\dimurn\right]^{2}$, we have 
\begin{align} \label{eq:xi-relation-1}
   & \hat{\mathcal{U}}^{m}_{n}((i,j)) \leq \Xitwo(\mathcal{I}^{m}_{i}\times\mathcal{I}^{m}_{j}, n), \\ \label{eq:xi-relation-2} &
   \sum_{(i,j) \in [\dimurn]^{2}} \left(\Xitwo(\mathcal{I}^{m}_{i}\times \mathcal{I}^{m}_{j}, n) - \hat{\mathcal{U}}^{m}_{n}((i,j))\right) =   \sum_{j=1}^{\dimurn}\hat{\mathcal{U}}^{m}_{n}((\dimurn+1, j)),
\end{align}
and 
\begin{align} \label{eq:parti-ineq-poly1}
    (\bfgamm \mathbf{a}) \cdot \hat{\mathcal{U}}^{m}_n \leq \cZ_n \leq \mathbf{a} \cdot \hat{\mathcal{U}}^{m}_n. 
\end{align}
for all $n \in \mathbb{N}_{0}$. 
\end{prop}
\begin{proof}
First sample the entire tree process $(\hat{\cT}_{n})_{n \in \mathbb{N}_{0}}$; we will use this to define the evolution of the urn processes. Moreover, for $i \in \left[\dimurn\right]$ let
\[
\eta_{n}(i) := \sum_{v \in \mathcal{T}_{n} : r(v) = i}  f(N^{+}(v, \mathcal{T}_{n}));
\]
i.e., the sum of fitnesses of vertices with weight belonging to $\mathcal{I}_{i}^{m}$. Also, for $i \in \left[\dimurn\right]$ define 
\[
\theta_{n}(i) : = (\bfgamm \, \mathbf{a} \,  \hat{\mathcal{U}}^{m}_{n})(i) + 
\sum_{j=1}^{\dimurn} (\bfgamm \, \mathbf{a} \,  \hat{\mathcal{U}}^{m}_{n})((i,j)).
\]
Finally, recall that $\cZ_{n}$ denotes the partition function associated with the tree at time $n$. Assume that at time $0$ the tree consists of a single vertex $0$ such that $r(W_0) = \ell \in \left[\dimurn\right]$.  Then, set $\hat{\mathcal{U}}^{m}_0 = \delta_{\ell}$. Using the definition of $r$, since $W_0 \in \mathcal{I}^{m}_{\ell}$
\[
0 < \cZ_{0} = h(W_0) \leq \hmax{\ell} = \mathbf{a} \cdot \hat{\mathcal{U}}^{m}_0,
\]
and by the choice of $\bfgamm$, we have 
\[ 
\eta_{0}(\ell) = h(W_0) \geq \hmin{\ell}
= (\bfgamm \, \mathbf{a} \, \hat{\mathcal{U}}^{m}_0)(\ell) = 
\theta_0(\ell).
\]
In this case, \eqref{eq:xi-relation-1} and \eqref{eq:xi-relation-2} are trivially satisfied since both sides of both equations are $0$. 
Now, assume inductively that after $n$ steps in the urn process, Equations~\eqref{eq:xi-relation-1} and \eqref{eq:xi-relation-2} are satisfied, we have \begin{align} \label{eq:eta-ineq}
\eta_{n}(k) \geq \theta_{n}(k) \quad  \text{for each} \quad k \in \left[\dimurn\right], \end{align} and moreover, $\cZ_n \leq \mathbf{a} \cdot \hat{\mathcal{U}}^{m}_n$. Note that~\eqref{eq:eta-ineq} implies the left hand side of~\eqref{eq:parti-ineq-poly1}, since 
\[ 
(\bfgamm \mathbf{a}) \cdot \hat{\mathcal{U}}^{m}_n = \sum_{k=1}^{\dimurn} \theta_{n}(k) \leq \sum_{k=1}^{\dimurn} \eta_{n}(k) = \cZ_{n}. 
\]
\\
Let $s$ be the vertex sampled from $\cT_n$ in the $(n+1)$st step, and assume that $r(W_s) = \ell'$, $r(W_{n+1}) = k$. Then, for the $(n+1)$th step in the urn:
sample an independent random variable $U_{n+1}$ uniformly distributed on $[0,1]$. Then:
\begin{itemize}
    \item If $U_{n+1} \leq \frac{\theta_{n}(\ell') \cZ_{n}}{\eta_{n}(\ell') \mathbf{a} \cdot \hat{\mathcal{U}}^{m}_n}$, add balls of type $(\ell', k)$ and $k$ to the urn (i.e. set $\hat{\mathcal{U}}^{m}_{n+1} = \hat{\mathcal{U}}^{m}_{n} + \delta_{(\ell',k)} + \delta_{k}$).
    \item Otherwise, add balls of type $\left(\dimurn+1,k\right), k$. 
\end{itemize}
Note that, in the first case, we have 
\begin{align*}
\Xitwo(\mathcal{I}_{\ell'}^{m} \times \mathcal{I}_{k}^{m}, n+1) = \Xitwo(\mathcal{I}_{\ell'}^{m} \times \mathcal{I}_{k}^{m}, n) + 1 & \geq \hat{\mathcal{U}}^{m}_{n}((\ell',k)) + 1 = \hat{\mathcal{U}}^{m}_{n+1}((\ell',k)) 
\end{align*}
and for $i \neq \ell'$ or $j \neq k$
\[\Xitwo(\mathcal{I}_{i}^{m} \times \mathcal{I}_{j}^{m}, n+1) = \Xitwo(\mathcal{I}_{i}^{m} \times \mathcal{I}_{j}^{m}, n) \geq \hat{\mathcal{U}}^{m}_{n}((i, j)) =  \hat{\mathcal{U}}^{m}_{n+1}((i, j)).\]
Also, in this case
\[
\eta_{n+1}(\ell') = \eta_{n}(\ell') + g(W_s, W_{n+1}) \geq \theta_{n}(\ell') + \gmin{\ell', k} = \theta_{n+1}(\ell'),
\]
and similarly, 
\[
\eta_{n+1}(k) = \eta_{n}(k) + h(W_{n+1}) \geq \theta_{n}(k) + \hmin{k} = \theta_{n+1}(k),
\]
so that Equation~\eqref{eq:eta-ineq} is satisfied. 
Finally, in this case,
\[
\cZ_{n+1} = \cZ_{n} + g(W_s, W_{n+1}) + h(W_{n+1}) \leq  \mathbf{a} \cdot \hat{\mathcal{U}}^{m}_n + \gmax{\ell',k} + \hmax{k} = \mathbf{a} \cdot \hat{\mathcal{U}}^{m}_{n+1}.
\]
Meanwhile, in the second case $\Xitwo(\mathcal{I}_{\ell'}^{m}\times \mathcal{I}_{k}^{m}, n)$ and 
$\eta_{n}(\ell')$ increase, while $\sum_{j=1}^{\dimurn} \hat{\mathcal{U}}^{m}_{n}((\ell',j))$ and $\theta_{n}(\ell')$ remain the same, and thus \eqref{eq:xi-relation-1} is satisfied and $\eta_{n+1}(\ell') \geq \theta_{n+1}(\ell')$. As this is the only case when $\Xitwo(\mathcal{I}_{\ell'}^{m} \times \mathcal{I}_{k}^{m}, n) - \hat{\mathcal{U}}^{m}_{n}((\ell',k))$ increases, and we add a ball of type $(\dimurn+1, k)$, \eqref{eq:xi-relation-2} also follows. Both $\eta_{n}(k)$ and $\theta_{n}(k)$ increase as in the first case. Next, 
\[
\cZ_{n+1} = \cZ_{n} + g(W_s, W_{n+1}) + h(W_{n+1}) \leq  \mathbf{a} \cdot \hat{\mathcal{U}}^{m}_n + g^{*}_{\max{}}(k) + \hmax{k} = \mathbf{a} \cdot \hat{\mathcal{U}}^{m}_{n+1}.
\]
As all other quantities remain the same, Equation~\eqref{eq:eta-ineq} is satisfied,  
and moreover, $\cZ_{n+1} \leq \mathbf{a} \cdot \hat{\mathcal{U}}^{m}_{n+1}$. To complete the proof, it remains to prove the following claim. 

\begin{clm} \label{clm:coup1}
For each $m\in \mathbb{N}$, almost surely (on the coupling space), the urn process $\hat{\mathcal{U}}^{m} = (\hat{\mathcal{U}}^{m}_{n})_{n \in \mathbb{N}_{0}}$ is distributed like the P\'{o}lya urn process $(\mathcal{U}^{m}_{n})_{n\in \mathbb{N}_{0}}$ with $\mathcal{U}^{m}_{0}$ consisting of an initial ball $\ell \in R$.
\end{clm}

\begin{proof}
First note that, since $W_0$ is sampled from $\mu$, conditionally on the positive probability event $\left\{h(W_0) > 0\right\}$, we have \[\Prob{W_0 \in \mathcal{I}^{m}_{\ell}, h(W_0) > 0} \leq \Prob{W_0 \in \mathcal{I}^{m}_{\ell}} = p^{m}_{\ell},\]
and thus, $\mathbb{P}$-a.s., we have $W_0 \in \mathcal{I}^{m}_\ell$ with $p^{m}_{\ell} > 0$. This, combined with the fact that $0 < h(W_0) \leq \hmax{\ell}$, implies that $\mathbb{P}$-a.s., the initial ball $\ell \in R$.

Now, note that in every step in $(\hat{\mathcal{U}}^{m}_{n})_{n \in \mathbb{N}_{0}}$, we add a ball of type $k$ for $k \in \left[\dimurn\right]$ with probability $p^{m}_{k}$, which is the same as in $(\mathcal{U}^{m}_{n})_{n \in \mathbb{N}_{0}}$. Moreover, given $\hat{\mathcal{U}}^{m}_{n}$, the probability of adding balls of type $(k,\ell)$ is 
\[
p^{m}_{\ell} \left(\frac{\eta_{n}(k)}{\cZ_{n}}
\times \frac{\theta_{n}(k) \cZ_{n}}{\eta_n(k) \mathbf{a} \cdot \hat{\mathcal{U}}^{m}_n} \right) = p^{m}_{\ell} \frac{\theta_{n}(k)}{\mathbf{a} \cdot \hat{\mathcal{U}}^{m}_n},
\]
which also agrees with the P\'{o}lya urn scheme. Finally, the probability of adding a ball of type $(\dimurn+1, \ell)$ is 
\[
p^{m}_{\ell} \sum_{j=1}^{\dimurn} \left[\left(1 - \frac{\theta_{n}(j) \cZ_n}{\eta_n(j) \mathbf{a} \cdot \hat{\mathcal{U}}^{m}_n} \right) \frac{\eta_{n}(j)}{\cZ_n} \right]
= p^{m}_{\ell} \left(1 - \sum_{j=1}^{\dimurn} \frac{\theta_{n}(j)}{\mathbf{a} \cdot \hat{\mathcal{U}}^{m}_n}\right),
\]
as required.
\end{proof}
\end{proof}

Note also, that, since the functions $h^{+}, g^{+}$ are non-increasing pointwise in $m$, on the coupling we have that for any fixed $n$, $\mathbf{a} \cdot \mathcal{U}^{m}_n$ is non-increasing in $m$. 
Combining this result with Corollary~\ref{cor:slln-polya-coupling}, we have the following corollary.
\begin{cor} \label{cor:lim-lambda}
The sequence $(\lambda_{m})_{m \in \mathbb{N}}$ is non-increasing in $m$. In particular, there exists a limit $\lambda_{\infty} \geq 0$ such that 
\[\lambda_{m} \downarrow \lambda_{\infty}\]
as $m \to \infty$. 
\end{cor}

\subsubsection{The Limiting Vectors of Urn Schemes associated with Urn~E} \label{sss:analysis-urn-scheme}
We now calculate the limiting vector $\mathbf{u}_{m}$ and the limiting eigenvalue $\lambda_{m}$. 
First note that by the definition of the urn process, for each $n \in \mathbb{N}_{0}$, $\ell \in \left[\dimurn\right]$ we have that
$\mathcal{U}^{m}_{n+1}(\ell) - \mathcal{U}^{m}_{n}(\ell)$ is Bernoulli distributed with parameter $p^m_{\ell}$. Thus, by the strong law of large numbers and Corollary~\ref{cor:slln-polya-coupling}, we have, for each $\ell \in \left[\dimurn\right]$, 
\begin{align} \label{eq:init-balls1}
\mathbf{u}_{m}(\ell) = \frac{p^m_{\ell}}{\lambda_m}.
\end{align}
Next, for any $i,j \in \left[\dimurn\right]$ using the definitions of $\bfgamm$ and $\mathbf{a}$ (\eqref{eq:bfgamm-def} and \eqref{eq:act-def}) we have 
\begin{align} \label{eq:rein-balls1}
    \nonumber \lambda_{m} \mathbf{u}_{m}((i,j)) & = p^m_{j} \sum_{\ell=1}^{\dimurn}(\bfgamm \, \mathbf{a} \mathbf{u}_{m})((i, \ell)) + p^m_{j} (\bfgamm \, \mathbf{a} \mathbf{u}_{m})(i) 
    \\ \nonumber &= p^m_{j} \sum_{\ell=1}^{\dimurn} \gmin{i,\ell} \mathbf{u}_{m}((i,\ell)) 
    + p^m_j \hmin{i} \mathbf{u}_{m}(i)
    \\ & \stackrel{\eqref{eq:init-balls1}}{=} p^m_{j} \sum_{\ell=1}^{\dimurn} \gmin{i,\ell} \mathbf{u}_{m}((i,\ell)) + \frac{p^m_{j} p^m_i \hmin{i}}{\lambda_m}.
\end{align}
We now define
\[
\mathcal{A}_{i} := \sum_{\ell=1}^{\dimurn} \gmin{i,\ell} \mathbf{u}_{m}((i,\ell)).
\]
Multiplying both sides of Equation~\eqref{eq:rein-balls1} by $\gmin{i, j}$ and taking the sum over $j \in \left[\dimurn\right]$ (and recalling the definition of $\tilde{g}_{-}(i)$ in \eqref{eq:var-defs}), we get
\begin{align*}
    \lambda_{m} \mathcal{A}_{i} & = \left(\mathcal{A}_{i} + \frac{p^m_i \hmin{i}}{\lambda_m}\right) \sum_{j=1}^{\dimurn} p^m_{j} \gmin{i,j}
    \\ &= \left(\mathcal{A}_{i} + \frac{p^m_i \hmin{i}}{\lambda_m}\right) 
    \tilde{g}_{-}(i).
\end{align*}
Thus, solving for $\mathcal{A}_{i}$
\begin{align} \label{eq:a-j}
    \mathcal{A}_i = \frac{p^m_i \hmin{i} \tilde{g}_{-}(i)}{\lambda_{m}(\lambda_{m} - \tilde{g}_{-}(i))}. 
\end{align}
Substituting Equation~\eqref{eq:a-j} into Equation~\eqref{eq:rein-balls1}, we have
\begin{align} \label{eq:rein-balls2}
\nonumber  \lambda_{m} \mathbf{u}_{m}((i,j)) & = p^m_{j} \left(\frac{p^m_i \hmin{i} \tilde{g}_{-}(i)}{\lambda_{m}(\lambda_{m} - \tilde{g}_{-}(i))} + \frac{p^m_i \hmin{i}}{\lambda_{m}} \right)
    \\ & = p^m_{j} \frac{p^m_{i}\hmin{i}}{\lambda_{m} - \tilde{g}_{-}(i)}.
\end{align}
Meanwhile, for each $j \in \left[\dimurn\right]$ we have 
\begin{align} \label{eq:rubb1}
\nonumber    \lambda_{m} \mathbf{u}_{m}((\dimurn+1,j)) & = p^m_j \left(\sum_{\ell=1}^{\dimurn} (\mathbf{a} \mathbf{u}_{m})((\dimurn+1, \ell))
+    \sum_{i=1}^{\dimurn} \sum_{\ell=1}^{\dimurn} (\mathbf{a} - \bfgamm \, \mathbf{a})((i, \ell))  + \sum_{i=1}^{\dimurn} (\mathbf{a} - \bfgamm \, \mathbf{a})(i)\right)
    \\ \nonumber & = p^m_j \left(\sum_{\ell=1}^{\dimurn} g^{*}(\ell)\mathbf{u}_{m}((\dimurn+1, \ell)) 
+    \sum_{i=1}^{\dimurn} \sum_{\ell=1}^{\dimurn} (\gmax{i,\ell} - \gmin{i,\ell})
\mathbf{u}_{m}((i,\ell)) \right. \\ \nonumber  & \hspace{4cm} \left.+ \sum_{i=1}^{\dimurn} (\hmax{i} - \hmin{i}) \mathbf{u}_{m}(i)
    \right)
    \\ & =: p^m_j \left(\mathcal{B}_{m} + \mathcal{E}_{m}\right);
\end{align}
where, in the last equation we set
\[\mathcal{B}_{m} := \sum_{\ell=1}^{\dimurn} g^{*}(\ell)\mathbf{u}_{m}((\dimurn+1, \ell))\]
and
\[\mathcal{E}_{m} := \sum_{i=1}^{\dimurn} \sum_{\ell=1}^{\dimurn} (\gmax{i,\ell} - \gmin{i,\ell})
\mathbf{u}_{m}((i,\ell)) + \sum_{i=1}^{\dimurn} (\hmax{i} - \hmin{i}) \mathbf{u}_{m}(i). \]
Multiplying both sides of \eqref{eq:rubb1} by $g^{*}(j)$ and taking the sum over $j$, we have
\[
\lambda_{m} \mathcal{B}_{m} = \left(\sum_{j=1}^{\dimurn} p^m_j g^{*}(j)\right)(\mathcal{B}_{m} + \mathcal{E}_{m}) = \tilde{g}^{*}_{+}(\mathcal{B}_{m} + \mathcal{E}_{m})
\]
and thus
\begin{align} \label{eq:cal-b1}
\mathcal{B}_{m} = \frac{\tilde{g}^{*}_{+}}{\lambda_{m} - \tilde{g}^{*}_{+}} \mathcal{E}_{m}.
\end{align}
We now apply Condition~\hyperlink{c1}{\textbf{C1}} in the following lemma (all of the previous analysis implicitly applied \hyperlink{c1}{\textbf{C2}}):

\begin{lem} \label{lem:lambda-bound}
Assume Conditions~\hyperlink{c1}{\textbf{C1}} and \hyperlink{c1}{\textbf{C2}}. Then, we have $\lambda_{\infty} : = \lim_{m \to \infty} \lambda_{m} > \tilde{g}^{*}$. 
\end{lem}

\begin{proof}
Note that, since we add two balls to the urn at each time-step, we have \[\|\mathcal{U}^{m}_{n+1}\|_{1} - \|\mathcal{U}^{m}_{n}\|_{1} = 2.\] Thus, by Equation~\eqref{eq:lim-urn-comp}, we have $\|\lambda_{m} \mathbf{u}_{m}\|_{1} = 2$. Now, by Equation~\eqref{eq:init-balls1}, we have $\lambda_{m}\sum_{\ell=1}^{\dimurn}  \mathbf{u}_{m}(\ell) = 1$, and thus, by Equation~\eqref{eq:rein-balls2}, we have
\begin{align}
    \sum_{j=1}^{\dimurn} \sum_{i=1}^{\dimurn}  \lambda_{m} \mathbf{u}_{m}((i,j)) = \E{\frac{\hmin{r\left(W\right)}}{\lambda_{m} - \tilde{g}_{-}\left(r\left(W\right)\right)}} \leq 1.
\end{align}
Note that as $m \to \infty$, $\hmin{r\left(W\right)} \uparrow h(W)$ and $\tilde{g}_{-}\left(r\left(W\right)\right) \uparrow \tilde{g}(W)$. Thus, by the monotone convergence theorem, we have 
\[
\E{\frac{h(W)}{\lambda_{\infty} - \tilde{g}(W)}} =  \lim_{m \to \infty} \E{\frac{\hmin{r\left(W\right)}}{\lambda_{m} - \tilde{g}_{-}\left(r\left(W\right)\right)}} \leq 1. 
\]
Now, since the eigenvectors $\mathbf{u}_{m}$ are non-negative, by Equation~\eqref{eq:cal-b1}, we have 
\[\lambda_{m} \geq \tilde{g}^{*}_{+},\]
and thus, $\lambda_{\infty} = \lim_{m \to \infty} \lambda_{m} \geq \lim_{m \to \infty} \tilde{g}^{*}_{+} = \tilde{g}^{*}$. But, if $\lambda_{\infty} = \tilde{g}^{*}$, since the expression in \eqref{eq:cond-c1} is decreasing in $\lambda^{*}$, we would have a contradiction to Condition~\hyperlink{c1}{\textbf{C1}}. The result follows. 
\end{proof}

\begin{lem} \label{lem:condensate-dies}
Assume Conditions~\hyperlink{c1}{\textbf{C1}} and \hyperlink{c1}{\textbf{C2}}. Then, we have $\mathcal{B}_{m} \downarrow 0$ and $\mathcal{E}_{m} \downarrow 0$ as $m \to \infty$. In particular, 
\begin{align} \label{eq:lambda-inf-malth}
    \E{\frac{h(W)}{\lambda_{\infty} - \tilde{g}(W)}} = 1,
\end{align}
so that $\lambda_{\infty} = \lambda^{*}$.
\end{lem}

\begin{proof}
First, note that by Corollary~\ref{cor:lim-lambda} and Lemma~\ref{lem:lambda-bound}, for each $m \in \mathbb{N}$, we have $\lambda_{m} \geq \lambda_{\infty} > \tilde{g}^{*}$. Combining this fact with the boundedness of $g$ and $h$ we observe that 
\[
\sup_{x,y \in [0,\wmax]}\left\{\frac{h(x)}{\lambda_{m}\left(\lambda_{m} - \tilde{g}(x)  \right)}, \frac{1}{\lambda_m}\right\}  < \sup_{x,y \in [0,\wmax]}\left\{\frac{h(x)}{\tilde{g}^{*}\left(\lambda_{\infty} - \tilde{g}(x) \right)}, \frac{1}{\lambda_\infty}\right\}
= : C < \infty,
\]
where the bound on the right is independent of $m$. Now, given $\eps > 0$, by applying Lemma~\ref{lem:unif-conv}, let $m$ be sufficiently large that 
\[
\sup_{(x,y) \in [0,\wmax] \times [0,\wmax]} \left(g^{+}(x,y) - g^{-}(x,y) \right) < \frac{\eps}{2C}
\quad \text{ and } \sup_{x \in [0,\wmax]} \left(h^{+}(x) - h^{-}(x) \right) < \frac{\eps}{2C}.
\]
Then we have
\begin{align*}
\mathcal{E}_{m} &= \sum_{i=1}^{\dimurn} \sum_{j=1}^{\dimurn} \left(\gmax{i,j} - \gmin{i,j}\right)
\mathbf{u}_{m}((i, j))  + \sum_{\ell=1}^{\dimurn} \left(\hmax{\ell} - \hmin{\ell}\right) \mathbf{u}_{m}(\ell)
\\ & \stackrel{\eqref{eq:init-balls1}, \eqref{eq:rein-balls2}}{=} \sum_{i=1}^{\dimurn} \sum_{j=1}^{\dimurn} \left(\gmax{i,j} - \gmin{i,j}\right) \frac{h_{\min}(i) p^m_i p^m_j}{\lambda_{m}(\lambda_m - \tilde{g}_{-}(i))}
 + \sum_{\ell=1}^{\dimurn} \left(\hmax{\ell} - \hmin{\ell}\right)\frac{p^m_{\ell}}{\lambda_{m}}
 \\& < \frac{\eps}{2C} \cdot C \left(\sum_{i=1}^{\dimurn} \sum_{j=1}^{\dimurn} p^{m}_{i} p^{m}_{j}\right)
+ \frac{\eps}{2C} \cdot C\left(\sum_{\ell=1}^{\dimurn} p^{m}_{\ell}\right) = \eps.
\end{align*}
 The result for $\mathcal{B}_{m}$ then follows from the fact that $\tilde{g}^{*}_{+} \downarrow \tilde{g}^{*}$, and Lemma~\ref{lem:lambda-bound}.
\end{proof}

We are now ready to prove our main results of this subsection.
\subsubsection{Proof of Theorem~\ref{th:conv-parti-c1}}
\label{sub:conv-parti-proof}
\begin{proof}
Note that, by Equation~\eqref{eq:parti-ineq-poly1} from Proposition~\ref{prop:coupling-tree-with-urn}, we have  
\[
0 \leq \mathbf{a} \cdot \mathcal{U}^{m}_{n} - \cZ_n \leq 
\left(\mathbf{a} - \bfgamm \mathbf{a}\right)\cdot \mathcal{U}^{m}_{n}. 
\]
Dividing by $n$ and taking limits as $n \to \infty$, by Equation~\eqref{eq:lambda-em} we have
\begin{align*}
0 \leq  \lambda_{m} - \limsup_{n \to \infty} \frac{\cZ_n}{n}
\leq \lambda_{m} - \liminf_{n \to \infty} \frac{\cZ_n}{n}
\leq \limsup_{n \to \infty} \left(\left(\mathbf{a} - \bfgamm \mathbf{a}\right)\cdot \frac{\mathcal{U}^{m}_{n}}{n}\right) = \mathcal{B}_{m} + \mathcal{E}_{m}.
\end{align*} 
The result follows by applying Lemma~\ref{lem:condensate-dies}. 
\end{proof}

In addition, recalling the definition of $\mathscr{I}^{m}$ from Equation~\eqref{eq:building-blocks}, note that 
\begin{align} \label{eq:sigma-building-blocks}
\sigma(\mathscr{I}^{m}) = \left\{S \subseteq [0,\wmax] : S = \bigcup_{i \in I} \mathcal{I}^{m}_{i}, I \subseteq [\dimurn]\right\};
\end{align}
(i.e. the $\sigma$-algebra generated by $\mathscr{I}^{m}$ is the set of finite unions of sets in $\mathscr{I}^{m}$).
Recalling that $\mathscr{I}^{m_2}$ is a refined partition of $\mathscr{I}^{m_1}$ for $m_1 < m_2$, by Lemma~\ref{lem:refined-partitions-partition} we have
\begin{align} \label{eq:set-fam-increase}
\sigma(\mathscr{I}^{m_1}) \subseteq \sigma(\mathscr{I}^{m_2}). 
\end{align}
We now prove Theorem~\ref{th:weak-conv-edge}.

\subsubsection{Proof of Theorem~\ref{th:weak-conv-edge}}
\label{sub:weak-conv-edge-proof}
\begin{proof}
We begin by proving the result for Cartesian products of the form $S \times S'$ with $S, S' \in \sigma(\mathscr{I}^{m'})$, for $m' \in \mathbb{N}$. 
Note that, by the definition of $\Xitwo(\cdot,n)$, we clearly have finite \textit{additivity}, that is, for any $S_1, S_2, S_3 \in  \mathscr{B}$ if $S_1 \cap S_2 = \varnothing$, we have \[\Xitwo((S_1 \cup S_2) \times S_3, n) = \Xitwo(S_1 \times S_3, n) + \Xitwo(S_2 \times S_3, n), \quad \text{and similarly,}\]
\[\Xitwo(S_3 \times (S_1 \cup S_2), n) = \Xitwo(S_3 \times S_1, n) + \Xitwo(S_3 \times S_2, n).\] Combining these facts with Proposition~\ref{prop:coupling-tree-with-urn}, Corollary~\ref{cor:slln-polya-coupling} and Equation~\eqref{eq:rein-balls2}, for sets $S \times S'$ with $S, S' \in \sigma(\mathscr{I}^{m'})$ we have, for each $m > m'$,
\begin{align*}
\E{\frac{h^{-}(W)}{\lambda_{m} - \tilde{g}_{-}\left(r\left(W\right)\right)} \mathbf{1}_{S}} \mu(S') & \leq \liminf_{n \to \infty} \frac{\Xitwo(S \times S',n)}{n} \\ &\leq \limsup_{n \to \infty} \frac{\Xitwo(S\times S',n)}{n}  \leq \E{\frac{h^{-}(W)}{\lambda_{m} - \tilde{g}_{-}\left(r\left(W\right)\right)} \mathbf{1}_{S}} \mu(S') + \mathcal{B}_{m} + \mathcal{E}_{m}.
\end{align*}
Taking limits as $m \to \infty$ and applying Lemma~\ref{lem:condensate-dies}, this proves the result for this family of sets. 
\\
Now, by the Portmanteau Theorem, we need only prove that for all sets $U \in  \mathcal{O}$ (where $\mathcal{O}$ denotes the class of open subsets of $[0, \wmax] \times [0,\wmax]$), 
\begin{align} \label{eq:port-open}
\liminf_{n \to \infty} \frac{\Xitwo(U,n)}{n} \geq (\psi_{*}\mu \times \mu)(U).
\end{align}
Now, let 
\begin{align} \label{eq:open-approx}
\mathcal{I}^{m}(U) := \bigcup_{i,j \in \left[\dimurn\right]: \mathcal{I}^{m}_{i} \times \mathcal{I}^{m}_{j} \subseteq U} \mathcal{I}^{m}_{i} \times \mathcal{I}^{m}_{j}.
\end{align}
Note that, since $U$ is open, and $\mathscr{I}^{m}$ is fine enough that the set of dyadic intervals $\left\{\mathcal{D}^{m}_{i}(\wmax)\right\}_{i \in [2^{m}]} \subseteq \sigma(\mathscr{I}^{m})$, we have
\begin{align} \label{eq:open-approx-two}
 \mathbf{1}_{\mathcal{I}^{m}(U)} \uparrow \mathbf{1}_{U} \quad \text{pointwise as $m \to \infty$.}
\end{align}
In addition, since $\mathcal{I}^{m}(U)\subseteq U$, for each $m \in \mathbb{N}$
\begin{align*}
    (\psi_{*}\mu \times \mu)(\mathcal{I}^{m}(U)) = \liminf_{n \to \infty} \frac{\Xitwo(\mathcal{I}^{m}(U),n)}{n} \leq \liminf_{n \to \infty} \frac{\Xitwo(U,n)}{n}.
\end{align*}
Equation~\eqref{eq:port-open} then follows by taking limits as $m \to \infty$. 
\end{proof}

\subsection{Urn~D} \label{subsec:urn-d}
In order to analyse the degree distribution in this model under Conditions~\hyperlink{c1}{\textbf{C1}} and \hyperlink{c1}{\textbf{C2}}, we introduce another collection of P\'{o}lya urns $(\mathcal{V}^{\truncurn}_{n})_{n \in \mathbb{N}_{0}}$, which not only depend on $m$, but also depends on a parameter $\truncurn \in \mathbb{N}$. These may be regarded as finite approximations of Urn~D. 
For brevity of notation, wherever possible in this subsection we will omit the dependence of these parameters on $m$.  For $i \in \mathbb{N}$, define $[\dimurn]^{i}$ so that \[[\dimurn]^{i}:= \left\{(u_0, \ldots u_{i-1}): u_0, \ldots, u_{i-1} \in [\dimurn]\right\}.\]
Now, we set
\[\mathcal{B}':= \left(\bigcup_{i=1}^{\truncurn + 1} [\dimurn]^{i}\right) \cup (\{\dimurn+1\} \times [\dimurn]).\]
The urn process $(\mathcal{V}^{\truncurn}_{n})_{n\geq 0}$ is then a vector-valued stochastic process taking values in $\mathbb{V}_{\mathcal{B}'}$. 
We now define the vectors $\actvectwo$, $\bfgammtwo$ associated with the urn process such that 
\begin{align} \label{eq:actvectwo-def}
\actvectwo(x) = \begin{cases}
h_{\max{}}(u_0) + \sum_{j=1}^{k} g_{\max{}}(u_0,u_j) & \text{if } x = (u_0, \ldots, u_k) \in [\dimurn]^{k+1} \\
g^{*}_{\max{}}(\ell) & \text{if } x = (\dimurn+1, \ell); 
\end{cases}
\end{align}
and, 
\begin{align} \label{eq:bfgammtwo-def}
\bfgammtwo(x) =
\begin{cases}
\frac{h_{\min{}}(u_0) + \sum_{j=1}^{k} g_{\min{}}(u_0,u_j)}{h_{\max{}}(u_0) + \sum_{j=1}^{k} g_{\max{}}(u_0,u_j)}, & \text{if } x = (u_0, \ldots, u_k) \in [\dimurn]^{k+1}, k < \truncurn, \actvectwo(x) > 0; \\
0, & \text{otherwise}.
\end{cases}
\end{align}
Now, given $\mathbf{u} = (u_0, \ldots, u_k) \in [\dimurn]^{k+1}$, $k < \truncurn$, 
and $\ell \in [\dimurn]$, we define their \textit{concatenation} $(\mathbf{u}, \ell) \in [\dimurn]^{k+2}$ such that 
\[(\mathbf{u}, \ell) := (u_0, \ldots, u_k, \ell). \]
Then, we define the replacement matrix $M'$ of the urn $(\mathcal{V}^{\truncurn}_{n})_{n \in \mathbb{N}_{0}}$ such that, given $x, x' \in \mathcal{B}'$,  
\[
M'_{x', x} = \begin{cases}
-(\bfgammtwo\actvectwo)(x) & \text{if } x' = x, x \in [\dimurn]^{k}, k \leq \truncurn; \\
(\bfgammtwo \actvectwo)(x) p^m_{\ell}, & \text{if } x' = (x, \ell), \ell \in [\dimurn], x \in \mathcal{B}';\\
(\actvectwo - \bfgammtwo \actvectwo)(x) p^m_{\ell}, & \text{if } x' = (\dimurn+1, \ell), \ell \in [\dimurn], x \in \mathcal{B}'; \\ 
\actvectwo(x) p^m_{\ell}, & \text{if } x' = \ell, x \in \mathcal{B}';\\ 
0 & \text{otherwise.}
\end{cases}
\]

Again, note that it may be the case that $M'$ is not irreducible, if either $\actvectwo(x) = 0$ for certain $x \in \mathcal{B}'$ or $p^{m}_{\ell} = 0$ for certain choices of $\ell$. Nevertheless, we define the sets 
\begin{align*}
    \mathscr{U}'_1 := \left\{x \in \mathcal{B}': M'_{x'x} = 0 \; \forall x' \in \mathcal{B}' \right\} = \left\{x \in \mathcal{B}': \actvectwo(x) = 0\right\},
\end{align*}
and 
\begin{align*}
    \mathscr{U}'_2 := \left\{x' \in \mathcal{B}': M'_{x'x} = 0 \; \forall x \in \mathcal{B}' \setminus \{x'\} \right\}.
\end{align*}
Again, we assume that $\mathscr{U}'_1 \cap \mathscr{U}'_2 = \varnothing$; if not, we replace $\mathscr{U}'_1$ by $\mathscr{U}'_1 \setminus \mathscr{U}'_2$.
We then set $R' = \mathcal{B}' \setminus (\mathscr{U}'_1 \cup \mathscr{U}'_2)$, and let $M'_{R'}$ be the restriction of $M'$ to $R'$. As before, $M'_{R'}$ satisfies the conditions of Lemma~\ref{lem:jans-irred}, and thus has a unique largest positive eigenvalue $\lambda'_{R'}$ with corresponding eigenvector $\mathbf{V}_{R'}$. But then, writing $M'$ in block form in a manner analogous to the previous subsection, 
$M$ has the same largest positive eigenvalue, with corresponding right eigenvector given (in block form) by 
\[
\mathbf{V}_{\truncurn} = \begin{bmatrix}
           \mathbf{V}_{R'} \\
           \left(\lambda'_{R'}\right)^{-1} A' \mathbf{V}_{R'} \\
           0
         \end{bmatrix}.
\]
Here, we assume $\mathbf{V}_{\truncurn}$ is normalised so that $\actvectwo \cdot \mathbf{V}_{\truncurn} = 1$. Also in a manner similar to the previous subsection, assuming we begin with a ball of type $x \in R'$, one readily verifies that the restriction of $M'$ to $R'$ and $\mathscr{U}'_1$ satisfies conditions (A1)-(A6) of Subsection~\ref{subsec:gen-polya}, and also, that for each $x \in \mathscr{U}'_2$ and $n \in \mathbb{N}_{0}$, $\mathcal{U}_{n}(x) = 0$ almost surely. Therefore, applying Theorem~\ref{thm:polyurnconvergence} again, we have the following corollary:

\begin{cor} \label{cor:slln-polya-coupling-2}
With $\mathbf{V}_{\truncurn}, \lambda'_{\truncurn}$ and $R'$ as defined above, assuming we begin with a ball $x \in R'$, we have 
\begin{align} \label{eq:lim-urn-comp-2}
    \frac{\mathcal{V}^{\truncurn}_n}{n} \xrightarrow{n \to \infty} \lambda'_{\truncurn} \mathbf{V}_{\truncurn}
\end{align}
almost surely. In particular, we have 
\begin{align} \label{eq:lambda-em-2}
    \frac{\mathbf{a} \cdot \mathcal{V}^{\truncurn}_{n}}{n} \xrightarrow{n \to \infty} \lambda'_{\truncurn}.
\end{align}
\end{cor}

As in the previous subsection, in the coupling below, the assumption of a ball $x \in R'$ is met by the tree process being initiated by a vertex $0$ with weight $W_0$ sampled at random from $\mu$ and satisfying $h(W_0) > 0$.


\subsubsection{Coupling Urn D with the Tree Process}
 Recall that we denote by $N_{\geq k}(B, n)$ the number of vertices of out-degree at least $k$ having weight belonging to $B \in \mathscr{B}$. We also define the analogue $\urndeg_{\geq k}(j, \mathcal{V}^{\truncurn}_{n})$ for $n \in \mathbb{N}_{0}$ and $j \in [\dimurn]$ such that 
 \begin{align} \label{eq:poly-tail-dist} 
\urndeg_{\geq k}(j, n) := \sum_{j=k}^{\truncurn +1} \sum_{\mathbf{u}_{j} \in [\dimurn]^{j}} \mathcal{V}^{\truncurn}_{n}(\mathbf{u}_{j}) \mathbf{1}_{\{j\}}(u_0).
 \end{align}
 This represents the number of balls in the urn $\mathcal{V}^{\truncurn}_{n}$ with type $\mathbf{u} = (u_0, \ldots)$ having dimension at least $k + 1$, with $u_0 =j$.
We then have the following analogue of Proposition~\ref{prop:coupling-tree-with-urn}:
\begin{prop}
There exists a coupling $(\hat{\mathcal{V}}^{\truncurn}_{n}, \hat{\cT}_{n})_{n \in \mathbb{N}_{0}}$ of the P\'{o}lya urn process $(\mathcal{V}^{\truncurn}_{n})_{n \in \mathbb{N}_{0}}$ and the tree process $(\cT_n)_{n \in \mathbb{N}_{0}}$ such that, almost surely (on the coupling space), $\mathcal{V}^{\truncurn}_{0}$ consists of a single ball $\ell \in R'$ and for all $n \in \mathbb{N}_{0}$,  $k \in \{0\} \cup [\truncurn]$, we have
\begin{align} \label{eq:coup2deg}
    & \urndeg_{\geq k}(j, n) \leq N_{ \geq k}\left(\mathcal{I}^{m}_{j}, n\right) \quad \text{and} 
    \\ \label{eq:coup2deg2} &  \sum_{j=1}^{\dimurn} \left(N_{ \geq k}\left(\mathcal{I}^{m}_{j}, n\right) - \urndeg_{\geq k}(j, n)\right)  \leq \sum_{j=1}^{\dimurn}\hat{\mathcal{V}}^{\truncurn}_{n}((\dimurn+1, j).  
\end{align} 
In addition, we have 
\begin{align} \label{eq:coup2partiineq}
    (\bfgammtwo \actvectwo) \cdot \hat{\mathcal{V}}^{\truncurn}_{n} \leq \cZ_{n} \leq \actvectwo \cdot \hat{\mathcal{V}}^{\truncurn}_{n}. 
\end{align}
\end{prop}

\begin{proof}
We proceed in a somewhat similar manner to Proposition~\ref{prop:coupling-tree-with-urn}, however, in this case, we first introduce a ``labelled'' P\'{o}lya urn $\left(\mathcal{L}_{n}\right)_{n \geq 0}$ where balls carry \textit{integer labels} from $\left\{-\dimurn, \ldots, 0, \ldots, n\right\}$. In addition, for $j \in \{0\} \cup [n]$, the label is independent of the \textit{type} of the ball: we denote by $b_{j}(n)$ the \textit{type} of a ball with label $j$ at time $n$.
One may interpret the ball with label $j$ as representing the evolution of vertex $j$ in the tree process - in this sense, the label may be interpreted as a ``time-stamp''. Balls of type $(\dimurn+1,j), \, j \in [\dimurn]$, however, are labelled $-j$ - we denote by $d_j = d_{j}(n)$ the number of balls with this label (since here, multiple balls may share the same label). We describe the labelled urn process $\mathcal{L}_{n}$ as an evolving vector in $\mathcal{B}' \times \mathbb{Z}$, so that $\mathcal{L}_{n} = \sum_{j=1}^{\dimurn}  d_{j} \cdot \delta_{(b_{j}(n), j)} + \sum_{i=0}^{n}\delta_{(b_{i}(n), i)}$. We set
\[\actvectwo (\mathcal{L}_{n}) = \sum_{j=-\dimurn}^{-1}  d_{j} \cdot\actvectwo(b_{j}(n)) +  \sum_{i=0}^{n} \actvectwo(b_{i}(n)), \; \text{and } (\bfgammtwo \actvectwo) (\mathcal{L}_{n}) = \sum_{i=0}^{n} (\bfgammtwo  \actvectwo)(b_{i}(n)). \]
Now, we use $\mathcal{L}_{n+1}$ to define $\hat{\mathcal{V}}^{\truncurn}_{n+1}$ by ``forgetting'' labels, so that, 
\[\text{if } \mathcal{L}_{n+1} = \sum_{j=-\dimurn}^{-1}  d_{j} \cdot \delta_{(b_{j}(n+1), j)} + \sum_{i=0}^{n+1}\delta_{(b_{i}(n+1), i)}, \text{ we set } \hat{\mathcal{V}}^{\truncurn}_{n+1} = \sum_{j=-\dimurn}^{-1}  d_{j} \cdot \delta_{b_{j}(n+1)} + \sum_{i=0}^{n+1}\delta_{b_{i}(n+1)}. \]
\\
Sample the entire tree process $(\hat{\cT}_n)_{n \in \mathbb{N}_{0}}$. 
If, at time $0$, the tree consists of a single vertex $0$ with weight $W_0 \in I^{m}_{\ell}$ then, we set $\mathcal{L}_{0} = \delta_{(\ell,0)}$, and note that we have
\[(\bfgammtwo \actvectwo) (\mathcal{L}_{0}) = h_{\min{}}(\ell) \leq h(W_0) = \cZ_{0} \leq \actvectwo(\mathcal{L}_{0}) = h_{\max{}}(\ell), \]
and 
\[f(N^{+}(0,\hat{\cT}_{0})) = h(W_0) \geq \left( \bfgammtwo \, \actvectwo \right)(b_0(0)) = h_{\min{}}(\ell).\]
Now, assume inductively that after $n$ steps in the process, for each $i \in \{0\}\cup[n]$ we have  
\begin{align} \label{eq:fitness-mon}
& f(N^{+}(i, \hat{\cT}_{n})) \geq \left( \bfgammtwo \, \actvectwo \right)(b_i(n)),
\quad \deg^{+}(i, \mathcal{T}_{n}) \geq \dim(b_{i}(n)) - 1, \quad \\ \label{eq:diff-rubbish-urn2} & \quad \sum_{i=0}^{n} 
\left(\deg^{+}(i, \mathcal{T}_{n}) - \dim(b_{i}(n)) + 1 \right) 
=  \sum_{j=1}^{\dimurn}\hat{\mathcal{V}}^{\truncurn}_{n}((\dimurn+1, j),
\end{align}
 and Equation~\eqref{eq:coup2partiineq} is satisfied.
 
 Let $s$ be the vertex sampled in the tree in the $(n+1)$st step, assume that $r(s) = \ell'$
and that $r(n+1) = k$. Then, for the $(n+1)$th step in the urn:
sample an independent random variable $U_{n+1}$ uniformly distributed on $[0,1]$. Then:
\begin{itemize}
    \item If $\dim{(b_{s}(n))} \leq \truncurn$ and $U_{n+1} \leq \frac{(\bfgammtwo \actvectwo) (b_{s}(n)) \cZ_{n}}{f(N^{+}(s, \hat{\cT}_{n}))\actvectwo(\mathcal{L}_{n})}$, remove the ball $(b_{s}(n), s)$ from the urn, and add balls $((b_{s}(n), k), s)$ and $(k,n+1)$ to the urn (i.e. set $\mathcal{L}_{n+1} = \mathcal{L}_{n} + \delta_{((b_{s}(n), \ell), s)} + \delta_{(k,n+1)} - \delta_{(b_{s}(n), s)}$).
    We call this step Case 1.  
    \item Otherwise, add balls of type $\left((\dimurn+1, k),-k\right), (k,n+1)$ - we call this Case 2. 
\end{itemize}

First note that 
\begin{align*} 
(\bfgammtwo \actvectwo) (b_{s}(n+1))  - (\bfgammtwo \actvectwo)(b_{s}(n)) & = 
\begin{cases}
g_{\min{}}(\ell',k), & \text{ in Case 1} \\
0, & \text{ in Case 2}
\end{cases}
\\ & \leq  
g(W_s, W_{n+1}) = 
f(N^{+}(s, \hat{\cT}_{n+1})) - f(N^{+}(s, \hat{\cT}_{n})),
\end{align*}
and likewise
\[
(\bfgammtwo \actvectwo) (b_{n+1}(n+1)) = h_{\min{}}(\ell) \leq h(W_{n+1}) = f(N^{+}(n+1, \hat{\cT}_{n+1})).
\]
 Additionally, in Case~1 the dimension of $b_{s}(n)$ and the degree of $s$ in $\hat{\mathcal{T}}_{n}$ both increase, whilst in Case~2 only the degree of $s$ increases whilst the dimension of $b_{s}(n)$ remains the same. 
This proves Equation~\eqref{eq:fitness-mon} at time $n+1$. In addition, Case~2 coincides with the addition of a ball of type $(\dimurn + 1, \ell)$, which yields Equation~\eqref{eq:diff-rubbish-urn2}. 
 Finally,
\begin{align*}
(\bfgammtwo \actvectwo)\cdot\left(\hat{\mathcal{V}}^{\truncurn}_{n+1} - \hat{\mathcal{V}}^{\truncurn}_{n} \right)
& = \begin{cases}
h_{\min{}}(k) + g_{\min{}}(\ell',k), & \text{in Case 1} \\
h_{\min{}}(k), & \text{in Case 2}
\end{cases} \\
& \leq h(W_{n+1}) + g(W_{s}, W_{n+1}) = \cZ_{n+1} - \cZ_{n}  
\\ & \leq 
\begin{cases}
h_{\max{}}(k) + g_{\max{}}(\ell',k), & \text{in Case 1} \\
h_{\max{}}(k) + g_{\max{}}^{*}(k), & \text{in Case 2}
\end{cases}
\\ & \leq (\actvectwo)\cdot(\hat{\mathcal{V}}^{\truncurn}_{n+1} - \hat{\mathcal{V}}^{\truncurn}_{n});
\end{align*}
which shows that Equation~\eqref{eq:coup2partiineq} is also satisfied at time $n+1$.

\begin{clm}
Almost surely (on the coupling space), the urn process $\hat{\mathcal{V}}^{\truncurn} = (\hat{\mathcal{V}}^{\truncurn}_{n})_{n \in \mathbb{N}_{0}}$ is distributed like the P\'{o}lya urn $(\mathcal{V}^{\truncurn}_{n})_{n \in \mathbb{N}_{0}}$ with $\mathcal{V}^{\truncurn}_{0}$ consisting of an initial ball $\ell \in R'$.
\end{clm}

\begin{proof}
The fact that, $\mathbb{P}$-a.s., the initial ball $\ell \in R'$ follows immediately from the fact that the initial weight $W_0$ is sampled from $\mu$ conditionally on the event $\{h(W_0) > 0 \}$ (analogous to Claim~\ref{clm:coup1}). Moreover, in every step in $\hat{\mathcal{V}}^{\truncurn}$, we add a ball of type $k$ for $k \in \left[\dimurn\right]$ with probability $p^m_{k}$, which is the same as in $\mathcal{V}^{\truncurn}$. Furthermore, given $\hat{\mathcal{V}}^{\truncurn}_{n}$ the probability of removing a ball of type $\mathbf{u}$ with $\dim{\mathbf{u}} \leq \truncurn$ and adding a ball of type $(\mathbf{u}, \ell)$ is 
\begin{align*}
p^m_{\ell} \sum_{s \in \mathcal{L}_{n}: b_{s}(n) = \mathbf{u}}  \frac{(\bfgammtwo \actvectwo) (b_{s}(n)) \cZ_{n}}{f(N^{+}(s, \hat{\cT}_{n}))\actvectwo(\mathcal{L}_{n})} \times \frac{f(N^{+}(s, \hat{\cT}_{n}))}{\cZ_{n}} &= p^m_{\ell} \sum_{s \in \mathcal{L}_{n}: b_{s}(n) = \mathbf{u}} \frac{(\bfgammtwo \actvectwo) (b_{s}(n))}{\actvectwo(\mathcal{L}_{n})} \\ &= p^m_{\ell} \frac{\hat{\mathcal{V}}^{\truncurn}_{n}(\mathbf{u})}{\cZ_{n}},
\end{align*}
which also agrees with the transition law of the P\'{o}lya urn scheme $\mathcal{V}$. 
Finally, the probability of adding a ball of type $(\dimurn+1, \ell)$ is 
\begin{align*}
p^m_{\ell} \sum_{s \in \mathcal{L}_{n} : \dim{b_{s}(n)} > \truncurn} \frac{f(N^{+}(s, \hat{\cT}_{n}))}{\cZ_{n}} 
&+ 
p^m_{\ell} \sum_{s \in \mathcal{L}_{n} : \dim{b_{s}(n)} \leq \truncurn}
\left(1 - \frac{(\bfgammtwo \actvectwo) (b_{s}(n)) \cZ_{n}}{f(N^{+}(s, \hat{\cT}_{n}))\actvectwo(\mathcal{L}_{n})}\right) \times \frac{f(N^{+}(s, \hat{\cT}_{n}))}{\cZ_{n}} 
\\ &= p^m_{\ell} \sum_{s \in \mathcal{L}_{n}}
\left(\frac{f(N^{+}(s, \hat{\cT}_{n})}{\cZ_{n}} \right) - p^m_{\ell} \sum_{s \in \mathcal{L}_{n} : \dim{b_{s}(n)} \leq \truncurn} \frac{(\bfgammtwo \actvectwo) (b_{s}(n)) }{\actvectwo(\mathcal{L}_{n})}
\\ &= p^m_{\ell}\left(1 -  \sum_{\mathbf{u} \in \hat{\mathcal{V}}^{\truncurn}_{n}: \dim{\mathbf{u}} \leq \truncurn}
\frac{(\bfgammtwo \actvectwo) (\hat{\mathcal{V}}^{K}(\mathbf{u})) }{\actvectwo(\hat{\mathcal{V}}^{K}_{n})}\right),
\end{align*}
which agrees with transition rule of $\mathcal{V}^{\truncurn}$.
\end{proof}
Finally, to complete the proof, we verify the following claim.
\begin{clm}
For all $n \in \mathbb{N}_{0}$, Equations~\eqref{eq:coup2deg} and \eqref{eq:coup2deg2} are satisfied for all $k \in \{0\} \cup [\truncurn]$.
\end{clm}

\begin{proof}
If we define $b_{i}(n)|_{0}$ such that 
 $b_{i}(n)|_{0} = x_0$ if
 $b_{i}(n) = (x_0, \ldots, x_{k})$, then, by construction of the labelled urn process $(\mathcal{L}_{n})_{n \in \mathbb{N}_{0}}$,  $b_{i}(n)|_{0} = x_0 \implies r(W_{i}) = x_0$, so that $W_{i} \in \mathcal{I}^{m}_{x_0}$. Therefore, for each $k \in \{0\} \cup [\truncurn], j \in [\dimurn]$, 
 \begin{align*}
    \urndeg_{\geq k}(j, n) & = \sum_{b_{i}(n): \dim(b_{i}(n)) \geq k+1} \mathbf{1}_{\{j\}}(b_{i}(n)|_{0})  \stackrel{\eqref{eq:fitness-mon}}{\leq}
    \sum_{i: \deg^{+}(i, \hat{\mathcal{T}}_{n}) \geq k} \mathbf{1}_{\mathcal{I}^{m}_j}(W_{i}) = N_{\geq k}(\mathcal{I}^{m}_j,n).
 \end{align*}
 Moreover, by \eqref{eq:diff-rubbish-urn2}, 
 \begin{align*}
     \sum_{j=1}^{\dimurn}\hat{\mathcal{V}}^{\truncurn}_{n}((\dimurn+1, j) & = 
     \sum_{i=0}^{n} \left(\deg^{+}(i, \hat{\mathcal{T}}_{n}) - \dim(b_{i}(n)) + 1\right)
    \\ & = \sum_{k=0}^{n} \sum_{j=1}^{\dimurn} \left(\left(N_{\geq k}\left(\mathcal{I}^{m}_{j}, n\right) - \urndeg_{\geq k}(j, n)\right)\right),
 \end{align*}
 which implies~\eqref{eq:coup2deg2}.
\end{proof}
\end{proof}

\subsubsection{Analysis of the P\'{o}lya urn  \texorpdfstring{$\mathcal{V}^{\truncurn}$}{TEXT}}
We now calculate the limiting vector $\mathbf{V}_{K}$ and limiting eigenvalue $\lambda'_{K}$ of the P\'{o}lya urn scheme $(\mathcal{V}^{\truncurn}_{n})_{n \geq 0}$. 
We first introduce some more notation: for any vector $\mathbf{u} = (u_0, \ldots, u_{k-1}) \in [\dimurn]^{k}$, and $i \in [k]$, denote by $\mathbf{u}|_{i} := (u_0, \ldots, u_{i-1}) \in [\dimurn]^{i}$.
We also define the following quantities: 
\begin{align} \label{eq:rubbish1}
\mathcal{R}_{\truncurn} : = \sum_{\ell = 1}^{\dimurn}  \actvectwo((\dimurn + 1, \ell)) \mathbf{V}_{\truncurn}((\dimurn + 1, \ell)), \quad
\mathcal{E}_{\truncurn} := \sum_{\mathbf{u}: \dim{\mathbf{u}} \leq \truncurn} (\actvectwo - \bfgammtwo \actvectwo)(\mathbf{u}) \mathbf{V}_{\truncurn}(\mathbf{u}), 
\end{align}
and 
\begin{align} \label{eq:rubbish2}
\mathcal{F}_{\truncurn} := \sum_{\mathbf{v}:\dim{\mathbf{v}} =  \truncurn + 1} \actvectwo(\mathbf{v})
    \mathbf{V}_{\truncurn}(\mathbf{v}). 
\end{align}
\begin{prop} \label{prop:eigen-urn2}
Let $\lambda'_{\truncurn}$ and $\mathbf{V}_{\truncurn}$ denote the limiting leading eigenvalue and corresponding right eigenvector of $M'$, respectively. Then, $\mathbf{V}_{\truncurn}$ satisfies
\begin{align} \label{eq:limit-eigenvector-urn2}
    \lambda'_{\truncurn}\mathbf{V}_{\truncurn}(x) = \begin{cases}
    p^m_{\ell}, & x = \ell \in [\dimurn];\\
    p^m_{u_0}
    \frac{\lambda'_{\truncurn}}{\lambda'_{\truncurn} + (\bfgammtwo \actvectwo)(\mathbf{u})}
    \prod_{i=1}^{k} \left[p^m_{u_{i}}\left(\frac{
    (\bfgammtwo \actvectwo)(\mathbf{u}|_{i})}{(\bfgammtwo \actvectwo)(\mathbf{u}|_{i}) + \lambda'_{\truncurn}}\right)\right], & x = \mathbf{u} = (u_0, \ldots, u_{k}) \in [\dimurn]^{k+1}, k < \truncurn; \\
    p^m_{u_0} p^m_{u_\truncurn}
    \prod_{i=1}^{\truncurn} \left[p^m_{u_{i}}\left(\frac{(\bfgammtwo \actvectwo)(\mathbf{u}|_{i})}{(\bfgammtwo \actvectwo)(\mathbf{u}|_{i}) + \lambda'_{\truncurn}}\right)\right], & x = \mathbf{u} = (u_0, \ldots, u_{\truncurn}) \in [\dimurn]^{\truncurn+1}.
    \end{cases}
\end{align}
In addition, we have
\begin{align} \label{eq:rubbish-polya2} 
    \mathcal{R}_{\truncurn} = \frac{\mathcal{E}_{\truncurn} + \mathcal{F}_{\truncurn}}{\lambda'_{\truncurn} - g^{*}_{+}}.
\end{align}
\end{prop}
\begin{proof}
First note that, as before, for each $\ell \in [\dimurn]$, since we add a ball of type $\ell$ with probability $p^{m}_{\ell}$ at each time-step, we have 
\begin{align} \label{eq:starter-balls}
\lambda'_{\truncurn} \mathbf{V}_{\truncurn}(\ell) = p^m_{\ell},
\end{align}
this implies the first case in Equation~\eqref{eq:limit-eigenvector-urn2}. 
Next, we have 
\begin{align} \label{eq:replacement-urn2}
    \lambda'_{\truncurn} \mathbf{V}_{\truncurn}(\mathbf{u}) = \begin{cases}
    p^m_{u_{k}}(\bfgammtwo \actvectwo)(\mathbf{u}|_{k}) \mathbf{V}_{\truncurn}(\mathbf{u}|_{k}) - (\bfgammtwo \actvectwo)(\mathbf{u})\mathbf{V}_{\truncurn}(\mathbf{u}), & \mathbf{u} \in [\dimurn]^{k+1}, \, k \in [\truncurn -1]; \\
    p^m_{u_{\truncurn}} (\bfgammtwo \actvectwo)  \mathbf{V}_{\truncurn }(\mathbf{u}|_{\truncurn}); & \mathbf{u} \in [\dimurn]^{\truncurn + 1};
    \end{cases}
\end{align}
so that, if $\mathbf{u} \in [\dimurn]^{k + 1}, \, k \in [\truncurn -1]$, 
\begin{align} \label{eq:replacement2-urn2}
  \mathbf{V}_{\truncurn}(\mathbf{u}) = 
 \frac{p^{m}_{u_k}(\bfgammtwo \actvectwo)(\mathbf{u}|_{k}) \mathbf{V}_{\truncurn}(\mathbf{u}|_{k})}{\lambda'_{\truncurn} + (\bfgammtwo \actvectwo)(\mathbf{u})}. 
\end{align}
Applying Equations~\eqref{eq:replacement-urn2} and \eqref{eq:replacement2-urn2}, recursing backwards, and using the fact that $\mathbf{V}_{\truncurn}(u_0) = p^m_{u_0}/ \lambda'_{\truncurn}$, completes the proof of Equation~\eqref{eq:limit-eigenvector-urn2}.  
Finally, for each $j \in [\dimurn]$, we have 
\begin{align} \label{eq:rubbish-polya-two}
    \lambda_{\truncurn}'\mathbf{V}_{\truncurn}((\dimurn + 1, j)) =& p^m_{j} \left(\sum_{\ell = 1}^{\dimurn} \actvectwo((\dimurn + 1, \ell)) \mathbf{V}_{\truncurn}((\dimurn + 1, \ell)) 
   \right. \left.+ \sum_{\mathbf{u}:\dim{\mathbf{u}} \leq \truncurn} (\actvectwo - \bfgammtwo \actvectwo)(\mathbf{u}) \mathbf{V}_{\truncurn}(\mathbf{u}) \right.\\
  \nonumber & \hspace{1cm}\left. + \sum_{\mathbf{v}:\dim{\mathbf{v}} =  \truncurn + 1} \actvectwo(\mathbf{v})
  \mathbf{V}_{\truncurn}(\mathbf{v})\right)  \\
  \nonumber & = p^m_{j} \left(\mathcal{R}_{\truncurn} + \mathcal{E}_{\truncurn} + \mathcal{F}_{\truncurn}\right);
\end{align}
where, in the last equation we recall the definitions in Equations~\eqref{eq:rubbish1} and \eqref{eq:rubbish2}. 
Now, multiplying both sides of Equation~\eqref{eq:rubbish-polya-two} by $\actvectwo((\dimurn + 1,j)) = g^{*}(j)$ and taking the sum over $j$, we have
\[
\lambda'_{\truncurn} \mathcal{R}_{\truncurn} = \left(\sum_{j=1}^{\dimurn} p^m_{j} g^{*}(j)\right)\left(\mathcal{R}_{\truncurn} + \mathcal{E}_{\truncurn} + \mathcal{F}_{\truncurn}\right) = \tilde{g}^{*}_{+} \left(\mathcal{R}_{\truncurn} + \mathcal{E}_{\truncurn} + \mathcal{F}_{\truncurn}\right). 
\]
Rearranging this proves Equation~\eqref{eq:rubbish-polya2}, thus completing the proof of the proposition. 
\end{proof}

Now, we recall the definition of the \textit{companion process} $(S_{i}(w))_{i \geq 0}$ from Subsection~\ref{subsec:main-results}: 
Recall that $W_1, W_2, \ldots$ were defined to be independent $\mu$-distributed random variables and let $w \in [0, \wmax]$. We then define the random process $(S_{i}(w))_{i \geq 0}$ inductively so that
\begin{align*}
S_0(w) := h(w);  \quad S_{i+1}(w) := S_{i}(w) + g(w, W_{i+1}), \; i \geq 0.
\end{align*}
Now, we also define the \textit{lower companion process} $(S^{-}_{i}(w))_{i \geq 0}$ in a similar way, but with functions $h^{-}, g^{-}$ instead. 

\begin{lem} \label{lem:fail-balls}
Assume Conditions~\hyperlink{c1}{\textbf{C1}} and \hyperlink{c1}{\textbf{C2}}. Then we have 
\[ 
\lim_{\truncurn \to \infty} \lim_{m \to \infty} \mathcal{F}_{\truncurn} = 0.
\]
\end{lem}
\begin{proof}
Note that by Equation~\eqref{eq:limit-eigenvector-urn2}, with $\maxurn$ being an upper bound on $\max\{h,g\}$,  we have
\begin{align*}
    \mathcal{F}_{\truncurn} = \sum_{\mathbf{u}:\dim{\mathbf{u}} = \truncurn +1} \actvectwo(\mathbf{u}) \mathbf{V}_{\truncurn}(\mathbf{u}) & = \sum_{\mathbf{u}:\dim{\mathbf{u}} = \truncurn +1} \actvectwo(\mathbf{u}) \, p^m_{u_0}
    \prod_{i=1}^{\truncurn} \left[p^m_{u_{i}}\left(\frac{(\bfgammtwo \actvectwo)(\mathbf{u}|_{i})}{(\bfgammtwo \actvectwo)(\mathbf{u}|_{i}) + \lambda'_{\truncurn}}\right)\right]
    \\ & \leq \maxurn (\truncurn +1) \cdot \sum_{\mathbf{u}:\dim{\mathbf{u}} = \truncurn +1} p^m_{u_0}
    \prod_{i=1}^{\truncurn} \left[p^m_{u_{i}}\left(\frac{(\bfgammtwo \actvectwo)(\mathbf{u}|_{i})}{(\bfgammtwo \actvectwo)(\mathbf{u}|_{i}) + \lambda'_{\truncurn}}\right)\right]
    \\ & = \maxurn (\truncurn +1) \cdot \E{\prod_{i=0}^{\truncurn-1} \left(\frac{S^{-}_{i}(W)}{S^{-}_{i}(W) + \lambda'_{\truncurn}}\right)}.
\end{align*}
Now, note that for all $m \in \mathbb{N}$,  $S^{-}(W)$ is stochastically bounded above by $S(W)$, and by Theorem~\ref{th:conv-parti-c1} and Equations~\eqref{eq:lambda-em-2} and \eqref{eq:coup2partiineq} $\lambda'_{\truncurn}$ is bounded below by $\lambda^{*}$ uniformly in $m$ and $\truncurn$. Therefore, since the function $x \mapsto \frac{x}{x+\lambda}$ is increasing in $x$ and decreasing in $\lambda$, 
we may bound the previous display above so that 
\begin{align*}
    \maxurn(\truncurn +1) \cdot \E{\prod_{i=0}^{\truncurn-1} \left(\frac{S^{-}_{i}(W)}{S^{-}_{i}(W) + \lambda'_{\truncurn}}\right)} & \leq \maxurn(\truncurn +1) \cdot \E{\prod_{i=0}^{\truncurn-1} \left(\frac{S_{i}(W)}{S_{i}(W) + \lambda'_{\truncurn}}\right)}
    \\ & \leq \maxurn(\truncurn +1) \cdot \E{\prod_{i=0}^{\truncurn-1} \left(\frac{S_{i}(W)}{S_{i}(W) + \lambda^{*}}\right)}.
\end{align*}
We complete the proof by proving the following claim. 
\begin{clm}
We have 
\[\lim_{k \to \infty} k \cdot \E{\prod_{i=0}^{k-1} \left(\frac{S_{i}(W)}{S_{i}(W) + \lambda^{*}}\right)} = 0\]
\end{clm}
\begin{proof}
First observe that
\begin{align*}
\E{\prod_{i=0}^{\infty}\left(\frac{S_{i}(W)}{S_{i}(W) + \lambda^{*}}\right)} & \leq \prod_{i=1}^{\infty}\left(\frac{\maxurn i}{\maxurn i + \lambda^{*}}\right) = \prod_{i=0}^{\infty}\left(1 - \frac{\lambda^{*}}{\maxurn i + \lambda^{*}}\right) \leq e^{- \sum_{i=1}^{\infty} \frac{\lambda^{*}}{\maxurn i + \lambda^{*}}} = 0.
\end{align*}
Therefore, we have
\begin{align*}
    k \cdot \E{\prod_{i=0}^{k-1} \left(\frac{S_{i}(W)}{S_{i}(W) + \lambda^{*}}\right)} & = k \cdot \sum_{j=k}^{\infty} \E{\left(1 - \frac{S_{j}(W)}{S_{j}(W) + \lambda^{*}}\right)\prod_{i=0}^{j-1} \left(\frac{S_{i}(W)}{S_{i}(W) + \lambda^{*}}\right)}
    \\ & =  k \cdot \sum_{j=k}^{\infty} \E{\frac{\lambda^{*}}{S_{j}(W) + \lambda^{*}}\prod_{i=0}^{j-1} \left(\frac{S_{i}(W)}{S_{i}(W) + \lambda^{*}}\right)}
    \\ & \leq  \sum_{j=k}^{\infty} j \cdot \E{\frac{\lambda^{*}}{S_{j}(W) + \lambda^{*}}\prod_{i=0}^{j-1} \left(\frac{S_{i}(W)}{S_{i}(W) + \lambda^{*}}\right)}.
\end{align*}
The series on the right of the previous display consists of non-negative terms, and for each $N \in \mathbb{N}$, we have
\begin{align} \label{eq:series-comparison}
    & \sum_{j=1}^{N} j \cdot \E{\frac{\lambda^{*}}{S_{j}(W) + \lambda^{*}}\prod_{i=0}^{j-1} \left(\frac{S_{i}(W)}{S_{i}(W) + \lambda^{*}}\right)} \\ & = \sum_{j=1}^{N} \left(j\cdot\E{\prod_{i=0}^{j-1} \left(\frac{S_{i}(W)}{S_{i}(W) + \lambda^{*}}\right)}  - j\cdot\E{\prod_{i=0}^{j} \left(\frac{S_{i}(W)}{S_{i}(W) + \lambda^{*}}\right)}  \right)
    \\ & = \sum_{j=1}^{N} \E{\prod_{i=0}^{j-1} \left(\frac{S_{i}(W)}{S_{i}(W) + \lambda^{*}}\right)}  - N \cdot\E{\prod_{i=0}^{N} \left(\frac{S_{i}(W)}{S_{i}(W) + \lambda^{*}}\right)}
     \\ & \leq \sum_{j=1}^{N} \E{\prod_{i=0}^{j-1} \left(\frac{S_{i}(W)}{S_{i}(W) + \lambda^{*}}\right)}.
\end{align}
Now, note that by Lemma~\ref{lem:conv-sum}, we have
\[\sum_{j=1}^{\infty} 
\E{\prod_{i=0}^{j-1} \left(\frac{S_{i}(W)}{S_{i}(W) + \lambda^{*}}\right)} < \infty, 
\]
and thus by \eqref{eq:series-comparison} and the monotone convergence theorem, we also have \[\sum_{j=1}^{\infty} j \cdot \E{\frac{\lambda^{*}}{S_{j}(W) + \lambda^{*}}\prod_{i=0}^{j-1} \left(\frac{S_{i}(W)}{S_{i}(W) + \lambda^{*}}\right)} < \infty.\] Therefore, 
\begin{align*}
\lim_{k \to \infty} k \cdot \E{\prod_{i=0}^{k-1} \left(\frac{S_{i}(W)}{S_{i}(W) + \lambda^{*}}\right)} \leq \lim_{k \to \infty} \sum_{j=k}^{\infty} j \cdot \E{\frac{\lambda^{*}}{S_{j}(W) + \lambda^{*}}\prod_{i=0}^{j-1} \left(\frac{S_{i}(W)}{S_{i}(W) + \lambda^{*}}\right)} = 0.
\end{align*}
\end{proof}
\end{proof}

\begin{lem} \label{lem:error-balls}
Assume Conditions~\hyperlink{c1}{\textbf{C1}} and \hyperlink{c1}{\textbf{C2}}. Then we have 
\begin{align} \label{eq:error-balls-urn-two} 
\lim_{\truncurn \to \infty} \lim_{m \to \infty} \mathcal{E}_{\truncurn} = 0, \quad \text{ and } \quad \lim_{\truncurn \to \infty} \lim_{m \to \infty} \mathcal{R}_{\truncurn} = 0.
\end{align}
In addition, 
\begin{align} \label{eq:lam-lim2}\lim_{\truncurn \to \infty} \lim_{m \to \infty}  \lambda'_{\truncurn} = \lambda^{*}.\end{align}
\end{lem}

\begin{proof}
The proof is similar to that of Lemma~\ref{lem:condensate-dies}. First, let $\eps > 0$ be given, and, by Lemma~\ref{lem:unif-conv}, let $m$ be sufficiently large that 
\begin{equation} \label{eq:bound1111}
\sup_{(x,y) \in [0,\wmax] \times [0,\wmax]} \left(g^{+}(x,y) - g^{-}(x,y) \right) < \frac{\eps \lambda^{*}}{2(\truncurn)^2}
\quad \text{ and } \sup_{x \in [0,\wmax]} \left(h^{+}(x) - h^{-}(x) \right) < \frac{\eps \lambda^{*}}{2(\truncurn)^2}.
\end{equation}
Now, we have
\begin{align*}
\mathcal{E}_{\truncurn} & = 
\sum_{\ell \in [\dimurn]} \left(\hmax{\ell} - \hmin{\ell}\right) \mathbf{V}_{\truncurn}(\ell)
+ \sum_{\mathbf{u}\in [\dimurn]^{\truncurn}} \sum_{i=2}^{\truncurn}
\left((\actvectwo - \bfgammtwo \actvectwo)(\mathbf{u}|_{i})\right)\mathbf{V}_{\truncurn}(\mathbf{u}|_{i})
\\ & \stackrel{\eqref{eq:limit-eigenvector-urn2}}{=} \sum_{\ell \in [\dimurn]} \left(\hmax{\ell} -\hmin{\ell}\right) \frac{p^{m}_{\ell}}{\lambda'_{\truncurn}}
+ \sum_{\mathbf{u}\in [\dimurn]^{\truncurn}} \sum_{i=2}^{\truncurn}
\left((\actvectwo - \bfgammtwo \actvectwo)(\mathbf{u}|_{i})\right)\mathbf{V}_{\truncurn}(\mathbf{u}|_{i}).
\end{align*}
By applying Equation~\eqref{eq:limit-eigenvector-urn2} again we may write the previous equation as
\begin{align} \label{eq:randeq2}
\nonumber &= \sum_{\ell \in [\dimurn]} \left(\hmax{\ell} -\hmin{\ell}\right) \frac{p^{m}_{\ell}}{\lambda'_{\truncurn}} + \sum_{\mathbf{u}\in [\dimurn]^{\truncurn}} \sum_{i=2}^{\truncurn}
    \frac{\left((\actvectwo - \bfgammtwo \actvectwo)(\mathbf{u}|_{i})\right) p^m_{u_0}}{\lambda'_{\truncurn} + (\bfgammtwo \actvectwo)(\mathbf{u}|_{i})}
    \prod_{j=1}^{i} \left[ p^m_{u_{j}}\left(\frac{
    (\bfgammtwo \actvectwo)(\mathbf{u}|_{j})}{(\bfgammtwo \actvectwo)(\mathbf{u}|_{j}) + \lambda'_{\truncurn}}\right) \right]
\\& \hspace{3cm} \stackrel{\eqref{eq:bound1111}}{<} \frac{\eps \lambda^{*}}{2} \left(\sum_{\ell \in [\dimurn]} \frac{p^{m}_{\ell}}{\lambda'_{\truncurn}}\right)
+ \sum_{\mathbf{u}\in [\dimurn]^{\truncurn}} \sum_{i=2}^{\truncurn}\frac{\left((\actvectwo - \bfgammtwo \actvectwo)(\mathbf{u}|_{i})\right) }{\lambda'_{\truncurn}}p^m_{u_0} \prod_{j=1}^{i} p^{m}_{u_j}.
\end{align}
Also note that, by Equation~\eqref{eq:bound1111}, for any $\mathbf{u} = (u_0, \ldots, u_{\truncurn-1}) \in [\dimurn]^{\truncurn}$, and each $i \in \left\{2, \ldots, \truncurn\right\}$ we have
\[
(\actvectwo - \bfgammtwo \actvectwo)(\mathbf{u}|_{i}) = \hmax{u_0} - \hmin{u_0} + \sum_{j=1}^{i-1} \left(\gmax{u_0, u_j} - \gmin{u_0, u_j}\right) <  \frac{\eps \lambda^{*}}{2 (\truncurn)^2} \cdot \truncurn = \frac{\eps \lambda^{*}}{2 \truncurn}.
\]
In addition, noting that uniformly in $\truncurn$ and $m$ we have $\lambda'_{\truncurn} \geq \lambda^{*}$ (as in the proof of Lemma~\ref{lem:fail-balls}) 
and thus, we may bound \eqref{eq:randeq2} by 
\[
\frac{\eps}{2} + \frac{\eps}{2 \truncurn} \left(  \sum_{i=2}^{\truncurn} \sum_{\mathbf{u}\in [\dimurn]^{\truncurn}} p^m_{u_0} \prod_{j=1}^{i} p^{m}_{u_j}\right) < \frac{\eps}{2} + \frac{\eps}{2 \truncurn}\cdot \truncurn =  \eps. 
\]
First sending $m \to \infty$, $\eps \to 0$, and $\truncurn \to \infty$
implies the first equation in~\eqref{eq:error-balls-urn-two}. 
Next, Equation~\eqref{eq:rubbish-polya2}, Lemma~\ref{lem:fail-balls},  and the facts that $\lambda'_{\truncurn} \geq \lambda^{*}$ and  $\lim_{m \to \infty} \tilde{g}^{*}_{+} = \tilde{g}^{*} < \lambda^{*}$ together imply the second limit in~\eqref{eq:error-balls-urn-two}. Finally, by Equation~\eqref{eq:coup2partiineq} and Proposition~\ref{prop:eigen-urn2} we have  
\begin{align*}
    \lambda'_{\truncurn} - \lambda^{*}
    & \leq \mathcal{E}_{\truncurn} + \sum_{\dim{\mathbf{u}} = \truncurn + 1} (\actvectwo - \bfgammtwo\actvectwo)(\mathbf{u}) + \mathcal{R}_{\truncurn}
    \\ & \leq \mathcal{E}_{\truncurn} + \mathcal{F}_{\truncurn} + \mathcal{R}_{\truncurn}.
\end{align*}
Equation~\eqref{eq:lam-lim2} then follows by taking limits as $m \to \infty$ and $\truncurn \to \infty$.
\end{proof}

\subsubsection{Proof of Theorem~\ref{th:degree-dist}} \label{subsec:proof-of-theorem-deg-dist}
\begin{proof}
First (recalling the definition of $\urndeg_{\geq k} (\cdot,n)$ from \eqref{eq:poly-tail-dist}), by Proposition~\ref{prop:eigen-urn2} for any $\ell \in [\dimurn]$ we have
\begin{align} 
\nonumber    \lim_{n \to \infty} \frac{\urndeg_{\geq k} (\ell,n)}{n} &= \sum_{j=k+1}^{\truncurn + 1}\sum_{\mathbf{u} \in [\dimurn]^{\truncurn + 1}}  \mathbf{V}_{\truncurn}(\mathbf{u}|_{j}) \mathbf{1}_{\{\ell\}}(u_0)
    \\ \nonumber & = \sum_{\mathbf{u} \in [\dimurn]^{\truncurn + 1}} \left(p^{m}_{u_0} p^{m}_{u_{\truncurn}} \prod_{i=1}^{\truncurn} \left[p^{m}_{u_i} \left(\frac{(\bfgammtwo \actvectwo)(\mathbf{u}|_{i})}{(\bfgammtwo \actvectwo)(\mathbf{u}|_{i}) + \lambda'_{\truncurn}}\right)\right]
    \right. \\ \nonumber & \hspace{2cm} \left.+  \sum_{j=k+1}^{\truncurn} p^m_{u_0}
    \frac{\lambda'_{\truncurn}}{\lambda'_{\truncurn} + (\bfgammtwo \actvectwo)(\mathbf{u}|_{j})}
    \prod_{i=1}^{j-1} \left[p^m_{u_{i}}\left(\frac{
    (\bfgammtwo \actvectwo)(\mathbf{u}|_{i})}{(\bfgammtwo \actvectwo)(\mathbf{u}|_{i}) + \lambda'_{\truncurn}}\right)\right] \right)\mathbf{1}_{\{\ell\}}(u_0).
\end{align}
Now, by the definitions of the functions $g^{-}, h^{-}$ and the definition of expectation, we may write the last equation  as
\begin{align}     \label{eq:polya-tail-limit}
\nonumber & = 
     \E{\prod_{i=0}^{\truncurn - 1}\left(\frac{S^{-}_{i}(W)}{S^{-}_{i}(W) + \lambda'_{\truncurn}}\right)\mathbf{1}_{\mathcal{I}^{m}_{\ell}}} + \sum_{j=k+1}^{\truncurn}\E{\frac{\lambda'_{\truncurn}}{S^{-}_{j-1}(W) + \lambda'_{\truncurn}}\prod_{i=0}^{j-2} \left(\frac{S^{-}_{i}(W)}{S^{-}_{i}(W) + \lambda'_{\truncurn}}\right) \mathbf{1}_{\mathcal{I}^{m}_{\ell}}}
    \\  \nonumber &= 
     \E{\prod_{i=0}^{\truncurn - 1}\left(\frac{S^{-}_{i}(W)}{S^{-}_{i}(W) + \lambda'_{\truncurn}}\right)\mathbf{1}_{\mathcal{I}^{m}_{\ell}}}  + \sum_{j=k+1}^{\truncurn}\E{\left(1 - \frac{S^{-}_{j-1}(W)}{S^{-}_{j-1}(W) + \lambda'_{\truncurn}}\right)\prod_{i=0}^{j-2} \left(\frac{S^{-}_{i}(W)}{S^{-}_{i}(W) + \lambda'_{\truncurn}}\right) \mathbf{1}_{\mathcal{I}^{m}_{\ell}}}
     \\ & = \E{\prod_{i=0}^{k - 1}\left(\frac{S^{-}_{i}(W)}{S^{-}_{i}(W) + \lambda'_{\truncurn}}\right)\mathbf{1}_{\mathcal{I}^{m}_{\ell}}}.
\end{align}
For $m' \in \mathbb{N}$, \eqref{eq:polya-tail-limit} allows us to prove the result for sets $S \in \sigma(\mathscr{I}^{m'})$ (where we recall the definition of $\mathscr{I}^{m'}$ in Equation~\eqref{eq:building-blocks}, and the facts that $\sigma(\mathscr{I}^{m'})$ consists of finite unions, and is increasing in $m'$ - Equations~\eqref{eq:sigma-building-blocks} and \eqref{eq:set-fam-increase}). Since $N(\cdot, n)$ is finitely additive,
if $S \in \sigma(\mathscr{I}^{m})$, by Equation~\eqref{eq:coup2deg} and 
Equation~\eqref{eq:polya-tail-limit} 
we have 
\begin{align*}
    \E{\prod_{i=0}^{k - 1}\left(\frac{S^{-}_{i}(W)}{S^{-}_{i}(W) + \lambda'_{\truncurn}}\right)\mathbf{1}_{S}(W)} & \leq \liminf_{n \to \infty} \frac{N_{\geq k}(S,n)}{n} \leq \limsup_{n \to \infty} \frac{N_{\geq k}(S,n)}{n} \\ &\leq \E{\prod_{i=0}^{k - 1}\left(\frac{S^{-}_{i}(W)}{S^{-}_{i}(W) + \lambda'_{\truncurn}}\right)\mathbf{1}_{S}(W)} + \mathcal{R}_{\truncurn} + \mathcal{E}_{\truncurn} + \mathcal{F}_{\truncurn}.
\end{align*}
Taking limits as $m \to \infty$ and then as $\truncurn \to \infty$, and applying Lemma~\ref{lem:fail-balls} and Lemma~\ref{lem:error-balls} now proves the result for sets in $\sigma(\mathscr{I}^{m'})$. Now, note that for each $k \in \mathbb{N}_{0}$, and measurable sets $S' \in \mathscr{B}$,  we have 
\begin{align} \label{eq:mu-bound}
    \limsup_{n \to \infty} \frac{N_{\geq k}(S')}{n} \leq \limsup_{n \to \infty} \frac{N_{\geq 0}(S')}{n} = \mu(S'),
\end{align}
where the last equality applies the strong law of large numbers. \\
We now prove the result for sets $U \in \mathcal{O}$ where $\mathcal{O}$ denotes the class of all open subsets of $[0,\wmax]$. 
For a fixed open set $U \in \mathcal{O}$, and $m \in \mathbb{N}$, recall that $\mathcal{I}^{m}(U) := \bigcup_{j \in \left[\dimurn\right]: \mathcal{I}^{m}_{j} \subseteq U} \mathcal{I}^{m}_{j}$. 
Also recall Equation~\eqref{eq:open-approx-two}, which states that
$\mathbf{1}_{\mathcal{I}^{m}(U)} \uparrow \mathbf{1}_{U}$ pointwise as $m \to \infty$.
Now, since each $\mathcal{I}^{m}(U) \in \sigma(\mathscr{I}^{m})$, by applying Equation~\eqref{eq:mu-bound} for each $k \leq \truncurn$ we have 
\begin{align*}
\E{\prod_{i=0}^{k - 1}\left(\frac{S_{i}(W)}{S_{i}(W) + \lambda'_{\truncurn}}\right)\mathbf{1}_{\mathcal{I}^{m}(U)}} & \leq \liminf_{n \to \infty} \frac{N_{\geq k}(U,n)}{n} \leq \limsup_{n \to \infty} \frac{N_{\geq k}(U,n)}{n} \\ & \leq \E{\prod_{i=0}^{k - 1}\left(\frac{S_{i}(W)}{S_{i}(W) + \lambda'_{\truncurn}}\right)\mathbf{1}_{\mathcal{I}^{m}(U)}} + \mu(U\setminus \mathcal{I}^{m}(U)).
\end{align*}
Taking limits as $m \to \infty$ and then $\truncurn \to \infty$ now proves the result for sets belonging to $\mathcal{O}$. 
\\
Finally, note that since $\mu$ is a \textit{regular} measure, for any $A \in \mathscr{B}$ we have
\begin{align*}
    \mu(A) = \inf_{U \in \mathcal{O}: A \subseteq U}\left\{\mu(U) \right\}.
\end{align*}
Thus, for a given measurable set $A$, and any $\eps > 0$, there exists an open set $U_{\eps}$ such that 
\[\mu(U_{\eps} \setminus A) \leq \eps.\]
Therefore by finite additivity and Equation~\eqref{eq:mu-bound} 
\begin{align*}
    \lim_{n \to \infty} \frac{N_{\geq k}(U_{\eps}, n)}{n} - \eps \leq \liminf_{n \to \infty} \frac{N_{\geq k}(A, n)}{n} \leq \limsup_{n \to \infty} \frac{N_{\geq k}(A, n)}{n} \leq \lim_{n \to \infty} \frac{N_{\geq k}(U_{\eps}, n)}{n}. 
\end{align*}
The proof for the general case now follows by applying the result for the class $\mathcal{O}$, and sending $\eps \to 0$.
\end{proof}

Theorem~\ref{th:degree-dist} now allows us to prove Theorem~\ref{thm:conv-edge-dist}.

\subsubsection{Proof of Theorem~\ref{thm:conv-edge-dist}} \label{subsec:proof-of-thm-conv-edge-dist}
Note that, if $N_{k}(A,n)$ denotes the number of vertices of out-degree $k$ in the tree at time $n$ having weight in $A$, by counting the edges in the tree in two ways we have 
\[
\Xi(A,n) = \sum_{k=1}^{n} k N_{k}(A,n) = \sum_{k=1}^{n} N_{\geq k}(A,n).
\]
But now, Lemma~\ref{lem:conv-sum} and using Fatou's Lemma in the last inequality, we have,
\begin{align*}
(\psi_{*})\mu(A) = \E{\frac{h(W)}{\lambda^{*} - \tilde{g}(W)} \mathbf{1}_{A}}
&= \sum_{k=1}^{\infty} \E{\prod_{i=0}^{k-1}\left(\frac{S_i(W)}{S_i(W) + \lambda^{*}}\right)\mathbf{1}_{A}}
\\ & = \sum_{k=1}^{\infty} \liminf_{n \to \infty} \frac{N_{\geq k}(A, n)}{n} \leq \liminf_{n \to \infty} \frac{\Xi(A,n)}{n};  
\end{align*}
and likewise, $\liminf_{n \to \infty} \frac{\Xi(A^{c},n)}{n} \geq (\psi_{*}\mu)(A^{c})$. 
Now, since we add one edge at each time-step, it follows that $\Xi([0, \wmax], n) = n$. Thus, by finite additivity, 
\begin{align*}
1 = \liminf_{n \to \infty} \left(\frac{\Xi(A, n)}{n} + \frac{\Xi(A^{c},n)}{n}\right)
& \leq \limsup_{n \to \infty} \frac{\Xi(A, n)}{n} + \liminf_{n \to \infty} \frac{\Xi(A^{c},n)}{n} \\ & \leq \limsup_{n \to \infty} \left(\frac{\Xi(A,n)}{n} + \frac{\Xi(A^{c},n)}{n}\right) = 1.
\end{align*}
But, since Equation~\eqref{eq:cond-c1} implies that $(\psi_{*}\mu)(\cdot)$ is a probability measure, this is only possible if 
\begin{align} \label{eq:supbound}
\limsup_{n \to \infty} \frac{\Xi(A,n)}{n} = (\psi_{*}\mu)(A) \text{ and } \liminf_{n \to \infty} \frac{\Xi(A^c,n)}{n} = (\psi_{*}\mu)(A^c) \text{ almost surely.} 
\end{align}
The result follows.

\section{The Condensation Regime} \label{sec:condensation}
Here, we extend the results of the previous section to the condensation regime. The techniques applied in this section are closely related to those of~\cite{asympt-gen}.
\\\\
The results of this subsection will depend on a sequences of auxiliary trees $\cT^{(\eps)}, \cT^{(-\eps)} , \eps > 0$. Given $\eps > 0$, and  $\mathcal{M}_{\eps}$ as defined in Equation~\eqref{eq:dominating-set}, define the functions $g_{\eps}, g_{-\eps}$ such that
\[g_{\eps}(p,q) := \mathbf{1}_{\mathcal{M}^{c}_{\eps}}(p) g(p,q) + \mathbf{1}_{\mathcal{M}_{\eps}}(p)g(x^{*},q)\]
and
\[g_{-\eps}(p,q) := \mathbf{1}_{\mathcal{M}^{c}_{\eps}}(p) g(p,q) + \mathbf{1}_{\mathcal{M}_{\eps}}(p)(g(x^{*}, q) - u_\eps(q));\]
and let $\cT^{(\eps)}, \cT^{(-\eps)}$ be the evolving trees with measure $\mu$, and associated functions $g_{\eps},h$ and $g_{-\eps}, h$ respectively. We also denote by $(\cZ_{n}^{(\eps)})_{n \geq 0}$ and $(\cZ_{n}^{(-\eps)})_{n \geq 0}$ the partition functions associated with $\cT^{(\eps)}, \cT^{(-\eps)}$, respectively. 

\begin{lem} \label{lem:aux-trees-c1}
Assume Conditions~\hyperlink{d}{\textbf{D1-D4}}. Then, for each $\eps > 0$ sufficiently small, $\cT^{(\eps)}$ and $\cT^{(-\eps)}$ satisfy Conditions~\hyperlink{c1}{\textbf{C1}} and \hyperlink{c1}{\textbf{C2}}. In addition, if $\lambda_{\eps}, \lambda_{-\eps}$ denote the Malthusian parameters associated with $\cT^{(\eps)}, \cT^{(-\eps)}$, then $\lambda_{\eps} \downarrow \tilde{g}^{*}$ and $\lambda_{-\eps} \uparrow \tilde{g}^{*}$ as $\eps \downarrow 0$. 
\end{lem}

\begin{proof}
First, since by \hyperlink{d}{\textbf{D2}} $g$ satisfies Condition~\hyperlink{c1}{\textbf{C2}}, we have 
\[g(x,y) =  \kappa\left(\phi^{(1)}_{1}(x),\ldots,\phi^{(N)}_{1}(x), \phi^{(1)}_{2}(y), \ldots, \phi^{(N)}_{2}(y)\right),\]
for measurable functions $\phi^{i}_{j}: [0,\wmax] \rightarrow [0, \maxurntwo]$, $j = 1,2$, $i \in [N]$ and a bounded continuous function $\kappa:[0, \maxurntwo]^{2N} \rightarrow \mathbb{R}_{+}$. Now, if we set $\phi^{(N+1)}_{1}(x):= \mathbf{1}_{\mathcal{M}_{\eps}}(x), \phi^{(N+2)}_{1}(x):= \mathbf{1}_{\mathcal{M}^{c}_{\eps}}(x)$, $\phi^{(N+1)}_{2}(y):= g(x^{*},y) - u_{\eps}(y)$ and define $\kappa'$ such that
\[
\kappa'(c_1, \ldots, c_{N+2}, d_1, \ldots d_{N+1}) := c_{N+2}\kappa(c_1, \ldots, c_N, d_{1}, \ldots, d_N) + c_{N+1}d_{N+1},   
\]
we clearly have that $\phi^{(N+1)}_{1}, \phi^{(N+2)}_{1}, \phi^{(N+1)}_{2}$ are bounded, non-negative measurable functions, and $\kappa'$ is bounded and continuous, taking values in $\mathbb{R}_{+}$. Noting that 
\[
g_{-\eps}(x,y) =  \kappa'\left(\phi^{(1)}_{1}(x),\ldots,\phi^{(N+2)}_{1}(x), \phi^{(1)}_{2}(y), \ldots, \phi^{(N+1)}_{2}(y)\right), 
\]
it follows that $g_{-\eps}$ satisfies Condition~\hyperlink{c1}{\textbf{C2}}. The proof of \hyperlink{c1}{\textbf{C2}} for $g_{\eps}$ is similar.

For \hyperlink{c1}{\textbf{C1}}, since $h$ is bounded, for sufficiently  large $\lambda > \tilde{g}^{*}$, we have
\[\E{\frac{h(W)}{\lambda - \tilde{g}_{\eps}(W)}} < 1.\]
Meanwhile, since, by Condition~\hyperlink{d}{\textbf{D4}}, $\mu(\mathcal{M}_{\eps})>0$ and $\tilde{g}_{\eps}(x) = \tilde{g}^{*}$ for any $x \in \mathcal{M}_{\eps}$, by monotone convergence
\[\lim_{\lambda \downarrow \tilde{g}^{*}} \E{\frac{h(W)}{\lambda - \tilde{g}_{\eps}(W)}} = 
\E{\frac{h(W)}{\tilde{g}^{*} - \tilde{g}_{\eps}(W)}} = \infty.\]
Thus, by continuity in $\lambda$, Condition~\hyperlink{c1}{\textbf{C1}} is satisfied for $\cT^{(\eps)}$. A similar argument also works for $\cT^{(-\eps)}$: if $\tilde{g}^{*}_{-\eps}$ denotes the maximum value of $\tilde{g}_{-\eps}(x)$, then this value is also attained on $\mathcal{M}_{\eps}$ which has positive measure. 
If $\lambda_{\eps}, \lambda_{-\eps}$ denote the associated Malthusian parameters associated with the trees, then, for each $\eps > 0$, $\lambda_{\eps} > \tilde{g}^{*}$ and $\lambda_{-\eps} > \tilde{g}^{*}_{-\eps}$. Moreover, since $g_{\eps}$ is non-increasing pointwise as $\eps$ decreases, $\lambda_{\eps}$ is non-increasing in $\eps$; likewise, $\lambda_{-\eps}$ is non-decreasing in $\eps$. Now, suppose $\lim_{\eps \downarrow 0} \lambda_{\eps} = \lambda_{+} > \tilde{g}^{*}$. Then we may apply dominated convergence, and 
\[
1 = \lim_{\eps \downarrow 0} \E{\frac{h(W)}{\lambda_{\eps} - \tilde{g}_{\eps}(W)}} =  \E{\lim_{\eps \downarrow 0}\frac{h(W)}{\lambda_{\eps} - \tilde{g}_{\eps}(W)}} = \E{\frac{h(W)}{\lambda_{+} - \tilde{g}(W)}},
\]
contradicting Equation~\eqref{eq:cond-c2}. The case for $\lambda_{-\eps}$ follows identically. 
\end{proof}

\begin{lem} \label{lem:coupl-aux-trees}
There exists a coupling $(\hat{\cT}^{(-\eps)}, \hat{\cT}, \hat{\cT}^{(\eps)})$ of these processes such that, almost surely (on the coupling space), for all $n \in \N_0$, 
\begin{align} \label{eq:coupl2-partit}
\cZ^{(-\eps)}_{n} \leq \cZ_n \leq \cZ^{(\eps)}_n,
\end{align}
and, for each vertex $v$ with $W_{v} \in \mathcal{M}^{c}_{\eps}$, we have 
\begin{align} \label{eq:fitness-vert-ineq}
    f(N^{+}(v, \hat{\cT}^{(\eps)}_{n})) \leq f(N^{+}(v, \hat{\cT}_{n})) \leq f(N^{+}(v, \hat{\cT}^{(-\eps)}_{n}))
\end{align}
and 
\begin{align} \label{eq:degrees-ineq}
    \deg{(v, \hat{\cT}^{(\eps)}_{n})} \leq \deg{(v, \hat{\cT}_{n})} \leq \deg{(v, \hat{\cT}^{(-\eps)}_{n})}.
\end{align}
\end{lem}

\begin{proof}
We initialise the trees with a 
single vertex $0$ having weight $W_0$ sampled independently from $\mu$, conditioned on $\{h(W_0) > 0 \}$ and will construct copies of these three tree processes on the same vertex set, which is identified with $\mathbb{N}_0$. 
Now, assume that at the $n$th time-step, \[ (\hat{\cT}^{(-\eps)}_{j})_{0 \leq j \leq n} \sim (\hat{\cT}^{(-\eps)}_{j})_{0 \leq j \leq n}, \quad  (\hat{\cT}_{j})_{0 \leq j \leq n} \sim (\cT_{j})_{0 \leq j \leq n} \quad \text{ and } \quad  (\hat{\cT}^{(\eps)}_{j})_{0 \leq j \leq n} \sim (\cT^{(\eps)}_{j})_{0 \leq j \leq n}.
\]
In addition, assume that Equations~\eqref{eq:coupl2-partit} and \eqref{eq:fitness-vert-ineq} are satisfied up to time $n$. 
Now, for the $(n+1)$st step: 
\begin{itemize}
\item Introduce vertex $n+1$ with weight $W_{n+1}$ sampled independently from $\mu$ in $\hat{\cT}^{(-\eps)}_{n}, \hat{\cT}_{n}$ and $\hat{\cT}^{(\eps)}_{n}$.
\item Form $\hat{\cT}^{(-\eps)}_{n+1}$ by sampling the parent $v$ of $n+1$ independently according to the law of $\mathcal{T}^{(-\eps)}$ (i.e. with probability proportional to $f(N^{+}(v, \hat{\cT}^{(-\eps)}_n))$). Then, in order to form $\hat{\cT}_{n+1}$ sample an independent uniformly distributed random variables $U_1$ on $[0,1]$.
    \begin{itemize}
        \item If $U_1 \leq \frac{\cZ^{(-\eps)}_{n}f(N^{+}(v, \hat{\cT}_n)) }{\cZ_n f(N^{+}(v, \hat{\cT}^{(-\eps)}_n))}$ and $W_{v} \in \mathcal{M}^{c}_{\eps}$, select $v$ as the parent of $n+1$ in $\hat{\cT}_{n+1}$ as well.
        \item Otherwise, form $\hat{\cT}_{n+1}$ by selecting the parent $v'$ of $n+1$ with probability proportional to $f(N^{+}(v', \hat{\cT}_n))$ out of all all the vertices with weight $W_{v'} \in \mathcal{M}_{\eps}$.
    \end{itemize}
    \item Then form $\hat{\cT}^{(\eps)}_{n+1}$ in a similar manner. Sample an independent uniform random variable $U_2$ on $[0,1]$. 
    \begin{itemize}
        \item If vertex $v$ (with weight $W_{v} \in \mathcal{M}^{c}_{\eps}$) was chosen as the parent of $n+1$ in $\hat{\cT}_{n+1}$ and
        $U_2 \leq \frac{\cZ_{n}f(N^{+}(v, \hat{\cT}^{(\eps)}_n))}{\cZ^{(\eps)}_n f(N^{+}(v, \hat{\cT}_n))}$, also select $v$ as the parent of $n+1$ in $\hat{\cT}^{\eps}_{n+1}$.
        \item Otherwise, form $\hat{\cT}^{(\eps)}_{n+1}$ by selecting the parent $v''$ of $n+1$ with probability proportional to $f(N^{+}(v'', \cT^{(\eps)}_n))$ out of all the vertices with weight $W_{v''} \in \mathcal{M}_{\eps}$.
    \end{itemize}
\end{itemize}
Clearly $\hat{\cT}^{(-\eps)}_{n+1} \sim \cT^{(-\eps)}_{n+1}$. On the other hand, in $\hat{\cT}_{n+1}$ the probability of choosing a certain parent $v$
of $n+1$ with weight $W_v \in \mathcal{M}^c_{\eps}$ is 
\[\frac{\cZ^{(-\eps)}_{n}f(N^{+}(v, \hat{\cT}_n)) }{\cZ_n f(N^{+}(v, \hat{\cT}^{(-\eps)}_n))} \times \frac{f(N^{+}(v, \hat{\cT}^{(-\eps)}_n))}{\cZ^{(-\eps)}_n} = \frac{f(N^{+}(v, \hat{\cT}_n))}{\cZ_n},\]
whilst the probability of choosing a parent $v'$ with weight $W_{v'} \in \mathcal{M}_{\eps}$ is 
\begin{align*}
    & \frac{f(N^{+}(v', \hat{\cT}_n))}{\sum_{v' :W_{v'} \in \mathcal{M}_\eps} f(N^{+}(v', \hat{\cT}_n))} \left(\sum_{v: W_{v} \in \mathcal{M}^{c}_{\eps}}\left(1- \frac{\cZ^{(-\eps)}_{n} f(N^{+}(v, \hat{\cT}_n)) }{\cZ_n f(N^{+}(v, \hat{\cT}^{(-\eps)}_n))}\right) \frac{f(N^{+}(v, \hat{\cT}^{(-\eps)}_n))}{\cZ^{(-\eps)}_n}\right) \\ & \hspace{5.5cm}
    + \frac{f(N^{+}(v', \hat{\cT}_n))}{\sum_{v':W_{v'} \in \mathcal{M}_{\eps}} f(N^{+}(v', \hat{\cT}_n))} \left(\sum_{v:W_v \in \mathcal{M}_{\eps}} \frac{f(N^{+}(v, \hat{\cT}^{(-\eps)}_n))}{\cZ^{(-\eps)}_n}\right)
    \\
    & \hspace{3cm} = \frac{f(N^{+}(v', \hat{\cT}_n))}{\sum_{v' :W_{v'} \in \mathcal{M}_{\eps}} f(N^{+}(v', \hat{\cT}_n))}\left(\sum_{v} \frac{f(N^{+}(v, \hat{\cT}^{(-\eps)}_n)))}{\cZ^{(-\eps)}_n}
- \sum_{v: W_{v} \in \mathcal{M}^c_{\eps}}  \frac{f(N^{+}(v, \hat{\cT}_n))}{\cZ_n} \right)\\
& \hspace{3cm} =  \frac{f(N^{+}(v', \hat{\cT}_n))}{\sum_{v' :W_{v'} \in \mathcal{M}_{\eps}} f(N^{+}(v', \hat{\cT}_n))} \left(1 -  \frac{\sum_{v: W_{v} \in \mathcal{M}^c_{\eps} }f(N^{+}(v, \hat{\cT}_n))}{\cZ_n}\right) = \frac{f(N^{+}(v', \hat{\cT}_n))}{\cZ_n},
\end{align*}
where we use the fact that $\sum_{v} f(N^{+}(v, \hat{\cT}_n)) = \cZ_n$. Thus, we have $\hat{\cT}_{n+1} \sim \cT_{n+1}$. 
Now, note that if the parent $v$ of $n+1$ in $\hat{\cT}^{(-\eps)}_{n+1}$ is such that $W_v \in \mathcal{M}^{c}_{\eps}$, the same parent is chosen in $\hat{\cT}_{n+1}$. Since $W_{v} \in \mathcal{M}^{c}_{\eps}$, we have 
\[f(N^{+}(v, \hat{\cT}^{(-\eps)}_{n+1})) - f(N^{+}(v, \hat{\cT}^{(-\eps)}_{n})) = g_{-\eps}(W_{v}, W_{n+1}) = g(W_v, W_{n+1}) = f(N^{+}(v, \hat{\cT}_{n+1})) - f(N^{+}(v, \hat{\cT}_{n})).\] Otherwise, the parent of $n+1$ in $\hat{\cT}_{n+1}$ has weight which belongs to $\mathcal{M}_{\eps}$, and thus $f(N^{+}(v, \hat{\cT}^{(-\eps)}_{n}))$ increases whilst $f(N^{+}(v,\hat{\cT}_{n}))$ stays the same. An increase in $f(N^{+}(v, \hat{\cT}^{(-\eps)}_{n}))$ coincides with the increase of $\deg{(v, \hat{\cT}^{(-\eps)}_{n})}$, and thus the right hand sides of Equations~\eqref{eq:fitness-vert-ineq} and \eqref{eq:degrees-ineq} are satisfied for time $n+1$. 
\\
Now, note that
\[\cZ^{(-\eps)}_{n+1} - \cZ^{(-\eps)}_{n} = h(W_{n+1}) + g_{-\eps}(W_{v},W_{n+1}), \text{ and } \cZ_{n+1} - \cZ_{n} = h(W_{n+1}) + g(W_{v'},W_{n+1})\]
where $v, v'$ denote the parent of $n+1$ in $\hat{\cT}_{n}$ and $\hat{\cT}^{(\eps)}_{n}$ respectively. Then we either have: 
\begin{itemize}
\item $v = v'$ (so that $g_{-\eps}(W_{v},W_{n+1}) = g(W_{v'},W_{n+1})$),
\item $v \in \mathcal{M}^{c}_{\eps}$ and $v' \in \mathcal{M}_{\eps}$, in which case, $\mathbb{P}$-a.s, using \hyperlink{d}{\textbf{D4}}
\[g_{-\eps}(W_{v}, W_{n+1}) = g(W_v, W_{n+1}) \leq g(x^{*}, W_{n+1}) - u_\eps(W_{n+1}) < g(W_{v'}, W_{n+1}),\] 
\item Both $v, v' \in \mathcal{M}_{\eps}$, in which case, $\mathbb{P}$-a.s., \[g_{-\eps}(W_{v}, W_{n+1}) = g(x^{*}, W_{n+1}) - u_\eps(W_{n+1}) < g(W_{v'},W_{n+1}).\]
\end{itemize}
In every case we have $\cZ^{(-\eps)}_{n+1} - \cZ^{(-\eps)}_{n} \leq \cZ_{n+1} - \cZ_{n},$ 
and thus Equation~\eqref{eq:coupl2-partit} is also satisfied at time $n+1$.
\\
Each of the statements concerning $\hat{\cT}^{(\eps)}$ follow in an analogous manner, applying Condition~\hyperlink{d}{\textbf{D3}}.
\end{proof}
\subsection{Proof of Theorem~\ref{th:conv-parti-d1}}
\label{subsec:proof-of-conv-parti-d1}
\begin{proof}
First note that by Equation~\eqref{eq:coupl2-partit} in Lemma~\ref{lem:coupl-aux-trees} (and Theorem~\ref{th:conv-parti-c1}), for each $\eps > 0$ we have, $\mathbb{P}$-a.s., 
\[
\lambda_{-\eps} = \lim_{n \to \infty} \frac{\cZ^{(-\eps)}_{n}}{n} \leq \liminf_{n \to \infty} \frac{\cZ_{n}}{n} \leq \limsup_{n \to \infty} \frac{\cZ_{n}}{n} = \lim_{n \to \infty} \frac{\cZ^{(\eps)}_{n}}{n}
= \lambda_{\eps}.
\]
The result follows by sending $\eps \to 0$, using Lemma~\ref{lem:aux-trees-c1}.
\end{proof}

In the following theorem, recall the definition of the measure $\psi_{*}\mu$ in Equation~\eqref{eq:pushforward-def}.

\subsection{Proof of Theorem~\ref{thm:condensation}} \label{subsec:proof-of-condensation}
\begin{proof}
By assumption, for each $\eps > 0$ sufficiently small, we have $A \subseteq \mathcal{M}^{c}_{\eps}$. Next, applying Equation~\eqref{eq:degrees-ineq}, if $\Xi^{(\eps)}$ and $\Xi^{(-\eps)}$ denote the edge distributions in the coupled trees $\hat{\mathcal{T}}^{(\eps)}, \hat{\mathcal{T}}^{(-\eps)}$, respectively, then for each $n \in \mathbb{N}_{0}$ 
\[
\Xi^{(\eps)}(A,n) \leq \Xi(A,n) \leq \Xi^{(-\eps)}(A,n),
\]
and thus, by Theorem~\ref{thm:conv-edge-dist}, we have 
\begin{align} \label{eq:inequality-outside-condensate}
    \E{\frac{h(W)}{\lambda_{\eps} - \tilde{g}_{\eps}(W)}\mathbf{1}_{A}} & \leq \liminf_{n\to \infty} \frac{\Xi(A,n)}{n} \leq \limsup_{n\to \infty} \frac{\Xi(A,n)}{n}
    \leq \E{\frac{h(W)}{\lambda_{-\eps} - \tilde{g}_{-\eps}(W)}\mathbf{1}_{A}}.
\end{align}
Now, noting that $\tilde{g}_{-\eps} = \tilde{g} = \tilde{g}_{\eps}$ on $A$, and $\lambda_{-\eps} > \tilde{g}_{-\eps}^{*} \geq \sup_{x \in A} \tilde{g}(x)$ and is non-decreasing in $\eps$, by applying Lemma~\ref{lem:aux-trees-c1} and dominated convergence we have 
\begin{align} \label{eq:trunc-edge-limits}
\lim_{\eps \to 0} \E{\frac{h(W)}{\lambda_{\eps} - \tilde{g}_{\eps}(W)}\mathbf{1}_{A}} = \lim_{\eps \to 0} \E{\frac{h(W)}{\lambda_{-\eps} - \tilde{g}_{-\eps}(W)}\mathbf{1}_{A}} = \E{\frac{h(W)}{\tilde{g}^{*} - \tilde{g}(W)}\mathbf{1}_{A}}.
\end{align}
Equation~\eqref{eq:cond-portmanteau} follows by combining Equations~\eqref{eq:inequality-outside-condensate} and \eqref{eq:trunc-edge-limits}.
Moreover, for each $\eps' > 0$, by setting $A = \mathcal{M}^{c}_{\eps'}$,
\[
\lim_{n\to \infty} \frac{\Xi(\mathcal{M}_{\eps'},n)}{n}
= \lim_{n\to \infty}\left(1 -  \frac{\Xi(\mathcal{M}^{c}_{\eps'},n)}{n}\right)
= 1 - \E{\frac{h(W)}{\tilde{g}^{*} - \tilde{g}(W)}\mathbf{1}_{\mathcal{M}^{c}_{\eps'}}}.\]
But then, again by dominated convergence,
\[
\lim_{\eps' \to 0} \E{\frac{h(W)}{\tilde{g}^{*} - \tilde{g}(W)}\mathbf{1}_{\mathcal{M}^{c}_{\eps'}}} = \E{\frac{h(W)}{\tilde{g}^{*} - \tilde{g}(W)}},
\]
and Equation~\eqref{eq:condensate-mass} follows. 
\end{proof}

\subsection{Proof of Corollary~\ref{cor:cor-seven} \label{subsec:cor-seven}}
\begin{proof}
By the Portmanteau theorem, it suffices to show that, $\mathbb{P}$-a.s. 
\[
\lim_{n \to \infty}\frac{\Xi(A,n)}{n} = \Pi(A),
\]
for any set $A \in \mathscr{B}$ with $\mu\left(\partial A\right) = 0$. Now, since $\mu(\mathcal{M}) = 0$, it suffices to prove this equation for sets $A\in \mathscr{B}$ with $\overline{A} \cap \mathcal{M} = \varnothing$. In view of Theorem~\ref{thm:condensation}, we need only show that for all $\eps > 0$ sufficiently small, we have $\overline{A} \cap \mathcal{M}_{\eps} = \varnothing$. Indeed, if this were not the case, then, since $(\overline{A} \cap \overline{\mathcal{M}}_{1/n})_{n \in \mathbb{N}}$ is a nested sequence of closed sets, by Cantor's intersection theorem, 
\[
\varnothing \neq \bigcap_{n \in \mathbb{N}} \left(\overline{A} \cap \overline{\mathcal{M}}_{1/n}\right) = \overline{A} \cap \bigcap_{n \in \mathbb{N}}  \overline{\mathcal{M}}_{1/n} = \overline{A} \cap \mathcal{M}, 
\]
a contradiction.

\end{proof}

The coupling also allows us to derive a result for the degree distribution. Recall the definition of the companion process $(S_i)_{i\geq 0}$ in Equation~\eqref{eq:def-comp-process}, and that, for $B \in \mathscr{B}$, $N_{\geq k}(B,n)$ denotes the number of vertices of out-degree at least $k$ with weight belonging to $B$ at time $n$. 
\subsection{Proof of Theorem~\ref{thm-deg-dist-2} \label{subsec-proof-of-deg-dist-2}}
\begin{proof}
Let $B \in \mathscr{B}$ be given.  For  $\eps > 0$, note that 
\[
\frac{N_{\geq k}(B\cap \mathcal{M}^{c}_{\eps}, n)}n \leq \frac{N_{\geq k}(B, n)}n \leq \frac{N_{\geq k}(B\cap \mathcal{M}^{c}_{\eps}, n)}n + 
\frac{N_{\geq 0}(\mathcal{M}_{\eps})}n.
\]
Now, by the strong law of large numbers, in the limit as $n \to \infty$ (as in Equation~\eqref{eq:mu-bound}), the second quantity tends to $\mu(\mathcal{M}_{\eps})$, and thus,
\begin{align} \label{eq:deg-monoton}
\liminf_{n \to \infty} \frac{N_{\geq k}(B\cap \mathcal{M}^{c}_{\eps}, n)}n \leq \limsup_{n \to \infty} \frac{N_{\geq k}(B, n)}n \leq    \limsup_{n\to \infty} \frac{N_{\geq k}(B\cap \mathcal{M}^{c}_{\eps}, n)}n + \mu(\mathcal{M}_{\eps}).
\end{align}
Now, let $N^{(-\eps)}_{\geq k}(\cdot,n), N^{(\eps)}_{\geq k}(\cdot,n)$ denote the associated quantities in the trees $\cT^{(-\eps)}, \cT^{(\eps)}$, and denote by $(S^{(-\eps)}_{i})_{i \geq 0}$ and $(S^{(\eps)}_{i})_{i \geq 0}$ the companion processes defined in terms of the functions $h, g_{-\eps}$ and $h, g_{+\eps}$ respectively.
Then, by Equation~\eqref{eq:degrees-ineq}, on the coupling in Lemma~\ref{lem:coupl-aux-trees}, we have 
\[
N^{(\eps)}_{\geq k}(B\cap \mathcal{M}^{c}_{\eps}, n) \leq N_{\geq k}(B\cap \mathcal{M}^{c}_{\eps}, n) \leq N^{(-\eps)}_{\geq k}(B\cap \mathcal{M}^{c}_{\eps}, n).
\]
Therefore, by Theorem~\ref{th:degree-dist}, (recalling the definitions of $\lambda_{\eps}, \lambda_{-\eps}$ in Lemma~\ref{lem:aux-trees-c1})
\begin{align*}
\E{\prod_{i=0}^{k-1}\left(\frac{S^{(\eps)}_i(W)}{S^{(\eps)}_i(W) + \lambda_{\eps}}\right)\mathbf{1}_{B \cap \mathcal{M}^{c}_{\eps}}} & \leq \liminf_{n \to \infty} \frac{N_{\geq k}(B\cap \mathcal{M}^{c}_{\eps}, n)}n 
\\ \nonumber &\leq \limsup_{n \to \infty} \frac{N_{\geq k}(B\cap \mathcal{M}^{c}_{\eps}, n)}n
\leq \E{\prod_{i=0}^{k-1}\left(\frac{S^{(-\eps)}_i(W)}{S^{(-\eps)}_i(W) + \lambda_{-\eps}}\right)\mathbf{1}_{B\cap \mathcal{M}^{c}_{\eps}}},
\end{align*}
and thus, by Equation~\eqref{eq:deg-monoton}, we have 
\begin{align} \label{eq:eps-ineq}
\E{\prod_{i=0}^{k-1}\left(\frac{S^{(\eps)}_i(W)}{S^{(\eps)}_i(W) + \lambda_{\eps}}\right)\mathbf{1}_{B\cap \mathcal{M}^{c}_{\eps}}} & \leq \liminf_{n \to \infty} \frac{N_{\geq k}(B, n)}n \\ \nonumber &\leq \limsup_{n \to \infty} \frac{N_{\geq k}(B, n)}n \leq \E{\prod_{i=0}^{k-1}\left(\frac{S^{(-\eps)}_i(W)}{S^{(-\eps)}_i(W) + \lambda_{-\eps}}\right)\mathbf{1}_{B\cap \mathcal{M}^{c}_{\eps}}} + \mu(\mathcal{M}_{\eps}).
\end{align}
Now, by dominated convergence, as $\eps \to 0$
\begin{align*}
&\E{\prod_{i=0}^{k-1}\left(\frac{S^{(\eps)}_i(W)}{S^{(\eps)}_i(W) + \lambda_{\eps}}\right)\mathbf{1}_{B\cap \mathcal{M}_{\eps}}} \to \E{\prod_{i=0}^{k-1}\left(\frac{S_i(W)}{S_i(W) + \tilde{g}^{*}}\right)\mathbf{1}_{B}}, \text{ and } \\ &\hspace{4cm} \E{\prod_{i=0}^{k-1}\left(\frac{S^{(-\eps)}_i(W)}{S^{(-\eps)}_i(W) + \lambda_{-\eps}}\right)\mathbf{1}_{B\cap \mathcal{M}_{\eps}}} \to \E{\prod_{i=0}^{k-1}\left(\frac{S_i(W)}{S_i(W) + \tilde{g}^{*}}\right)\mathbf{1}_{B}}, 
\end{align*}
and, since $\mathcal{M}$ is a $\mu$-null set (by Equation~\eqref{eq:cond-c2}), $\mu(\mathcal{M}_{\eps}) \to 0$. Combining these statements with~\eqref{eq:eps-ineq} completes the proof.
\end{proof}

\section{Appendix}
\subsection{Proof of Lemma~\ref{lem:conv-sum}} \label{subsec:proof-of-lem:convsum}
In order to prove Lemma~\ref{lem:conv-sum} we first introduce an auxiliary, piecewise constant continuous time Markov process $(\mathcal{Y}_{w}(t), r_{w}(t))_{t \geq 0}$ taking values in $\mathbb{N} \times [0, \infty)$. Let $(W_i)_{i \geq 0}$ be independent $\mu$-distributed random variables,
and define $(S_{i}(w))_{i \geq 0}$ according to \eqref{eq:def-comp-process}, that is, \[S_0(w) := h(w);  \quad S_{i+1}(w) := S_{i}(w) + g(w, W_{i+1}), \; i \geq 0.\]
In addition, set $\tau_0 = 0$, and define $(\tau_{i})_{i\geq 1}$ recursively so that
\begin{align} \label{eq:exp-clock}
\tau_{i} - \tau_{i-1} \equiv \text{Exp}(r_{w}(\tau_{i-1}));  \end{align}
where $\text{Exp}(r_{w}(\tau_{i-1}))$ denotes an exponentially distributed random variable with parameter $r_{w}(\tau_{i-1})$. Then, we set
\[
\mathcal{Y}_{w}(t) := \sum_{n=1}^{\infty} \mathbf{1}_{[\tau_{n},\infty)}(t), \quad \text{ and } \quad r_{w}(t) := \sum_{n=0}^{\infty} S_{n}(w) \mathbf{1}_{[\tau_{n},\tau_{n+1})}(t).  
\]
Now, let $(\mathcal{F}_{t})_{t\geq 0}$ denote the filtration generated by the process $(\mathcal{Y}_{w}(t), r_{w}(t))_{t \geq 0}$.
\begin{clm} \label{clm-martingale}
The process $\mathcal{Y}_{w}(t) - \int_{0}^{t} r_{w}(s) \dd s$ is a martingale with respect to the filtration $(\mathcal{F}_{t})_{t\geq 0}$. 
\end{clm}
\begin{proof}
This follows from the fact that the difference between jump times is exponentially distributed (Equation~\eqref{eq:exp-clock}), and by applying, for example, [Theorem~1.33, \cite{jacodshiryaev}] (page 149). 
\end{proof}
In addition, 
\begin{clm} \label{clm-non-explosion}
For all $t \in [0, \infty)$, we have $\E{\mathcal{Y}_{w}(t)} < \infty$ almost surely. In particular, for each $t \in [0, \infty)$,
\begin{equation} \label{eq:expectation-identity}
\E{\mathcal{Y}_{w}(t)} = \int_{0}^{t} \E{r_{w}(s)} \dd s. 
\end{equation}
\end{clm}
\begin{proof}
 Let $\alpha$ be an independent exponentially distributed random variable with parameter $a > 0$, and set $\mathcal{Y}_{w}(\alpha) := \inf_{t \geq \alpha}(\mathcal{Y}_{w}(t))$. Then,
\begin{align} \label{eq:append-rec-eq1}
\E{\mathbf{1}_{\mathcal{Y}_{w}(\alpha) \geq k} |  S_{k-1}(w),  \mathbf{1}_{\left\{\mathcal{Y}_{w}(\alpha) \geq k - 1\right\}}} & = \E{\alpha \geq \tau_k |  S_{k-1}(w), \mathbf{1}_{\left\{\mathcal{Y}_{w}(\alpha) \geq k - 1\right\}}} \\ &= \Prob{\min{(\alpha - \tau_{k-1}, \tau_{k}- \tau_{k-1})} = \tau_k - \tau_{k-1} |  S_{k-1}(w)}\mathbf{1}_{\left\{\mathcal{Y}_{w}(\alpha) \geq k - 1\right\}} \\ &= \frac{S_{k-1}(w)}{a+S_{k-1}(w)} \mathbf{1}_{\left\{\mathcal{Y}_{w}(\alpha) \geq k - 1\right\}},
\end{align}
where in the last equality we have used \eqref{eq:exp-clock} and the memory-less property of the exponential distribution. Note also, that for any $j \leq k-1$, the random variables $(S_{j}(w), \ldots, S_{k-1}(w))$ and $\mathbf{1}_{\left\{\mathcal{Y}_{w}(\alpha) \geq j\right\}}$ are conditionally independent given the random variables $S_{j-1}(w), \mathbf{1}_{\left\{\mathcal{Y}_{w}(\alpha) \geq j-1\right\}}$. Indeed, for each $\ell \in \left\{j, \ldots, k-1\right\}$,
\[
S_{\ell}(w) = S_{j-1}(w) + \sum_{i=j}^{\ell} g(w,W_{i}),
\]
where $W_{j}, \ldots, W_{k-1}$ are independent random variables sampled from $\mu$, while
\[
\mathbf{1}_{\left\{\mathcal{Y}_{w}(\alpha) \geq j\right\}} = \mathbf{1}_{\left\{\mathcal{Y}_{w}(\alpha) \geq j-1\right\}} \times \mathbf{1}_{\left\{ \min\left(\mathcal{S}_{j-1}, \alpha \right) = \mathcal{S}_{j-1}\right\}}
\]
where $\mathcal{S}_{j-1}$ is an independent exponentially distributed random variable with parameter $S_{j-1}(w)$. As a result, we have

\begin{align} \label{eq:append-rec-eq2}
   & \E{\left(\prod_{i=j}^{k-1} \frac{S_{i}(w)}{S_{i}(w) + a}\right) \mathbf{1}_{\left\{\mathcal{Y}_{w}(\alpha) \geq j\right\}} \bigg |S_{j-1}(w), \mathbf{1}_{\left\{\mathcal{Y}_{w}(\alpha) \geq j-1\right\}}} \\ & =     \E{\left(\prod_{i=j}^{k-1} \frac{S_{i}(w)}{S_{i}(w) + a}\right) \bigg |S_{j-1}(w), \mathbf{1}_{\left\{\mathcal{Y}_{w}(\alpha) \geq j-1\right\}}} \E{\mathbf{1}_{\left\{\mathcal{Y}_{w}(\alpha) \geq j\right\}} \bigg |S_{j-1}(w), \mathbf{1}_{\left\{\mathcal{Y}_{w}(\alpha) \geq j-1\right\}}}.
\end{align}

Therefore, we have 
\begin{align*}
\Prob{\mathcal{Y}_{w}(\alpha) \geq k} &= 
\E{\mathbf{1}_{\left\{\mathcal{Y}_{w}(\alpha)\geq k\right\}}} 
= \E{\E{\mathbf{1}_{\left\{\mathcal{Y}_{w}(\alpha)\geq k\right\}}| S_{k-1}(w), \mathbf{1}_{\left\{\mathcal{Y}_{w}(\alpha)\geq k-1\right\}}}}
\\& \stackrel{\eqref{eq:append-rec-eq1}}{=} \E{\frac{S_{k-1}(w)}{a+S_{k-1}(w)} \mathbf{1}_{\left\{\mathcal{Y}_{w}(\alpha)\geq k-1\right\}}}
\\&= \E{\E{\frac{S_{k-1}(w)}{a+S_{k-1}(w)}\mathbf{1}_{\left\{\mathcal{Y}_{w}(\alpha)\geq k-1\right\}} \bigg | S_{k-2}(w), \mathbf{1}_{\left\{\mathcal{Y}_{w}(\alpha)\geq k-2\right\}}}}
\\& \stackrel{\eqref{eq:append-rec-eq2}}{=} 
\E{\E{\frac{S_{k-1}(w)}{a+S_{k-1}(w)}\bigg|S_{k-2}(w), \mathbf{1}_{\left\{\mathcal{Y}_{w}(\alpha)\geq k-2\right\}}} \E{\mathbf{1}_{\left\{\mathcal{Y}_{w}(\alpha)\geq k-1\right\}}|S_{k-2}(w), \mathbf{1}_{\left\{\mathcal{Y}_{w}(\alpha)\geq k-2\right\}}}}
\\& \stackrel{\eqref{eq:append-rec-eq1}}{=} \E{\E{\frac{S_{k-1}(w)}{a+S_{k-1}(w)}\times \frac{S_{k-2}(w)}{a+S_{k-2}(w)} \mathbf{1}_{\left\{\mathcal{Y}_{w}(\alpha)\geq k-2\right\}} \bigg|S_{k-2}(w), \mathbf{1}_{\left\{\mathcal{Y}_{w}(\alpha)\geq k-2\right\}}}}
\\& = \E{\frac{S_{k-1}(w)}{a+S_{k-1}(w)} \times \frac{S_{k-2}(w)}{a+S_{k-2}(w)} \mathbf{1}_{\left\{\mathcal{Y}_{w}(\alpha)\geq k-2\right\}}}.
\end{align*}
Iterating in this manner and noting that $\mathcal{Y}_{w}(\alpha) \geq 0$ almost surely, we deduce that the previous expression is $\E{\prod_{i=0}^{k-1} \frac{S_{i}(w)}{a+S_{i}(w) }}$.
This now implies that
\begin{equation} \label{eq:exp-one-way}
\E{\mathcal{Y}_{w}(\alpha)} = \sum_{k=1}^{\infty} \E{\prod_{i=0}^{k-1} \frac{S_{i}(w)}{a+S_{i}(w)}}.
\end{equation}
Now, the display on the right is increasing in $S_{i}(w)$, and using the fact that $g$ and $h$ are bounded by $\maxurn$, we may bound this above by 
\[\sum_{k=1}^{\infty} \prod_{i=1}^{k} \frac{\maxurn i}{\maxurn i + a} < \infty 
\quad \text{ for all $a > \maxurn$ (by applying, for example, Stirling's approximation).}\]
Thus, for a suitable choice of $a$, 
$\E{\mathcal{Y}_{w}(\alpha)}$ is finite, so that, in particular, for each $t \in [0, \infty)$, since the random variable $\mathcal{Y}_{w}(t)$ is independent of the event $\{\alpha \geq t\}$ which occurs with positive probability, 
\[\E{\mathcal{Y}_{w}(t)} \leq  \frac{\E{Y_{w}(\alpha)\mathbf{1}_{\left\{\alpha \geq t\right\}}}}{\Prob{\alpha \geq t}} < \infty.\] Now \eqref{eq:expectation-identity} follows from Claim~\ref{clm-martingale}.
\end{proof}
We require an additional claim:
\begin{clm}
We have 
\begin{equation} \label{eq:linear-rate}
\E{r_{w}(t)} = h(w) + \E{g(w,W)} \E{\mathcal{Y}_{w}(t)} = h(w) + \tilde{g}(w) \E{\mathcal{Y}_{w}(t)}.
\end{equation}
\end{clm}
\begin{proof}
First note that, since $r_{w}(t)$ jumps by $g(w,W)$ whenever $\mathcal{Y}_{w}(t)$ jumps, 
we have \[\E{r_{w}(t)} - h(w) = \E{\sum_{i=1}^{\mathcal{Y}_{w}(t)} g(w,W_i)}.\]
Assume that $g(w, W_{i})$ are bounded by $\maxurn$. In addition, for each $n \in \mathbb{N}$,
\begin{align*}
    \E{g(w,W_{n}) \mathbf{1}_{\left\{\mathcal{Y}_{w}(t) \geq n\right\}}} & =
    \E{g(w,W_{n})} - \E{g(w,W_{n}) \mathbf{1}_{\left\{\mathcal{Y}_{w}(t) < n\right\}}} \\ &= \E{g(w,W_{n})}\left(1 - \Prob{\mathcal{Y}_{w}(t) < n}\right) = \E{g(w,W_n)}\Prob{\mathcal{Y}_{w}(t) \geq n},
\end{align*}
where the second to last equality follows from the fact that the event $\left\{\mathcal{Y}_{w}(t) < n\right\}$ depends only on $(S_{i}(w))_{i=0, \ldots, n-1}$, and is thus independent of $W_{n}$. Finally, by Claim~\ref{clm-non-explosion}, $\E{Y_{w}(t)} < \infty$, and thus the result follows by applying Wald's Lemma.
\end{proof}
\begin{proof}[Proof of Lemma~\ref{lem:conv-sum}]
First note that by Equations~\eqref{eq:expectation-identity} and \eqref{eq:linear-rate}, we have 
\[
\frac{\dd}{\dd t} \E{\mathcal{Y}_{w}(t)} = \tilde{g}(w) \E{\mathcal{Y}_{w}(t)}, 
\]
and solving this differential equation, with initial condition $\E{\mathcal{Y}_{w}(0)} = h(w)$, we have 
\begin{equation} \label{eq:expectation-cal-y}
\E{\mathcal{Y}_{w}(t)} = h(w)e^{\tilde{g}(w) t}. 
\end{equation}
Now, let $\Lambda$ be an exponentially distributed random variable with parameter $\lambda$. Then, on the one hand, by Equation~\eqref{eq:exp-one-way}
\[\E{\mathcal{Y}_{w}(\Lambda)} = \sum_{k=1}^{\infty} \E{\prod_{i=0}^{k-1} \frac{S_{i}(w)}{S_{i}(w) + \lambda}}.\]
On the other hand,
\begin{align*}
    \E{\mathcal{Y}_{w}(\Lambda)} = \E{\E{\mathcal{Y}_{w}(u) | \Lambda = u}} = \int_{0}^{t} \lambda e^{-\lambda u} \E{\mathcal{Y}_{w}(u)} \dd u
    \stackrel{\eqref{eq:expectation-cal-y}}{=} \int_{0}^{t} \lambda h(w) e^{- (\lambda - \tilde{g}(w))u}\dd u
    = \frac{h(w)}{\lambda - \tilde{g}(w)}
\end{align*}
where, in the last equality we have used the fact that $\lambda > \tilde{g}_{+}$. The result follows.
\end{proof}

\bibliographystyle{unsrt}
\bibliography{ref} 

\begin{thebibliography}{10}

\bibitem{Barabasi509}
Albert-L{\'a}szl{\'o} Barab{\'a}si and R{\'e}ka Albert.
\newblock Emergence of scaling in random networks.
\newblock {\em Science}, 286(5439):509--512, 1999.

\bibitem{bollobaspreferential}
B\'ela Bollob\'as, Oliver Riordan, Joel Spencer, and G\'abor Tusn\'ady.
\newblock The degree sequence of a scale-free random graph process.
\newblock {\em Random Structures Algorithms}, 18(3):279--290, 2001.

\bibitem{mori-trees-2002}
T.~F. M\'{o}ri.
\newblock On random trees.
\newblock {\em Studia Sci. Math. Hungar.}, 39(1-2):143--155, 2002.

\bibitem{prodingerurbanek}
Helmut Prodinger and Friedrich~J. Urbanek.
\newblock On monotone functions of tree structures.
\newblock {\em Discrete Appl. Math.}, 5(2):223--239, 1983.

\bibitem{szymanski}
Jerzy Szyma\'{n}ski.
\newblock On a nonuniform random recursive tree.
\newblock In {\em Random graphs '85 ({P}ozna\'{n}, 1985)}, volume 144 of {\em
  North-Holland Math. Stud.}, pages 297--306. North-Holland, Amsterdam, 1987.

\bibitem{mahmoud92}
Hosam~M. Mahmoud.
\newblock Distances in random plane-oriented recursive trees.
\newblock volume~41, pages 237--245. 1992.
\newblock Asymptotic methods in analysis and combinatorics.

\bibitem{mahmoudetal93}
Hosam~M. Mahmoud, R.~T. Smythe, and Jerzy Szyma\'{n}ski.
\newblock On the structure of random plane-oriented recursive trees and their
  branches.
\newblock {\em Random Structures Algorithms}, 4(2):151--176, 1993.

\bibitem{chen1994}
Wen-Chin Chen and Wen-Chun Ni.
\newblock Internal path length of the binary representation of heap-ordered
  trees.
\newblock {\em Inform. Process. Lett.}, 51(3):129 -- 132, 1994.

\bibitem{szaboalava2002}
G\'{a}bor Szab\'{o}, Mikko Alava, and J\'{a}nos Kert\'{e}sz.
\newblock Shortest paths and load scaling in scale-free trees.
\newblock {\em Phys. Rev. E (3)}, 66:026101, 09 2002.

\bibitem{bollobasriordan04}
B\'ela Bollob\'as and Oliver Riordan.
\newblock Shortest paths and load scaling in scale-free trees.
\newblock {\em Phys. Rev. E (3)}, 69:036114, 03 2004.

\bibitem{Oliveira-spencer}
Roberto Oliveira and Joel Spencer.
\newblock Connectivity transitions in networks with super-linear preferential
  attachment.
\newblock {\em Internet Math.}, 2(2):121--163, 2005.

\bibitem{rudas}
Anna Rudas, B\'{a}lint T\'{o}th, and Benedek Valk\'{o}.
\newblock Random trees and general branching processes.
\newblock {\em Random Structures Algorithms}, 31(2):186--202, 2007.

\bibitem{holmgren-janson}
Cecilia Holmgren and Svante Janson.
\newblock Fringe trees, {C}rump-{M}ode-{J}agers branching processes and
  {$m$}-ary search trees.
\newblock {\em Probab. Surv.}, 14:53--154, 2017.

\bibitem{dereich-morters-sublinear}
Steffen Dereich and Peter M\"{o}rters.
\newblock Random networks with sublinear preferential attachment: degree
  evolutions.
\newblock {\em Electron. J. Probab.}, 14:no. 43, 1222--1267, 2009.

\bibitem{bianconibarabasi2001}
Ginestra Bianconi and Albert-L{\'a}szl{\'o} Barab{\'a}si.
\newblock Bose-{E}instein condensation in complex networks.
\newblock {\em Phys. Rev. Lett.}, 86 24:5632--5, 2001.

\bibitem{borgs-chayes}
Christian Borgs, Jennifer Chayes, Constantinos Daskalakis, and Sebastien Roch.
\newblock First to market is not everything: an analysis of preferential
  attachment with fitness.
\newblock In {\em S{TOC}'07---{P}roceedings of the 39th {A}nnual {ACM}
  {S}ymposium on {T}heory of {C}omputing}, pages 135--144. ACM, New York, 2007.

\bibitem{dereich-unfolding}
Steffen Dereich.
\newblock Preferential attachment with fitness: unfolding the condensate.
\newblock {\em Electron. J. Probab.}, 21:Paper No. 3, 38, 2016.

\bibitem{dereich-ortgiese}
Steffen Dereich and Marcel Ortgiese.
\newblock Robust analysis of preferential attachment models with fitness.
\newblock {\em Combin. Probab. Comput.}, 23(3):386--411, 2014.

\bibitem{ergun}
G.~Erg\"{u}n and G.J. Rodgers.
\newblock Growing random networks with fitness.
\newblock {\em Phys. A.}, 303(1):261 -- 272, 2002.

\bibitem{wrt1}
K.~A. Borovkov and V.~A. Vatutin.
\newblock On the asymptotic behaviour of random recursive trees in random
  environments.
\newblock {\em Adv. in Appl. Probab.}, 38(4):1047--1070, 2006.

\bibitem{wrt2delphin}
Delphin {S{\'e}nizergues}.
\newblock {Geometry of weighted recursive and affine preferential attachment
  trees}.
\newblock arXiv preprint arXiv:1904.07115, 2019.

\bibitem{bas}
Bas Lodewijks and Marcel Ortgiese.
\newblock A phase transition for preferential attachment models with additive
  fitness.
\newblock arXiv preprint arXiv:2002.12863, 2020.

\bibitem{bas2}
Bas Lodewijks and Marcel Ortgiese.
\newblock The maximal degree in random recursive graphs with random weights.
\newblock arxiv preprint arxiv:2007.05438, 2020.

\bibitem{asympt-gen}
Tejas Iyer.
\newblock Degree distributions in recursive trees with fitnesses.
\newblock arxiv preprint arxiv:2005.02197, 2020.

\bibitem{Yule}
G.~Udny Yule.
\newblock A mathematical theory of evolution, based on the conclusions of {Dr.
  J. C. Willis, F.R.S.}
\newblock {\em Philosophical Transactions of the Royal Society of London.
  Series B, Containing Papers of a Biological Character}, 213:21--87, 1925.

\bibitem{Simon}
Herbert~A. Simon.
\newblock On a class of skew distribution functions.
\newblock {\em Biometrika}, 42:425--440, 1955.

\bibitem{kingman-1978}
J.~F.~C. Kingman.
\newblock A simple model for the balance between selection and mutation.
\newblock {\em J. Appl. Probability}, 15(1):1--12, 1978.

\bibitem{athreya-arka-sethuraman-2008}
Krishna~B. Athreya, Arka~P. Ghosh, and Sunder Sethuraman.
\newblock Growth of preferential attachment random graphs via continuous-time
  branching processes.
\newblock {\em Proc. Indian Acad. Sci. Math. Sci.}, 118(3):473--494, 2008.

\bibitem{bhamidi}
Shankar Bhamidi.
\newblock Universal techniques to analyze preferential attachment trees: global
  and local analysis, 2007.
\newblock Preprint available at
  \url{https://pdfs.semanticscholar.org/e7fb/8c999ff62a5f080e4c329a7a450f41fb1528.pdf}.

\bibitem{at_embedding}
Krishna~B. Athreya and Samuel Karlin.
\newblock Embedding of urn schemes into continuous time {M}arkov branching
  processes and related limit theorems.
\newblock {\em Ann. Math. Statist.}, 39:1801--1817, 1968.

\bibitem{nerman_81}
Olle Nerman.
\newblock On the convergence of supercritical general ({C}-{M}-{J}) branching
  processes.
\newblock {\em Z. Wahrsch. Verw. Gebiete}, 57(3):365--395, 1981.

\bibitem{kingman1975}
J.~F.~C. Kingman.
\newblock The first birth problem for an age-dependent branching process.
\newblock {\em Ann. Probab.}, 3(5):790--801, 10 1975.

\bibitem{pittel}
B.~Pittel.
\newblock Note on the heights of random recursive trees and random {$m$}-ary
  search trees.
\newblock {\em Random Structures Algorithms}, 5(2):337--347, 1994.

\bibitem{dereich-mailler-morters}
Steffen Dereich, C\'{e}cile Mailler, and Peter M\"{o}rters.
\newblock Nonextensive condensation in reinforced branching processes.
\newblock {\em Ann. Appl. Probab.}, 27(4):2539--2568, 2017.

\bibitem{vdh-aging-mult-fitness-2017}
Alessandro Garavaglia, Remco van~der Hofstad, and Gerhard Woeginger.
\newblock The dynamics of power laws: Fitness and aging in preferential
  attachment trees.
\newblock {\em J. Stat. Phys.}, 168:1137--1179, 2017.

\bibitem{jonathan-co-existing}
Jonathan Jordan.
\newblock Preferential attachment graphs with co-existing types of different
  fitnesses.
\newblock {\em J. Appl. Probab.}, 55(4):1211--1227, 2018.

\bibitem{jonathan-geometric-preferential-attachment}
Jonathan Jordan and Andrew~R. Wade.
\newblock Phase transitions for random geometric preferential attachment
  graphs.
\newblock {\em Adv. in Appl. Probab.}, 47(2):565--588, 2015.

\bibitem{jonathan-freeman-extensive-cond}
Nic Freeman and Jonathan Jordan.
\newblock Extensive condensation in a model of preferential attachment with
  fitness.
\newblock {\em Electron. J. Probab.}, 25:Paper No. 68, 42 pp., 2020.

\bibitem{jonathan-haslegrave-location-based}
John Haslegrave, Jonathan Jordan, and Mark Yarrow.
\newblock Condensation in preferential attachment models with location-based
  choice.
\newblock {\em Random Structures Algorithms}, 56(3):775--795, 2020.

\bibitem{dynamical-simplices}
Nikolaos Fountoulakis, Tejas Iyer, C\'{e}cile Mailler, and Henning Sulzbach.
\newblock Dynamical models for random simplicial complexes.
\newblock arXiv preprint arXiv:1910.12715, 2019.

\bibitem{bogachev}
V.~I. Bogachev.
\newblock {\em Measure theory. {V}ol. {I}, {II}}.
\newblock Springer-Verlag, Berlin, 2007.

\bibitem{janson_urns}
Svante Janson.
\newblock Functional limit theorems for multitype branching processes and
  generalized {P}\'{o}lya urns.
\newblock {\em Stochastic Process. Appl.}, 110(2):177--245, 2004.

\bibitem{split_times}
Krishna~B. Athreya and Samuel Karlin.
\newblock Limit theorems for the split times of branching processes.
\newblock {\em J. Math. Mech.}, 17:257--277, 1967.

\bibitem{maivil}
C{\'e}cile Mailler and Denis Villemonais.
\newblock Stochastic approximation on non-compact measure spaces and
  application to measure-valued {P}\'olya processes.
\newblock {\em Ann.\ Appl.\ Probab.}, 20(5):2393--2438, 2020.

\bibitem{jacodshiryaev}
Jean Jacod and Albert~N. Shiryaev.
\newblock {\em Limit theorems for stochastic processes}, volume 288 of {\em
  Grundlehren der Mathematischen Wissenschaften [Fundamental Principles of
  Mathematical Sciences]}.
\newblock Springer-Verlag, Berlin, second edition, 2003.

\end{thebibliography}
\end{document}